\documentclass{amsart}
\usepackage{enumerate}
\usepackage[initials]{amsrefs}
\usepackage{mathrsfs}
\usepackage{comment,euscript}
\usepackage{amssymb}

\theoremstyle{plain}
\newtheorem{thm}{Theorem}[section]
\newtheorem{lem}[thm]{Lemma}
\newtheorem{cor}[thm]{Corollary}
\newtheorem{prop}[thm]{Proposition}
\newtheorem{ques}[thm]{Question}
\newtheorem{subthm}{Theorem}[subsection]
\newtheorem{sublem}[subthm]{Lemma}
\newtheorem{subcor}[subthm]{Corollary}
\newtheorem{subconj}[subthm]{Conjecture}

\theoremstyle{remark}
\newtheorem{rem}[thm]{Remark}
\newtheorem{subrem}[subthm]{Remark}

\theoremstyle{definition}
\newtheorem{defi}[thm]{Definition}
\newtheorem{hypoth}[thm]{Hypothesis}
\newtheorem{subdefi}[subthm]{Definition}
\newtheorem{subnota}[subthm]{Notation}

\usepackage{color}

\newcount\theTime
\newcount\theHour
\newcount\theMinute
\newcount\theMinuteTens
\newcount\theScratch
\theTime=\number\time
\theHour=\theTime
\divide\theHour by 60
\theScratch=\theHour
\multiply\theScratch by 60
\theMinute=\theTime
\advance\theMinute by -\theScratch
\theMinuteTens=\theMinute
\divide\theMinuteTens by 10
\theScratch=\theMinuteTens
\multiply\theScratch by 10
\advance\theMinute by -\theScratch

\def\today{{\number\day\space
 \ifcase\month\or
  January\or February\or March\or April\or May\or June\or
  July\or August\or September\or October\or November\or December\fi
 \space\number\year}}

\newcommand\Bc{{\mathcal{B}}}
\newcommand\Bfr{{\mathfrak B}}
\newcommand\Cc{{\mathcal{C}}}
\newcommand\Compact{{\operatorname{Compact}}}
\newcommand\Cpx{{\mathbb C}}
\newcommand\Dc{{\mathcal{D}}}

\newcommand\eps{\varepsilon}
\newcommand\Et{{\tilde E}}
\newcommand\Exp{{\operatorname{Exp}}}
\newcommand\FauxExp{{\operatorname{FauxExp}}}
\newcommand\Fsc{{\mathscr{F}}}
\newcommand\Gfr{{\mathfrak G}}
\newcommand\HEu{{\EuScript H}}                   

\newcommand\Kc{{\mathcal{K}}}
\newcommand\Lc{{\mathcal{L}}}
\newcommand\Mcal{{\mathcal{M}}}
\newcommand\Nats{{\mathbb N}}
\newcommand\Nc{{\mathcal{N}}}
\newcommand\phibar{{\overline\phi}}
\newcommand\pit{{\tilde\pi}}
\newcommand\Reals{{\mathbb R}}
\newcommand\restrict{{\upharpoonright}}
\newcommand\Sc{{\mathcal{S}}}
\newcommand\sigmat{{\tilde\sigma}}

\newcommand\Tt{{\tilde T}}

\newcommand\Vc{{\mathcal{V}}}
\newcommand\Xc{{\mathcal{X}}}
\newcommand\Yc{{\mathcal{Y}}}
\newcommand\Yt{{\tilde Y}}
\newcommand\Zt{{\tilde Z}}

\begin{document}

\title[A decomposition theorem for unbounded operators]{An upper triangular decomposition theorem for some unbounded operators affiliated to II$_1$-factors}

\author[Dykema]{K. Dykema$^*$}
\address{Ken Dykema, Department of Mathematics, Texas A\&M University, College Station, TX, USA.}
\email{ken.dykema@math.tamu.edu}
\thanks{\footnotesize ${}^{*}$ Research supported in part by NSF grant DMS--1202660.}
\author[Sukochev]{F. Sukochev$^{\S}$}
\address{Fedor Sukochev, School of Mathematics and Statistics, University of New South Wales, Kensington, NSW, Australia.}
\email{f.sukochev@math.unsw.edu.au}
\thanks{\footnotesize ${}^{\S}$ Research supported by ARC}
\author[Zanin]{D. Zanin$^{\S}$}
\address{Dmitriy Zanin, School of Mathematics and Statistics, University of New South Wales, Kensington, NSW, Australia.}
\email{d.zanin@math.unsw.edu.au}



\begin{abstract}
Results of Haagerup and Schultz \cite{HS2} about existence of invariant subspaces that decompose the Brown measure
are extended to a large class of unbounded operators affiliated to a tracial von Neumann algebra.
These subspaces are used to decompose an arbitrary operator in this class
into the sum of a normal operator and a spectrally negligible operator.
This latter result is used to prove that, on a bimodule over a tracial von Neumann algebra that is closed
with respect to logarithmic submajorization, every trace is spectral, in the sense that the trace value on an operator
depends only on the Brown measure of the operator.
\end{abstract}

\date{September 10, 2015}

\maketitle

\begin{center}
Dedicated to the memory of Uffe Haagerup
\end{center}
\smallskip

\section{Introduction}

Consider a tracial von Neumann algebra, which is a pair $(\Mcal,\tau)$, where $\Mcal$ is a von Neumann algebra and $\tau$ is
a normal, faithful, tracial state on $\Mcal$.
For $T\in\Mcal$, there is a sort of spectral distribution measure, invented by L.\ Brown~\cite{Brown}
and called the Brown measure $\nu_T$.
(Note:  Brown also considered certain unbounded operators affiliated to semifinite von Neumann algebras.)
In a remarkable article~\cite{HS2}, Uffe Haagerup and Hanne Schultz proved
that, 
for arbitrary $T\in\Mcal$ and arbitrary Borel set $\Bc\in\Cpx$,
there is a unique $T$-invariant projection $p\in\Mcal$ so
that the restriction of $T$ to $p$ and the compression of $T$ to $1-p$
split the Brown measure of $T$ along $\Bc$ and its complement $\Bc^c$.
Namely, they show that $p$ is what we call a Haagerup-Schultz projection for $T$ and $\Bc$
(see Definition~\ref{def:hs proj} below.)
They also show that this projection $p$ is unique and that it is $T$-hyperinvariant, i.e., that it is invariant under every bounded
Hilbert space operator that commutes with $T$.

Consider $\Mcal$ embedded in the $*$-algebra $\Sc(\Mcal,\tau)$, of possibly unbounded operators affiliated with $\Mcal$.
The Brown measure and the Fuglede--Kadison determinant~\cite{FK}, on which it depends, are most naturally defined
on a larger algebra, $\Lc_{\log}(\Mcal,\tau)$, of all elements $T\in\Sc(\Mcal,\tau)$ satisfying $\tau(\log(1+|T|))<\infty$;
Haagerup and Schultz, in~\cite{HS1}, have developed these objects in this case.
Thus, it is natural to ask whether Haagerup and Schultz's results, about invariant subspaces that
split the Brown measure, can be extended to such $T$.

Our first main result is such an existence proof,
after passing to a larger von Neumann algebra.
(By this, we mean a tracial von Neumann algebra $(\Mcal_1,\tau_1)$
with $\Mcal$ included as a unital von Neumann subalgebra of $\Mcal_1$ so that the restriction of $\tau_1$ to $\Mcal$ is $\tau$.)

\begin{thm}\label{construction theorem}
Let $(\Mcal,\tau)$ be tracial von Neumann algebra.
Then there exists a larger tracial von Neumann algebra $(\Mcal_1,\tau_1)$
such that, for every $T\in \Lc_{\log}(\Mcal,\tau)$ and for every Borel set $\mathcal{B}$,
there exists
a projection $P(T,\mathcal{B})\in\Mcal_1$ such that
\begin{enumerate}[{\rm (a)}]
\item\label{csta} $TP(T,\mathcal{B})=P(T,\mathcal{B})TP(T,\mathcal{B}).$
\item\label{cstb} $\nu_T(\mathcal{B})=\tau_1(P(T,\mathcal{B})).$
\item\label{cstc} If $0<\nu_T(\Bc)<1$, then
the Brown measures of $TP(T,\mathcal{B})$ and $(1-P(T,\mathcal{B}))T$ are given by the formulas
$$\nu_{TP(T,\mathcal{B})}=\frac1{\nu_T(\mathcal{B})}\nu_T|_{\mathcal{B}},
\qquad \nu_{(1-P(T,\mathcal{B}))T}=\frac1{\nu(\mathbb{C}\backslash\mathcal{B})}\nu_T|_{\mathbb{C}\backslash\mathcal{B}}.$$
\item If $\mathcal{B}_1\subset\mathcal{B}_2$, then $P(T,\mathcal{B}_1)\leq P(T,\mathcal{B}_2)$.
\end{enumerate}
In particular, $(\Mcal_1,\tau_1)$ can be taken to be the ultrapower $((\Mcal*L(\mathbb{F}_4))_\omega,\tau_\omega)$
of the free product of $\Mcal$ with the free group factor $L(\mathbb{F}_4)$, for any free ultrafilter $\omega$ on $\Nats$.
\end{thm}

Our next main result is an upper triangular decomposition result
for $T\in\Lc_{\log}(\Mcal,\tau)$,
using the Haagerup--Schultz projections constructed in the previous theorem.
It is an analogue of the Schur upper triangular form for a matrix, and depends on choice of a continuous 
spectral ordering
(arising from a continous mapping from $[0,\infty)$ onto $\Cpx$).

\begin{thm}\label{decomposition theorem}
Let $(\Mcal,\tau)$ be tracial von Neumann algebra
where $\Mcal$ has separable predual
and let $T\in \Lc_{\log}(\Mcal,\tau)$.
Then there exists a larger tracial von Neumann algebra $(\Mcal_2,\tau_2)$ and there exist $N,Q\in \Lc_{\log}(\Mcal_2,\tau_2)$ such that
\begin{enumerate}[{\rm (a)}]
\item $T=N+Q$
\item $N$ is normal and the Brown measures of $N$ and $T$ coincide.
\item The Brown measure of $Q$ is concentrated at $0.$
\end{enumerate}
\end{thm}

A stronger version of this theorem (without passing to a larger von Neumann algebra) for bounded operators $T\in\Mcal$ was proved 
in~\cite{DSZ}.

Our final main result is an application of Theorem~\ref{decomposition theorem} to
the theory of singular traces
or more precisely, to the theory of traces on
sub-$\Mcal$-bimodules $\Bfr$ of $\Sc(\Mcal,\tau)$.
A {\em trace} on $\Bfr$ is a linear functional $\varphi$ on $\Bfr$ that satisfies $\varphi(XY)=\varphi(YX)$
for all $X\in\Mcal$ and $Y\in\Bfr$.
See also Section 2.7 in~\cite{LSZ}.
\begin{thm}\label{spectral trace thm}
Let $(\Mcal,\tau)$ be a II$_1$-factor; let
$\mathfrak{B}(\Mcal,\tau)\subset\Lc_{\log}(\Mcal,\tau)$
be an operator bimodule of $\Mcal$ that is
closed with respect to the logarithmic submajorization; let $\varphi$ be a trace on $\mathfrak{B}(\Mcal,\tau)$.
Then for every $T\in\mathfrak{B}(\Mcal,\tau),$ $\varphi(T)$ depends only on the Brown measure of $T$.
\end{thm}
The above result
is a far reaching generalization of the famous Lidskii trace formula~\cite{Li}, which established
that the value Tr(A) of the standard trace Tr on a trace class operator A is given by the sum of the eigenvalues of A (counting
multiplicities).
An analogue of Lidskii's result in the case of the standard trace on a type II$_1$ or II$_\infty$ factor
was obtained by L.\ Brown~\cite{Brown};
indeed, this was a major motivation for his introduction of Brown measure.
Whether a singular trace of an operator depends only on the operator's eigenvalues is a long standing and difficult problem
first suggested by Albrecht Pietsch~\cite{P90} (see also~\cite{P81}).
In the setting of ideals of compact operators, an answer was given in~\cite{DK98} and~\cite{SZ-AiM},
based on the fundamental papers~\cite{Kalton} and~\cite{DFWW}.
A positive answer was given for relevant operators in geometrically stable submodules of II$_\infty$ factors
in Corollary~6.10 of~\cite{DK-fourier}.
Theorem~\ref{spectral trace thm} answers the analogue of
Pietsch's question positively
in the setting of operator bimodules of $\Mcal$ that are closed with respect to logarithmic submajorization.

\medskip
Here is brief guide to the rest of the paper.
Section~\ref{sec:prelims} contains preliminary definitions and descriptions of results from other papers that we use.
Section~\ref{sec:strategy} contains a description of the strategy of the proofs of the main results.
The actual proofs are carried out in Sections \ref{sec:Llogultrapower}--\ref{sec:spectralityTr}
(see Section~\ref{sec:strategy} for more detail).
Finally, Appendix~\ref{app:hyperinv} contains some thoughts and questions about hyperinvariant subspaces for unbounded operators
affiliated to finite von Neumann algebras and their relevance to Haagerup--Schultz projections,
and Appendix~\ref{sec:NoCondExp} contains a result about lack of existence of conditional expectations for
unbounded operators.

\medskip
\noindent
{\bf Acknowledgement.}
The authors thank Yulia Kuznetsova,
Galina Levitina and Anna Tomskova for helpful comments.

\section{Preliminaries}
\label{sec:prelims}

\subsection{Tracial von Neumann algebras, II$_1$--factors and affiliated operators}
\label{subsec:II1}

We will study pairs $(\Mcal,\tau)$ where $\Mcal$ is a von Neumann algebra and $\tau$ is a normal, faithful, tracial state on $\Mcal$.
We will sometimes refer to the pair $(\Mcal,\tau)$ as a {\em tracial von Neumann algebra}.
(Note that a von Neumann algebra possesses a normal, faithful, tracial state
if and only if it is finite and countably decomposable.)
$\Mcal$ is called {\em diffuse} if it has no minimal (nonzero) projections and it is called a {\em factor} if its center is trivial.
The infinite dimensional finite von Neumann algebra factors are diffuse and are called II$_1$--factors,
and each of these has a unique tracial state $\tau,$ which is normal and faithful.
See, e.g., \cite{KR1} for details.

We will write $\Mcal\subseteq B(\HEu)$ for a Hilbert space $\HEu$.
For some purposes, it will be important that $\Mcal$ have separable predual,
which is equivalent to acting on a separable Hilbert space (see, for example, Lemma 1.8 of~\cite{Ya07})
It is well known and easy to show that every von Neumann algebra with separable predual is countably generated;
indeed, it has a countable weakly dense subset --- for example, this can be shown by (a) considering a countable,
norm-dense subset $\Phi$ of the predual, (b) for every finite subset $F\subseteq\Phi$, every mapping $\alpha:F\to\mathbb{Q}$
and every $n\in\Nats$, choosing $T(F,\alpha,n)\in\Mcal$ so that $|\phi(T)-\alpha(\phi)|<\frac1n$ for every $\phi\in F$,
whenever such exists, and (c) showing that the resulting collection of elements $T(F,\alpha,n)$ is weakly dense in $\Mcal$.
However, the converse statement does not hold;
an easy counter-example being $\ell^\infty(\Reals)$ where $\Reals$ is equipped with counting measure.
But for a tracial von Neumann algebra, countable generation does imply separable predual;
indeed the rational linear combinations of finite products of words in the generators are dense in $L^2(\Mcal,\tau)$,
which is the Hilbert space on which $\Mcal$ acts via the standard representation.

A closed, densely defined operator $A$ in $\HEu$ is said to be {\em affiliated to} $\Mcal$
if it commutes with every element in the commutant of $\Mcal$.
Whenever $\Mcal$ is finite, we have that every affiliated operator $A$ is $\tau$-measurable
(that is, $\tau(E_{|A|}(s,\infty))\to0$ as $s\to\infty$). The set of operators affiliated to $\Mcal$ is denoted by $\Sc(\Mcal,\tau)$.
It is a $*$-algebra and is finite, in the sense that all one-sided invertible elements are invertible.
Indeed, the latter assertion follows because, using the polar decomposition, one sees that
$T\in\Sc(\Mcal,\tau)$
is invertible if and only if its kernel is zero
and of course $T^*$ invertible if and only if $T$ is invertible.
$\Sc(\Mcal,\tau)$
also has the natural order structure defined by $A\ge0$ if and only if $A$ is self-adjoint and positive in the usual sense for unbounded operators.
See pp.\ 719--720 in \cite{DDP} for details.

The {\em measure topology} of E.\ Nelson~\cite{Ne74}
on $\Sc(\Mcal,\tau)$
is the translation invariant topology having neighborhood base at $0$ the set $\{N_{\eta,\delta}\mid \eta,\delta>0\}$,
where
\[
N_{\eta,\delta}=\{A\in\Sc(\Mcal,\tau)\mid \tau(1_{[\delta,\infty)}(|A|))<\eta\}.
\]
This topology is, thus, metrizable, for example by the metric
$$
d_\tau(A,B)=\sum_{n=1}^\infty 2^{-n}\tau(1_{[2^{-n},\infty)}(|A-B|).
$$
and this metric is complete (see~\cite{Ne74}).
It is clear that $\Mcal$ is dense in $\Sc(\Mcal,\tau)$ with respect to the measure topology.

If we have inclusion $(\mathcal{M}_1,\tau_1)\subset(\mathcal{M}_2,\tau_2)$ of tracial von Neumann algebras, meaning, an inclusion
$\Mcal_1\subset\Mcal_2$ so that the restriction of $\tau_2$ to $\Mcal_1$ is $\tau_1$, 
then the metric $d_{\tau_2}$ when restricted to $\Mcal_1$ agrees with $d_{\tau_1}$.
Hence, taking completions, the inclusion $\Mcal_1\subset\Mcal_2$ extends uniquely to a measure topology continuous
inclusion $\Sc(\Mcal_1,\tau_1)\subset\Sc(\Mcal_2,\tau_2)$.

\begin{subnota}\label{nota:relcomm}
For subsets $\Xc,\Yc\subseteq\Sc(\Mcal,\tau)$, the {\em relative commutant} of $\Yc$ in $\Xc$ is
$$\Xc\cap\Yc'=\{x\in\Xc\mid\forall y\in\Yc,\,xy=yx\}.$$
For $\Yc\subseteq\Sc(\Mcal,\tau)$, by {\em the von Neumann algebra generated by $\Yc$}
we mean the von Neumann algebra generated by the set consisting of all spectral projections of $|T|$
and the polar part in the polar decomposition of $T$, as $T$ ranges over all elements of $\Yc$.
By von Neumann's double commutant theorem and standard facts about spectral measure,
this is precisely $\Mcal\cap(\Mcal\cap\Yc')'$. 
\end{subnota}

\subsection{Singular value function}

For every $T$ affiliated to $\Mcal,$ the generalised singular value function $\mu(T),$ denoted $t\mapsto\mu(t,T)$ for $t\in(0,1),$
is defined by the formula (see, e.g., \cite{FackKosaki})
$$\mu(t,T)=\inf\{\|T(1-p)\|_{\infty}\mid p\in\operatorname{Proj}(\Mcal),\;\tau(p)\leq t\}.$$
It is continuous from the right in $t.$ Equivalently, $\mu(T)$ can be defined in terms of the distribution function $d_{|T|}$ of the operator $|T|.$ That is, setting
$$d_{|T|}(s)=\tau(E^{|T|}(s,\infty)),\quad s\geq0,$$
we obtain
$$\mu(t,T)=\inf\{s\geq0\mid d_{|T|}(s)\leq t\},\quad t>0.$$
Here, $E^{|T|}$ denotes the projection valued spectral measure of the operator $|T|.$
Thus we have $T\in\Lc_1(\Mcal,\tau)$ if and only if $\mu(T)$ is integrable, and then $\tau(|T|)=\int_0^1\mu(t,T)\,dt$.

\subsection{The topological $*$-algebra $\Lc_{\log}(\Mcal,\tau)$}

In what follows,
$$\Lc_{\log}(\Mcal,\tau)=\{A\in\Sc(\Mcal,\tau)\mid \log^+(|A|)\in\Lc_1(\Mcal,\tau)\},$$
where $\log^+(\lambda)=\max\{\log(\lambda),0\}$ for every $\lambda\geq0$
and we let
$$\|T\|_{\log}=\tau(\log(1+|T|)),\qquad(T\in\Lc_{\log}).$$
Thus, $T\in\Lc_{\log}(\Mcal,\tau)$ if and only if $\log^+(\mu(T))$ is integrable and then
$$\|T\|_{\log}=\int_0^1\log(1+\mu(t,T))\,dt=\int_{[0,\infty)}\log(1+t)\,d\nu_{|T|}(r),$$
where $\nu_{|T|}$ is the spectral distribution measure of $|T|$, i.e., $\tau$ composed with the spectral measure of $|T|$.

Note that if $(\Mcal_1,\tau_1)$ and $(\Mcal_2,\tau_2)$ are tracial von Neumann algebras and if $\sigma:\Mcal_1\to\Mcal_2$
is a trace-preserving normal $*$-homomorphism, (which must, then, be injective) then
the inclusion
$\Sc(\Mcal_1,\tau_1)\to\Sc(\Mcal_2,\tau_2)$ described in \S\ref{subsec:II1} restricts to an inclusion
\begin{equation}\label{eq:sigmat}
\sigmat:\Lc_{\log}(\Mcal_1,\tau_1)\to\Lc_{\log}(\Mcal_2,\tau_2),
\end{equation}
which is an extension of $\sigma$ and preserves $\|\cdot\|_{\log}$.

In Theorem 4.9 of \cite{DSZ:Llog}, the authors prove:
\begin{subthm}\label{thm:LlogTopalg} $\|\cdot\|_{\log}$ is an F-norm on $\Lc_{\log}(\Mcal,\tau)$
and, with respect to the metric $d_{\log}(S,T)=\|S-T\|_{\log}$, $\Lc_{\log}(\Mcal,\tau)$ is a complete topological $*$-algebra.
\end{subthm}

Note that being an F-norm entails
$$\|S+T\|_{\log}\le\|S\|_{\log}+\|T\|_{\log}.$$

The following elementary inequalities (see Lemmas 4.1 and 4.3 of~\cite{DSZ:Llog} for proofs) will be useful.
\begin{sublem}\label{lem:elemineq} Let $S,T\in\Lc_{\log}(\Mcal,\tau).$ Then
\begin{enumerate}[{\rm (a)}]
\item\label{it:elemBddMult} if $S$ is bounded, then $\|ST\|_{\log}\le\max(\|S\|,1)\,\|T\|_{\log}$,
\item\label{it:elemto0} $\lim_{r\to0}\|rT\|_{\log}=0$,
\item\label{it:elemMult} $\|ST\|_{\log}\le\|S\|_{\log}+\|T\|_{\log}.$
\end{enumerate}
\end{sublem}

We will also need the following elementary estimate, which is in a similar vein.
\begin{sublem}\label{lem:elemlog+}
For every $X,Y\in\Lc_{\log}(\Mcal,\tau),$ we have
$$\tau(\log^+(|X+Y|))\leq \tau(\log^+(2|X|))+\tau(\log^+(2|Y|)).$$
\end{sublem}
\begin{proof} By Lemma 4.3 in \cite{FackKosaki}, it suffices to prove the assertion for $X,Y\geq 0$.
Put $T=\max\{2X,1\}$ and $S=\max\{2Y,1\}$.
We have
$$\Delta(S^{-1}\cdot\frac{T+S}{2}\cdot T^{-1})=\Delta(\frac{S^{-1}+T^{-1}}{2})\leq\Delta(1)=1.$$
Hence,
$$\Delta(\frac{T+S}{2})\leq\Delta(T)\Delta(S).$$
Taking the logarithm, we arrive at
\begin{align*}
\tau(\log_+(X+Y))\leq\tau(\log_+(\frac{T+S}{2}))=\log(\Delta(\frac{T+S}{2}))\leq\\
\leq\log(\Delta(S))+\log(\Delta(T))=\tau(\log_+(2X))+\tau(\log_+(2Y)).
\end{align*}
\end{proof}

\subsection{Fuglede-Kadison determinant and Brown measure}
\label{subsec:FK}

In \cite{HS1}, Haagerup and Schultz constructed a mapping $\Delta:\Lc_{\log}(\Mcal,\tau)\to\mathbb{R}_+$ which is a homomorphism with respect to the multiplication.
That is,
\begin{equation}\label{det homomorphism}
\Delta(ST)=\Delta(S)\Delta(T),\quad S,T\in\Lc_{\log}(\Mcal,\tau).
\end{equation}
This mapping is called the {\em Fuglede--Kadison determinant};
it was introduced in~\cite{FK}.
It is defined by
$$
\Delta(T)=\exp(\tau(\log(|T|))),\quad T\in\Lc_{\log}(\Mcal,\tau).
$$
Using $T\in\Lc_{\log}(\Mcal,\tau)$ and monotone convergence, we immediately see
\begin{equation}\label{eq:Deltaeps}
\Delta(T)^2=\Delta(|T|^2)=\lim_{\eps\to0}\Delta(|T|^2+\eps).
\end{equation}
We also have
$$
\Delta(S)\leq\Delta(T),\quad 0\leq S\leq T\in\Lc_{\log}(\Mcal,\tau).
$$
It follows from Lemma 2.4 in \cite{HS1} that $T\in\Lc_{\log}(\Mcal,\tau)$ implies $T^{-1}\in\Lc_{\log}(\Mcal,\tau)$ whenever $\Delta(T)\neq0.$

For every operator $T\in\Lc_{\log}(\Mcal,\tau),$ the function
\begin{equation}\label{subhar}
\lambda\to\tau(\log(|T-\lambda|)),\quad\lambda\in\mathbb{C},
\end{equation}
is shown to be subharmonic in \cite{Brown} (for the special case of bounded $T$) and in \cite{HS1} in general.
Moreover, this function is not identically $-\infty$ (see, e.g., the proof of Proposition 2.9 of~\cite{HS1}).
A probability measure $\nu_T$ was constructed in \cite{HS1},
based on earlier results of L.\ Brown~ \cite{Brown}.
It is called the {\em Brown measure} of $T$ and satisfies
\begin{equation}\label{br def}
\tau(\log(|T-\lambda|))=\int_{\mathbb{C}}\log(|z-\lambda|)d\nu_T(z),\qquad(\lambda\in\mathbb{C}).
\end{equation}
This $\nu_T$ can be viewed as the II$_1$--analogue of the spectral counting measure (according
to algebraic multiplicity) on matrices. It can be recovered by taking the Laplacian of the mapping in \eqref{subhar}.

Note that $\nu_{T^*}$ is the push-forward measure of $\nu_T$ under conjugation and (see Proposition 2.16 in \cite{HS1}) if $T^{-1}\in\Lc_{\log}(\Mcal, \tau)$, then $\nu_{T^{-1}}$ is the push-forward measure of $\nu_T$ under $z\mapsto z^{-1}.$

Recall for $T\in\Sc(\Mcal,\tau)$ that a projection $p\in\Mcal$ is said to be {\em $T$-invariant} if $Tp=pTp$, and that we can then write
$T=\left(\begin{smallmatrix}
A&B\\
0&C
\end{smallmatrix}\right),$
where $A=Tp,$ $B=pT(1-p),$ and $C=(1-p)T$.

The following is Proposition 2.24 of \cite{HS1} and a consequence of it.
For  Fuglede--Kadison determinants, we write
$\Delta_\Mcal$ for the one in $(\Mcal,\tau)$,
$\Delta_{p\Mcal p}$ for the one in $(p\Mcal p,\tau(p)^{-1}\tau|_{p\Mcal p})$
and 
$\Delta_{(1-p)\Mcal(1-p)}$ for the one in the corresponding cut-down by $1-p$.
\begin{subthm}\label{det nu matrix}
If $T\in\Lc_{\log}(\Mcal,\tau)$ and if $p\in\Mcal$ is a $T$-invariant projection, then
\begin{equation}\label{DeltaT}
\Delta_{\Mcal}(T)=\Delta_{p\Mcal p}(A)^{\tau(p)}\Delta_{(1-p)\Mcal(1-p)}(C)^{\tau(1-p)}
\end{equation}
and
\begin{equation}\label{nuT}
\nu_T=\tau(p)\nu_A+\tau(1-p)\nu_C,
\end{equation}
where $\nu_A$ means the Brown measure of $A$ in $(p\Mcal p,\tau(p)^{-1}\tau|_{p\Mcal p})$ and similarly for $\nu_C$ and the cut-down by $1-p$.
\end{subthm}
We will use the equation \eqref{DeltaT} also in the case of $p=0$ or $p=1,$ by making the convention $\Delta_{\{0\}}(0)^0=1.$

\begin{subcor}\label{nu matrix cor} If $T\in\Lc_{\log}(\Mcal,\tau)$ and if $P_1,P_2\in\Mcal$ are $T$-invariant projections such that $P_1\leq P_2,$ then
$$
{\rm supp}(\nu_{TP_1})\subset{\rm supp}(\nu_{TP_2}).
$$
\end{subcor}
\begin{proof} One may write
$TP_2=\left(\begin{smallmatrix}
A&B\\
0&C
\end{smallmatrix}\right),$
where $A=TP_1,$ $B=P_1T(P_2-P_1),$ and $C=(P_2-P_1)T.$ It follows from Theorem \ref{det nu matrix} that
$$\nu_{TP_2}=\frac{\tau(P_1)}{\tau(P_2)}\nu_{TP_1}+\frac{\tau(P_2-P_1)}{\tau(P_2)}\nu_C.$$
In particular, we have
$$\nu_{TP_2}\geq\frac{\tau(P_1)}{\tau(P_2)}\nu_{TP_1}.$$
This proves the assertion.
\end{proof}

We say that an operator $Q\in\Lc_{\log}(\Mcal,\tau)$ {\em $\tau$-quasinilpotent} if $\nu_Q=\delta_0$.
\begin{subconj}
$Q$ is $\tau$-quasinilpotent if and only if $|Q^n|^{1/n}\to0$ in measure.
\end{subconj}
The conjecture is known to be true when $Q$ is bounded.
Indeed, as a consequence of Theorem~8.1 of~\cite{HS2}, $\nu_Q=\delta_0$ is equivalent to 
$|Q^n|^{1/n}\to0$ in strong-operator-topology,
which, for a bounded sequence in a tracial von Neumann algebra is equivalent to convergence in measure to $0$.

\subsection{Mild operators and Haagerup--Schultz projections}

Recall that, for an operator $T$ affiliated to $\Mcal$, a projection $p\in\Mcal$ (acting on a Hilbert space $\HEu$)
is said to be $T$-invariant if $Tp=pTp$,
which is equivalent to the statement that $T\xi\in p\HEu$ whenever $\xi\in p\HEu\cap\operatorname{dom}(T)$.
Clearly, the set of $T$-invariant projections forms a lattice that is closed under taking arbitrary suprema.

Though the paper \cite{HS2} is mostly concerned with bounded operators, one of its central technical results (Theorem 6.6) is stated and proved for the operators in $\Lc_{\log}(\Mcal,\tau)$.
Here we state a slightly different form of this theorem, whose proof is only a minor perturbation of Haagerup and Schultz's proof of their result.
\begin{subthm}\label{hs integral}
Let $\mathcal{B}$ be a closed disk in $\mathbb{C}$.
Suppose $T\in\Lc_{\log}(\Mcal,\tau)$ has empty point spectrum and suppose that for some $p\in(\frac12,1)$
the map $\lambda\mapsto(T-\lambda)^{-1}$ is Lipschitz from the boundary of $\Bc$ to $\Lc_p(\Mcal,\tau)$.
Suppose $0<\nu_T(\Bc)<1$.
Then there is a $T$-invariant projection $q$ for $T$, such that
\begin{enumerate}[{\rm (i)}]
\item the Brown measure of $Tq$ with respect to the renormalized trace $\tau(q)^{-1}\tau|_{q\Mcal q}$
is supported in $\Bc$
\item the Brown measure of $(1-q)T$ with respect to  the renormalized trace $\tau(1-q)^{-1}\tau|_{(1-q)\Mcal(1-q)}$
is supported in the closure of $\Cpx\backslash\Bc$.
\end{enumerate}
\end{subthm}
Briefly, the proof goes by construction of an (unbounded) idempotent
\begin{equation}\label{idempotent def}
E(T,\mathcal{B})=\frac1{2\pi i}\int_{\partial\mathcal{B}}\frac{dz}{z-T},
\end{equation}
where they show that the above integral converges in $\Lc_p(\Mcal,\tau)$ as a Riemann integral,
and then letting the projection
$q$ be the range projection for $E(T,\mathcal{B}).$

For future use,
we now prove a straightforward result about changes of variable in the Riemann integrals
of the sort used in~\eqref{idempotent def} above.
\begin{sublem}\label{lem:RIcov}
Suppose $p>\frac12$, $f:[0,1]\to\Lc_p([\Mcal,\tau)$ satisfies
\begin{equation}\label{eq:fLip}
\|f(s)-f(t)\|_p\le C|s-t|^\alpha
\end{equation}
for some $\alpha>1-\frac1p$, some $C>0$ and all $s,t\in[0,1]$.
If $\gamma:[0,1]\to[0,1]$ is a smooth bijection, then
$$\int_0^1f(s)ds=\int_0^1f(\gamma(s))\gamma'(s)ds.$$
\end{sublem}
\begin{proof} Both integrals do exist by Theorem 6.2 in \cite{HS2} (in fact, by the argument on p.67 immediately before Theorem 6.2). Consider the expression
$$I_n=\sum_{k=0}^{n-1}f(\gamma(\frac{k}{n}))(\gamma(\frac{k+1}{n})-\gamma(\frac{k}{n})).$$
By Theorem 6.2 in \cite{HS2}, we have that
$$\lim_{n\to\infty}I_n=\int_0^1f(s)ds$$
as $n\to\infty.$ On the other hand, we have
$$\gamma(\frac{k+1}{n})-\gamma(\frac{k}{n})=\frac1n\gamma'(\frac{k}{n})+\frac{O(1)}{n^2}.$$
Thus, $I_n=J_n+O_n,$ where
$$J_n=\frac1n\sum_{k=0}^{n-1}f(\gamma(\frac{k}{n}))\gamma'(\frac{k}{n}),\quad O_n=\sum_{k=0}^{n-1}f(\gamma(\frac{k}{n}))\frac{O(1)}{n^2}.$$
Clearly, $J_n$ is a Riemann sum for the function $s\to f(\gamma(s))\gamma'(s).$ By Theorem 6.2 in \cite{HS2}, we have that
$$\lim_{n\to\infty}J_n=\int_0^1f(\gamma(s))\gamma'(s)\,ds.$$
Finally, from~\eqref{eq:fLip}, the map $f$ is continuous from with respect to the $\Lc_p$ norm, and we have
\begin{multline*}
\|O_n\|_p^p=n^{-2p}O(1)\|\sum_{k=0}^{n-1}f(\gamma(\frac{k}{n}))\|_p^p\leq n^{-2p}O(1)\sum_{k=0}^{n-1}\|f(\gamma(\frac{k}{n}))\|_p^p \\
\leq\frac{O(1)}{n^{2p-1}}\max_{t\in[0,1]}\|f(t)\|_{\mathcal{L}_p}^p.
\end{multline*}
Thus, $O_n\to0$ as $n\to\infty$ and the assertion follows.
\end{proof}

Here is a condition that is slightly stronger than the ones employed by Haagerup and Schultz, but will be convenient for us.

\begin{subdefi}\label{def:mild} Let $T\in\Lc_{\log}(\Mcal,\tau)$ and let $\Dc\subseteq\Cpx$. 
We say $T$ is {\em mild with Lipschitz domain} $\Dc$ or, as abbreviation, just mild$(\Dc)$, if it satisfies the following conditions:
\begin{enumerate}[{\rm (a)}]
\item\label{it:milda} $\Delta(T-\lambda)\neq0$ for every $\lambda\in\mathbb{C}.$
\item\label{it:mildb} The Brown measure of $T$ is absolutely continuous with respect to the Lebesgue measure.
\item\label{it:mildc} The operator-valued function $z\mapsto (z-T)^{-1}$ is locally Lipschitz from $\Dc$ to $\Lc_p(\Mcal,\tau)$
for some $p\in(\frac12,1)$.
\end{enumerate}
Clearly, if $\Dc_1\subseteq\Dc_2$, and $T$ is mild$(\Dc_2)$, then $T$ is mild$(\Dc_1)$.
If we say that $T$ is simply {\em mild}, then we mean it is mild$(\Cpx)$.
\end{subdefi}

\begin{subdefi}\label{def:hs proj}
Let $T\in\Lc_{\log}(\Mcal,\tau)$ and let $\mathcal{B}\subset\mathbb{C}$ be a Borel set. A projection $p\in\Mcal$ is said to be a
{\em Haagerup--Schultz projection} for $T$ and the set $\mathcal{B}$ if
\begin{enumerate}[{\rm (a)}]
\item\label{hdefa} $p$ is $T$-invariant.
\item\label{hdefb} $\tau(p)=\nu_T(\mathcal{B}).$
\item\label{hdefc} provided $\nu_T(\mathcal{B})\ne0,1$, we have
$$\nu_{Tp}=\frac1{\nu_T(\mathcal{B})}\nu_T|_{\mathcal{B}},\quad \nu_{(1-p)T}=\frac1{\nu_T(\mathbb{C}\backslash\mathcal{B})}\nu_T|_{\mathbb{C}\backslash\mathcal{B}},$$
where the above Brown measures are taken with respect to the appropriately renormalized traces.
\end{enumerate}
\end{subdefi}

Note that \eqref{hdefb} actually follows from \eqref{hdefa} and \eqref{hdefc} by Theorem \ref{det nu matrix}.

\begin{subrem}\label{rem:mildHS}
For every disk $\Bc\subseteq\Cpx$, if $T\in\Lc_{\log}(\Mcal,\tau)$ is mild$(\partial\Bc)$, then
Haagerup and Schultz's Theorem~\ref{hs integral} above constructs a Haagerup-Schultz projection $P(T,\Bc)$
for $T$ and $\Bc$.
(The disk can be closed or open or anything in between, because condition~\eqref{it:mildb} of Definition~\ref{def:mild}
ensures that $\nu_T(\partial B)=0$.)
\end{subrem}

\subsection{Certain perturbations}
\label{subsec:Z}

A key technical tool in~\cite{HS2} was the use of an unbounded operator $Z=xy^{-1}$ affiliated to $\Mcal$, where $x,y\in\Mcal$ are 
circular operators so that the triple consisting of $T$, $x$ and $y$ is $*$--free.
Enlarging $\Mcal$ by taking the free product with $L(\mathcal{F}_4)$, if necessary, we may
without loss of generality assume there is such an operator $Z$.
It was shown in Theorem 5.2 from \cite{HS1} that $Z\in\Lc_p(\Mcal,\tau)$ for all $0<p<1$.
Also the following key result is a combination of Proposition 4.5, Corollary 4.6 and Theorem 5.1 in~\cite{HS2}.
Though the quoted theorems are for $T\in\Mcal,$ they trivially extend to $\Lc_{\log}(\Mcal,\tau)$.
In fact, we will repeat Theorem~5.1 of~\cite{HS2} in the full generality that its proof in~\cite{HS2} permits:
\begin{subthm}\label{thm:HS S(M)}
Let $Z$ be as described above.
For every $T\in\Sc(\Mcal,\tau)$ and every $p\in(0,1)$, $T+Z$ has an inverse $(T+Z)^{-1}\in\Lc_p(\Mcal,\tau)$.
Moreover, for each $p\in(0,\frac23)$, there is a constant $C_p$ such that
$$\forall S,T\in\Sc(\Mcal,\tau),\qquad\|(S+Z)^{-1}-(T+Z)^{-1}\|_p\le C_p\|S-T\|.$$
\end{subthm}

Combining the above with Proposition 4.5 and Corollary 4.6 of~\cite{HS2}, we have:
\begin{subthm}\label{thm:TZmild}
Given $T\in\Lc_{\log}(\Mcal,\tau)$, letting $Z$ be as described above and taking $\eps>0$, the operator $T+\eps Z$ is mild.
\end{subthm}

\subsection{Conditional expectation}
\label{subsec:condexp}

If $\mathcal{D}$ is a von Neumann subalgebra of the finite von Neumann algebra $\Mcal$ with normal, faithful, tracial state $\tau,$
then there exists a unique linear operator $\Exp_{\mathcal{D}}:\Mcal\to\mathcal{D}$ such that,  for all $A\in\Mcal$ and $B\in\mathcal{D}$,
\begin{enumerate}[{\rm (a)}]
\item $\Exp_{\mathcal{D}}(AB)=\Exp_{\mathcal{D}}(A)B$
\item $\Exp_{\mathcal{D}}(BA)=B\,\Exp_{\mathcal{D}}(A)$
\item $\tau(\Exp_{\mathcal{D}}(A))=\tau(A)$.
\end{enumerate}

Though conditional expectation is undefined for general unbounded operators (outside of $\Lc_1(\Mcal,\tau)$), we use the same notation $\Exp_{\mathcal{D}}$ for certain mappings we construct in the subsequent sections. On bounded operators, this operation coincides with the conditional expectation as defined above.

\subsection{Ultrapowers of von Neumann algebras}
\label{subsec:ultraP}

Let $(\Mcal,\tau)$ be a tracial von Neumann algebra. Consider the $*$-algebra
$$l_{\infty}(\Mcal)=\Big\{\{A(k)\}_{k\geq0}\in\prod_{k=0}^{\infty}\Mcal\;\Big|\; \sup_{k\geq0}\|A(k)\|_{\infty}<\infty\Big\}.$$
Fix a free ultrafilter $\omega$ on $\mathbb{Z}_+$ and consider the ideal
$$\mathscr{I}=\Big\{\{A(k)\}_{k\geq0}\in l_{\infty}(\Mcal)\;\Big|\; \lim_{k\to\omega}\tau(A(k)^*A(k))=0\Big\}.$$
The $*$-algebra quotient $\Mcal_\omega=l_{\infty}(\Mcal)/\mathscr{I}$ is known (see, for example, Theorem 4.6 of \cite{Tak3})
to be a finite von Neumann algebra equipped with a normal faithful tracial state $\tau_\omega$ defined by the formula
\begin{equation}\label{ultra trace def}
\tau_\omega(\pi(\{A(k)\}_{k\geq0}))=\lim_{k\to\omega}\tau(A(k)),\quad \{A(k)\}_{k\geq0}\in l_{\infty}(\Mcal),
\end{equation}
where $\pi:l_{\infty}(\Mcal)\to\Mcal_\omega$ is the canonical factorization. We say that $(\Mcal_\omega,\tau_\omega)$ is an ultrapower of $(\Mcal,\tau).$

We let $\delta_\Mcal:\Mcal\to l^\infty(\Mcal)$ be the diagonal embedding $x\mapsto(x,x,\ldots)$ and we see that $\pi\circ\delta_\Mcal$ is
an embedding of $(\Mcal,\tau)$ into $(\Mcal_\omega,\tau_\omega)$, namely, a $*$-homomorphism $\Mcal\to\Mcal_\omega$
that satisfies $\tau_\omega\circ\pi\circ\delta_\Mcal=\tau$.

\subsection{Operator bimodules and traces}

If $(\Mcal,\tau)$ is a tracial von Neumann algebra,
then a linear subspace $\mathfrak{B}(\Mcal,\tau)\subset\Sc(\Mcal,\tau)$ is said to be
an operator bimodule if $AB,BA\in\mathfrak{B}(\Mcal,\tau)$ for all $A\in\mathfrak{B}(\Mcal,\tau),$ $B\in\Mcal$.
If $\Mcal$ is a factor, then $A\in\mathfrak{B}(\Mcal,\tau)$ and $\mu(C)=\mu(A)$ imply that $C\in\mathfrak{B}(\Mcal,\tau).$

If $\Mcal$ is a unital subfactor of a type II$_1$-factor $\mathcal{N}$, (with trace also denoted $\tau)$, then we can define an operator bimodule
$$\mathfrak{B}(\mathcal{N},\tau)=\{A\in\Sc(\mathcal{N},\tau)\mid \mu(A)=\mu(C)\mbox{ for some }C\in\mathfrak{B}(\Mcal,\tau)\}.$$

A linear functional $\varphi:\mathfrak{B}(\Mcal,\tau)\to\mathbb{C}$ is called trace
if $\varphi(AB)=\varphi(BA)$ for all $A\in\mathfrak{B}(\Mcal,\tau),$ $B\in\Mcal$.
When $\Mcal$ is a factor,
it follows that $\varphi(A)=\varphi(B)$ whenever $0\leq A,B\in\mathfrak{B}(\Mcal,\tau)$ are such that $\mu(A)=\mu(B).$
(This assertion is folklore.
For convenience, we present a proof of slightly more general result in Lemma \ref{normal phi lemma} below.)

The following theorem is a restatement of Theorem 5.2 in \cite{KScanada}. Though the latter theorem is stated in a less general fashion, its proof is perfectly applicable.

\begin{subthm}\label{ks extension thm} Let $(\Mcal,\tau)$ be a II$_1$-factor; let $\mathfrak{B}(\Mcal,\tau)$ be an operator bimodule; let $\varphi:\mathfrak{B}(\Mcal,\tau)\to\mathbb{C}$ be a trace. If $\Mcal$ is a subfactor of $\mathcal{N}$ as above,
then $\varphi$ extends to a trace on $\mathfrak{B}(\mathcal{N},\tau)$ by setting
$$\varphi(A)=\varphi(C),\mbox{ whenever }\mu(A)=\mu(C),\quad 0\leq A\in\mathfrak{B}(\mathcal{N},\tau),\quad  0\le C\in\mathfrak{B}(\Mcal,\tau).$$
\end{subthm}

The operator $B\in\Lc_{\log}(\Mcal,\tau)$ is said to be logarithmically submajorized by the operator $A\in\Lc_{\log}(\Mcal,\tau)$ (written $B\prec\prec_{\log}A$) if
$$\int_0^t\log(\mu(s,B))ds\leq\int_0^t\log(\mu(s,A))ds,\qquad 0<t<1,$$
with $-\infty$ allowed for values of the integrals.

\begin{subdefi} We say that bimodule $\mathfrak{B}(\mathcal{M},\tau)\subset \Lc_{\log}(\Mcal,\tau)$ is closed with respect to the logarithmic submajorization if for every $A\in\mathfrak{B}(\mathcal{M},\tau)$ and for every $B\in\Lc_{\log}(\Mcal,\tau)$
with $B\prec\prec_{\log} A$,
we also have $B\in\mathfrak{B}(\mathcal{M},\tau).$
\end{subdefi}

\section{Strategy of the proofs}
\label{sec:strategy}

The proof of our first main result, Theorem~\ref{construction theorem}, on existence of Haagerup--Schultz
projections,
was inspired by and builds upon the proof by Haagerup and Schultz~\cite{HS2} for bounded operators,
but many
further ideas and results are required to treat the unbounded operators $T\in\Lc_{\log}(\Mcal,\tau)$.
Here is the basic outline of the method.
\begin{enumerate}[{\rm (a)}]
\item Perturb $T$ to $T+\frac1{n+1}Z$, where $Z$ is the unbounded operator described in~\S\ref{subsec:Z}.
By results of Haagerup and Schultz, the resulting operator is mild.
This is Theorem~\ref{thm:TZmild}.
\item\label{it:HSmilddisk}
Use analogues of Dunford functional calculus to construct Haagerup-Schultz projections $P(\Tt,\Bc)$ for mild operators $\Tt$
and disks $\Bc$.
This is Haagerup and Schultz's Theorem~\ref{hs integral}.
\item\label{it:mildHS}
Manipulate and assemble the projections $P(\Tt,\Bc)$ obtained as described
in~\eqref{it:HSmilddisk} above to find Haagerup--Schultz
projectios $P(\Tt,\Bc)$ for mild operators $\Tt$ and closed sets $\Bc$ whose boundaries are Lebesgue null sets.
We do this in Section~\ref{mild section}, where we also prove some convenient properties of the $P(\Tt,\Bc)$.
\item\label{it:HSprojs}
Using sequences of Haagerup--Schultz projections $P(T+\frac1{n+1}Z,\Bc)$ constructed as described
in~\eqref{it:mildHS}, find Haagerup--Schultz projections $P(T,\Bc)$ for $T$.
These projections are constructed in a larger ultrapower von Neumann algebra, $\Mcal_\omega$.
These results are pushed further to obtain $P(T,\Bc)$ for arbitrary Borel subsets $\Bc$,
thus proving Theorem~\ref{construction theorem}.
This is done in Section~\ref{hs construct}.
The use of ultrapowers in the context of operators belonging to the class of $\Lc_{\log}$ requires special attention.
One cannot just take the ultrapower of $\Lc_{\log}(\Mcal,\tau)$.
This theory is developed in Section~\ref{sec:Llogultrapower}.
\end{enumerate}

The proof of our second main result, Theorem~\ref{decomposition theorem}, about upper triangular decompositions, 
was inspired by and builds upon our proof~\cite{DSZ} for bounded operators, but again many further ideas
and results are required to treat unbounded operators.
In brief, 
by using a Peano curve to ``order'' the spectrum of $T$, consider the abelian
von Neumann algebra generated by a corresponding increasing family of Haagerup--Schultz projections
constructed as described in~\eqref{it:HSprojs}.
This is used together with direct integral decomposition theory (for unbounded operators) from~\cite{DNSZ}
and some sort of conditional expectations to prove Theorem~\ref{decomposition theorem}.
We do this in Section~\ref{sec:pfMain}.
We cannot use genuine conditional expectations in the context of unbounded operators, (see Appendix~\ref{sec:NoCondExp})
but in the special circumstances we consider, there is a \lq\lq Faux Expectation\rq\rq (again using ultrapowers
and the results of Section~\ref{sec:Llogultrapower}) that is enough.
We develop the theory of Faux Expectations in Section~\ref{sec:FauxExp}.

The proof of our main application, Theorem~\ref{spectral trace thm} on spectrality of traces, is accomplished by using Theorem~\ref{decomposition theorem} to  write $T=N+Q$, where $N$ is a normal operator having the same Brown measure as $T$ and $Q$ has Brown measure $\delta_0.$ The key step in the proof is to show that the \lq\lq spectrally trivial\rq\rq part $Q$ must vanish under all traces. This is achieved via fundamental characterization of commutator subspaces $[\mathcal{M},\mathfrak{B}(\mathcal{M},\tau)]$ achieved in \cite{DK-fourier}. The appeal to that characterization is made possible thanks to our assumption that bimodule $\mathfrak{B}(\mathcal{M}, \tau)$ is closed with respect to logarithmic submajorization and spectral estimates proved in Section 9 using the technique earlier developed in \cite{Kalton,DK-fourier,LSZ}. The details of of the argument are presented in Section~\ref{sec:spectralityTr}.

\section{Some operators affiliated to an ultrapower of a tracial von Neumann algebra}
\label{sec:Llogultrapower}

This section is partly inspired by the theory of ultrapowers $\Lc_p(\Mcal,\tau)$.
It is generally known that the ultrapower of $\Lc_p(\Mcal,\tau)$ is strictly larger than the
$\Lc_p$-space of the ultrapower of $(\Mcal,\tau)$.
However, Lemma 2.13 in \cite{HRS} constructs a subspace in the ultrapower of $\Lc_p(\Mcal,\tau)$,
which is canonically isomorphic to $\Lc_p$ of the ultrapower of $(\Mcal,\tau)$.
The subspace is implicitly characterized in \cite{HRS} in terms of uniform integrability of the $p$-th power
(see Definition 2.4 in \cite{HRS}).

Let $l_{\infty}(\Lc_{\log}(\mathcal{M},\tau))$ denote the set of all sequences $A=\{A(k)\}_{k\geq0}$ in $\Lc_{\log}(\mathcal{M},\tau)$ that are bounded with respect to $\|\cdot\|_{\log}$
with termwise operations.
Using Lemma~\ref{lem:elemineq}, we see that it is a $*$-algebra.
The algebra $\Fsc(\mathcal{M},\tau)$ below can be described as the subalgebra of those elements in $l_{\infty}(\Lc_{\log}(\mathcal{M},\tau))$ whose logarithms are uniformly integrable.

Consider the set
$$\Fsc(\Mcal,\tau)=\Big\{A\in l_{\infty}(\Lc_{\log}(\mathcal{M},\tau))\;\Big|\;\lim_{n\to\infty}\sup_{k\geq0}\left\|\frac{A(k)}n\right\|_{\log}=0\Big\}.$$
Using Lemma~\ref{lem:elemineq}, we see that it is a $*$-algebra.
For example, if $A,B\in \Fsc(\Mcal,\tau)$, then by Lemma~\ref{lem:elemineq}\eqref{it:elemMult} we have 
$$\lim_{n\to\infty}\sup_{k\geq0}\left\|\frac{A(k)B(k)}n\right\|_{\log}\leq \lim_{n\to\infty}\sup_{k\geq0}\left\|\frac{A(k)}{n^{1/2}}\right\|_{\log}+\lim_{n\to\infty}\sup_{k\geq0}\left\|\frac{B(k)}{n^{1/2}}\right\|_{\log},$$
so $AB\in\Fsc(\Mcal,\tau)$.

The next lemma gives an alternative characterization of $\Fsc(\Mcal,\tau)$.

\begin{lem}\label{lem:Fchar}
Let $A=\{A(k)\}_{k\geq0}\in l_{\infty}(\Lc_{\log}(\mathcal{M},\tau))$.
Then the following are equivalent:
\begin{enumerate}[{\rm (a)}]
\item\label{it:AinF} $A\in\Fsc(\Mcal,\tau)$
\vskip1ex
\item\label{it:log+An} $\displaystyle\lim_{n\to\infty}\sup_{k\ge0}\tau(\log^+\left(\frac{|A(k)|}n\right))=0$,
\vskip1ex
\item\label{it:AE} $\displaystyle\lim_{n\to\infty}\sup_{k\ge0}\|A(k)E_{|A(k)|}(n,\infty)\|_{\log}=0.$
\end{enumerate}
\end{lem}
\begin{proof}
Since
$$\tau(\log^+\left(\frac{|A(k)|}n\right))\le\tau(\log\left(1+\frac{|A(k)|}n\right))=\left\|\frac{A(k)}n\right\|_{\log},$$
we see immediately that \eqref{it:AinF} implies \eqref{it:log+An}

To show \eqref{it:log+An} implies \eqref{it:AE}, suppose $\eps>0$ and $n\in\mathbb{N}$ are such that
$$\tau(\log^+(\frac{|A(k)|}{n^{\frac12}}))<\eps,\quad k\geq0.$$
We have
$$\log(1+t)\chi_{\{t>n\}}\leq 2\log(t)\chi_{\{t>n\}}\leq 4\log(\frac{t}{n^{\frac12}})\chi_{\{t>n\}}\leq 4\log^+(\frac{t}{n^{\frac12}}),\quad t>0.$$    
Since 
$$
\big|A(k)E_{|A(k)|}(n,\infty)\big|=|A(k)|E_{|A(k)|}(n,\infty),
$$
it follows that
\begin{multline*}
\|A(k)E_{|A(k)|}(n,\infty)\|_{\log}
=\int_{(n,\infty)}\log(1+t)\,d\nu_{|A(k)|}(t)\le 4\tau(\log^+(\frac{|A(k)|}{n^{\frac12}}))<4\eps
\end{multline*}
for all $k\geq0.$

To show \eqref{it:AE} implies \eqref{it:AinF}, suppose $\eps>0$ and $n\in\mathbb{N}$ are such that $n\geq\varepsilon^{-2}$ and such that
$$\|A(k)E_{|A(k)|}(n^{\frac12},\infty)\|_{\log}<\eps,\quad k\geq0.$$
Using the triangle inequality in $\Lc_{\log}(\mathcal{M},\tau)$, we infer that
\begin{multline*}
\left\|\frac{A(k)}n\right\|_{\log}\leq \|\frac{A(k)}{n}E_{|A(k)|}(n^{\frac12},\infty)\|_{\log}+\|\frac{A(k)}{n}E_{|A(k)|}[0,n^{\frac12}]\|_{\log}\\
\leq \|A(k)E_{|A(k)|}(n^{\frac12},\infty)\|_{\log}+\|n^{-\frac12}E_{|A(k)|}[0,n^{\frac12}]\|_{\log}\leq \varepsilon+\frac{1}{n^{\frac12}}\leq 2\varepsilon.
\end{multline*}
\end{proof}

We already observed that $\Fsc(\Mcal,\tau)$ is a $*$-algebra. We also endow it with an ordering determined by setting $0\le A$ if and only if $0\le A(k)$ for every $k\ge0$, where the latter inequality
is with respect to the order
on $\Sc(\Mcal,\tau)$. Similarly, in $\Fsc(\Mcal,\tau)$ we define the absolute value by $|A|(k)=|A(k)|$.

\begin{thm}\label{isomorphism thm}
The canonical quotient map $\pi: l_{\infty}(\Mcal)\to\Mcal_\omega$
(see~\S\ref{subsec:ultraP}) extends
to an order preserving $*$-homomorphism $\pit:\Fsc(\Mcal,\tau)\to\Lc_{\log}(\Mcal_\omega,\tau_\omega)$.
\end{thm}
\begin{proof}
We will construct $\pit$ as follows:
given $A\in\Fsc(\Mcal,\tau)$, we find a sequence $\{A_n\}$, each $A_n=\{A_n(k)\}_{k\ge0}\in l_{\infty}(\Mcal)$, such that
\begin{equation}\label{primary one}
\lim_{n\to\infty}\sup_{k\geq0}\|A_n(k)-A(k)\|_{\log}=0.
\end{equation}
Then we set
$$
\tilde\pi(A)=\lim_{n\to\infty}\pi(A_n),
$$
where the limit is taken in $\Lc_{\log}(\Mcal_\omega,\tau_\omega)$ with respect to $\|\cdot\|_{\log}$.

To see the correctness of the definition above, we have to show that
\begin{enumerate}[{\rm (i)}]
\item\label{it:Aexists}
For every $A\in\Fsc(\Mcal,\tau),$ there exists a sequence $\{A_n\}_{n\geq0}\subset l_{\infty}(\Mcal)$ satisfying \eqref{primary one}.
\item\label{it:pitAExists}
If $A\in\Fsc(\Mcal,\tau)$ and if a sequence $\{A_n\}_{n\geq0}\subset l_{\infty}(\Mcal)$ satisfies \eqref{primary one},
then the sequence $\{\pi(A_n)\}_{n\ge0}$ converges with respect to $\|\cdot\|_{\log}$.
\item\label{it:pitAunique}
If $A\in\Fsc(\Mcal,\tau)$ and if sequences $\{A_n'\}_{n\geq0},\{A_n''\}_{n\geq0}\subset l_{\infty}(\Mcal)$ both satisfy \eqref{primary one},
then $\lim_{n\to\infty}\|\pi(A'_n)-\pi(A''_n)\|_{\log}=0$.
\end{enumerate}

We now proceed with the proof of correctness.
To show~\eqref{it:Aexists}, given $A\in\Fsc(\Mcal,\tau),$ set
\begin{equation}\label{eq:Ank}
A_n(k)=A(k)E_{|A(k)|}[0,n].
\end{equation}
Then $A(k)-A_n(k)=A(k)E_{|A(k)|}(n,\infty)$ and, by Lemma~\ref{lem:Fchar},
\eqref{primary one} holds.

To show~\eqref{it:pitAExists},
let $A\in\Fsc(\Mcal,\tau)$ and suppose that a sequence $\{A_n\}_{n\geq0}\subset l_{\infty}(\Mcal)$ satisfies \eqref{primary one}.
By Theorem~\ref{thm:LlogTopalg},
it will suffice to show that the sequence $\{\pi(A_n)\}_{n\ge0}$ is Cauchy with respect to $\|\cdot\|_{\log}$.
Using the triangle inequality in $\Lc_{\log}$, we have
\begin{multline}\label{eq:AnAmbound}
\|\pi(A_n)-\pi(A_m)\|_{\log}=\tau_\omega(\log(1+|\pi(A_n-A_m)|)) \\
=\lim_{k\to\omega}\tau(\log(1+|A_n(k)-A_m(k)|))
=\lim_{k\to\omega}\|A_n(k)-A_m(k)\|_{\log} \\
\le\lim_{k\to\omega}\|A_n(k)-A(k)\|_{\log}+\lim_{k\to\omega}\|A_m(k)-A(k)\|_{\log} \\
\le\sup_{k\ge0}\|A_n(k)-A(k)\|_{\log}+\sup_{k\ge0}\|A_m(k)-A(k)\|_{\log}
\end{multline}
and, by hypothesis, the right-most upper bound tends to $0$ as $n,m\to\infty$.

To show~\eqref{it:pitAunique},
let $A\in\Fsc(\Mcal,\tau)$ and suppose that sequences $\{A_n'\}_{n\geq0},\{A_n''\}_{n\geq0}\subset l_{\infty}(\Mcal)$
both satisfy \eqref{primary one}.
By an estimate like~\eqref{eq:AnAmbound}, we have
$$
\|\pi(A'_n)-\pi(A''_n)\|_{\log}
\le\sup_{k\ge0}\|A'_n(k)-A(k)\|_{\log}+\sup_{k\ge0}\|A''_n(k)-A(k)\|_{\log}
$$
and, by hypothesis, this upper bound tends to $0$ as $n\to\infty$.
We have finished the demonstration that the mapping $\tilde\pi:\Fsc(\Mcal,\tau)\to\Lc_{\log}(\Mcal_\omega,\tau_\omega)$ is well defined.

To conclude the proof of the theorem, we must establish the following, for arbitrary $A,B\in\Fsc(\Mcal,\tau)$.
\begin{enumerate}[{\rm (i)}]
\setcounter{enumi}{3}
\item\label{it:pitrestricts} $\tilde\pi|_{l_{\infty}(\Mcal)}=\pi$,
\item\label{it:pit+} $\tilde\pi(A+B)=\tilde\pi(A)+\tilde\pi(B)$ and $\tilde\pi(A^*)=\tilde\pi(A)^*$,
\item\label{it:pit*} $\tilde\pi(AB)=\tilde\pi(A)\tilde\pi(B)$,
\item\label{it:pitpos} $\pit(A)\ge0$ if $A\in\Fsc(\Mcal,\tau)$ and $A\ge0$.
\end{enumerate}

To show~\eqref{it:pitrestricts}, if $A\in l_{\infty}(\Mcal)$,
then we take the sequence $\{A_n\}_{n\geq0}\subset l_{\infty}(\Mcal)$ with $A_n=A$ for every $n\geq0$.
Clearly, this sequence satisfies \eqref{primary one} and yields $\pit(A)=\pi(A)$.

To show~\eqref{it:pit+} Let $A,B\in\Fsc(\Mcal,\tau)$ and suppose that sequences $\{A_n\}_{n\geq0},\{B_n\}_{n\geq0}\subset l_{\infty}(\Mcal)$ satisfy \eqref{primary one} for $A$ and for $B,$ respectively.
It follows from the triangle inequality in $\Lc_{\log}$
that the sequence $\{A_n+B_n\}_{n\geq0}\subset l_{\infty}(\Mcal)$ satisfies \eqref{primary one} for $A+B$.
Hence, using Theorem~\ref{thm:LlogTopalg}, we have
$$\tilde\pi(A+B)=\lim_{n\to\infty}\pi(A_n+B_n)=\lim_{n\to\infty}\pi(A_n)+\pi(B_n)=\tilde\pi(A)+\tilde\pi(B).$$
The equality $\tilde\pi(A^*)=\tilde\pi(A)$ is immediate.

To show~\eqref{it:pit*}, let $A,B\in\Fsc(\Mcal,\tau)$ and let sequences $\{A_n\}_{n\geq0},\{B_n\}_{n\geq0}\subset l_{\infty}(\Mcal)$
satisfy \eqref{primary one} for $A$ and for $B$, respectively.
We will show that the sequence $\{A_nB_n\}_{n\ge0}$ satisfies \eqref{primary one} for $AB$,
and then Theorem~\ref{thm:LlogTopalg} will imply
$$\tilde\pi(AB)=\lim_{n\to\infty}\pi(A_nB_n)=\lim_{n\to\infty}\pi(A_n)\pi(B_n)=\tilde\pi(A)\tilde\pi(B).$$

Choose numbers $R_n,S_n>1$ such that $\lim_{n\to\infty}R_n=\lim_{n\to\infty}S_n=+\infty$ and
\begin{equation}\label{eq:RnA}
\lim_{n\to\infty}R_n\sup_{k\ge0}\|A(k)-A_n(k)\|_{\log}=\lim_{n\to\infty}S_n\sup_{k\ge0}\|B(k)-B_n(k)\|_{\log}=0.
\end{equation}
Using 
\begin{multline*}
A_n(k)B_n(k)-A(k)B(k)=(A_n(k)-A(k))(B_n(k)-B(k)) \\
+(A_n(k)-A(k))B(k)+A(k)(B_n(k)-B(k))
\end{multline*}
and Lemma~\ref{lem:elemineq}\eqref{it:elemMult}, as well as multiplying and dividing by the numbers $R_n$ and $S_n$, we have
\begin{align*}
\|A_n(k)&B_n(k)-A(k)B(k)\|_{\log} \\
&\le \|(A_n(k)-A(k))(B_n(k)-B(k))\|_{\log}+\|(A_n(k)-A(k))B(k)\|_{\log} \\
&\quad+\|A(k)(B_n(k)-B(k))\|_{\log} \\
&\le \|A_n(k)-A(k)\|_{\log}+\|B_n(k)-B(k)\|_{\log}+R_n\|A_n(k)-A(k)\|_{\log} \\
&\quad+\left\|\frac{B(k)}{R_n}\right\|_{\log}+\left\|\frac{A(k)}{S_n}\right\|_{\log}+S_n\|B_n(k)-B(k)\|_{\log}
\end{align*}
Since $A,B\in\Fsc(\Mcal,\tau)$ and $R_n,S_n\to\infty$, we have
$$
\lim_{n\to\infty}\sup_{k\ge0}\left\|\frac{B(k)}{R_n}\right\|_{\log}=\lim_{n\to\infty}\sup_{k\ge0}\left\|\frac{A(k)}{S_n}\right\|_{\log}=0.
$$
Since $R_n$ and $S_n$ were chosen so that~\eqref{eq:RnA} holds,
we obtain
$$
\lim_{n\to\infty}\sup_{k\ge0}\|A_n(k)B_n(k)-A(k)B(k)\|_{\log}=0,
$$
as required.

Finally, for~\eqref{it:pitpos}, the meaning of $A\ge0$ is $A(k)\ge0$ for all $k$.
Now the construction~\eqref{eq:Ank} of $A_n$ in the proof of~\eqref{it:Aexists}, above,
yields $A_n\ge0$, so we have $\pi(A_n)\ge0$ for all $n$.
Since the set of positive elements is closed in $\Lc_{\log}(\Mcal,\tau)$
(see of Remark~4.8 of~\cite{DSZ:Llog})
taking the limit as $n\to\infty$ yields $\pit(A)\ge0$.
\end{proof}

The map $\sigmat=\pit\circ\delta_\Fsc$ described below will be called
the {\em diagonal embedding} of $\Lc_{\log}(\Mcal,\tau)$ into $\Lc_{\log}(\Mcal_\omega,\tau_\omega)$.
\begin{lem}\label{lem:diagemb}
The map
$\delta_\Fsc:\Lc_{\log}(\Mcal,\tau)\to\Fsc(\Mcal,\tau)$
given by $\delta_\Fsc(T)=(T,T,\ldots)$
is a $*$-algebra morphism.
Consider the usual diagonal embedding $\sigma:=\pi\circ\delta_\Mcal:\Mcal\to\Mcal_\omega$ (see \S\ref{subsec:ultraP})
and its extension
$$
\sigmat:\Lc_{\log}(\Mcal,\tau)\to\Lc_{\log}(\Mcal_\omega,\tau_\omega)
$$
(see~\eqref{eq:sigmat} and the discussion near it).
Then $\sigmat=\pit\circ\delta_\Fsc$. 
\end{lem}
\begin{proof}
Using Lemma~\ref{lem:Fchar} and the Dominated Convergence Theorem, it is easy to see that $\delta_\Fsc$ maps
$\Lc_{\log}(\Mcal,\tau)$ into $\Fsc(\Mcal,\tau)$.
It is clearly a $*$-algebra homomorphism.

To show $\sigmat=\pit\circ\delta_\Fsc$, note that  both sides are tautologically equal to $\sigma$ when restricted
to $\Mcal$.
For $T\in\Lc_{\log}(\Mcal,\tau)$, by the construction in Theorem~\ref{isomorphism thm}, we have
$$
\pit\circ\delta_{\Fsc}(T)=\lim_{n\to\infty}\pi\circ\delta_\Mcal(TE_{|T|}[0,n])=\lim_{n\to\infty}\sigma(TE_{|T|}[0,n]),
$$
where the convergence is in $\|\cdot\|_{\log}$, hence, also in measure topology. 
However, $\pi\circ\delta_\Mcal(TE_{|T|}[0,n])=\sigma(TE_{|T|}[0,n])$ and $\lim_{n\to\infty}TE_{|T|}[0,n]=T$,
with convergence in measure.
Since $\sigmat$ is continuous with respect to convergence in measure, we have
$\lim_{n\to\infty}\sigma(TE_{|T|}[0,n])=\sigmat(T)$.
This finishes the proof.
\end{proof}

The next lemma belongs in this section but will first be used in Section~\ref{hs construct}.

\begin{lem}\label{lem:Yt}
Let $Y\in\Lc_{\log}(\Mcal,\tau)$, let $r_k>0$, suppose $\lim_{k\to\infty}r_k=0$ and let $\Yt=\{r_kY\}_{k\ge0}$.
Then $\Yt\in\Fsc(\Mcal,\tau)$ and $\pit(\Yt)=0$.
\end{lem}
\begin{proof}
Let $r=(r_1I,r_2I,\ldots)\in l_\infty(\Mcal)$, where $I$ is the identity in $\Mcal$.
Then $\Yt=r\delta_\Fsc(Y)\in\Fsc(\Mcal,\tau)$ (see Lemma \ref{lem:diagemb}).
Since $\pi(r)=0$, we have $\pit(\Yt)=\pi(r)\pit(\delta_\Fsc(Y))=0$.
\end{proof}

\begin{thm}\label{det converge} If $A=\{A(k)\}_{k\geq0}\in\Fsc(\Mcal,\tau)$ is such that $A\geq 1,$ then
$$\Delta(\pit(A))=\lim_{k\to\omega}\Delta(A(k)).$$
\end{thm}
\begin{proof}
Suppose first that $A\in l_{\infty}(\Mcal)$. It follows that $\{\log(|A(k)|)\}_{k\geq0}\subset\Mcal$.
Hence,
\begin{multline*}
\log(\Delta(\pit(A)))=\log(\Delta(\pi(A)))=\tau_\omega(\log(\pi(A))) \\
=\lim_{k\to\omega}\tau(\log(A(k)))=\lim_{k\to\omega}\log(\Delta(A(k))).
\end{multline*}
This concludes the proof for $A\in l_{\infty}(\Mcal)$.

Consider now the general case.
Fix $n\in\mathbb{N}$ and let $\{B_n\}_{n\ge1}$ and $\{C_n\}_{n\ge1}$ in $\Fsc(\Mcal,\tau)$ be given by
\begin{align*}
B_n(k)&=\min\{A(k),n\}:=A(k)E_{A(k)}[1,n]+nE_{A(k)}(n,\infty) \\
C_n(k)&=\max\{\frac1nA(k),1\}:=E_{A(k)}[1,n]+\frac1nA(k)E_{A(k)}(n,\infty),
\end{align*}
where we note that the hypothesis $A\ge1$ implies $A(k)\ge1$ for each $k$.
For each $n$, we have
\begin{equation}\label{eq:ABC}
A(k)=B_n(k)C_n(k).
\end{equation}
Since $\|B_n(k)\|\le n$ remains bounded as $k\to\infty$ and $1\le B_n(k)\le A(k)$, it follows that
\begin{equation}\label{eq:DeltaB}
\lim_{k\to\omega}\Delta(A(k))\geq\lim_{k\to\omega}\Delta(B_n(k))=\Delta(\pi(B_n)).
\end{equation}
In fact, we have
$1\le B_n(k)\le A(k)E_{A(k)}[0,n]$, so we get
$$
\|A(k)-B_n(k)\|_{\log}\le\|A(k)E_{A(k)}(n,\infty)\|_{\log}.
$$
Invoking Lemma~\ref{lem:Fchar}, we have
$$
\lim_{n\to\infty}\sup_{k\ge0}\|A(k)-B_n(k)\|_{\log}=0
$$
and by the definition of $\pit$ in the proof of Theorem~\ref{isomorphism thm}, we have
$\pit(A)=\lim_{n\to\infty}\pi(B_n)$, where convergence is with respect to $\|\cdot\|_{\log}$.
By Remark 4.8 of~\cite{DSZ:Llog}, we have convergence in measure of $\pi(B_n)$ to $\pit(A)$.
Since $B_n\le B_{n+1}$, also $\pit(B_n)\le\pit(B_{n+1})$.
Since $B_n\geq 1$ for all $n$, it follows from Fack and Kosaki's version of the the Monotone Convergence Theorem,
namely,  Theorem~3.5(ii) of~\cite{FackKosaki},
that
\begin{equation}\label{eq:limDelB}
\lim_{n\to\infty}\Delta(\pi(B_n))=\Delta(\pit(A)).
\end{equation}
From this and~\eqref{eq:DeltaB} we infer
$$\lim_{k\to\omega}\Delta(A(k))\geq\Delta(\pit(A)).$$

We now prove the reverse inequality.
Evidently,
$$\log(C_n(k))=\log^+\left(\frac1nA(k)\right)\le\log\left(1+\frac{A(k)}n\right),\quad(n,k\geq0),$$
so
$$
\log\Delta(C_n(k))=\tau(\log(C_n(k))\le\left\|\frac{A(k)}n\right\|_{\log}.
$$
Since $A\in\Fsc(\Mcal,\tau)$, we conclude
$$
\limsup_{n\to\infty}\left(\sup_{k\ge0}\Delta(C_n(k))\right)\le1.
$$
Using~\eqref{eq:ABC} and~\eqref{eq:limDelB}, we have
\begin{multline*}
\lim_{k\to\omega}\Delta(A(k))=\lim_{k\to\omega}\Delta(B_n(k))\cdot\lim_{k\to\omega}\Delta(C_n(k))\leq\Delta(\pi(B_n))\cdot\sup_{k\geq0}\Delta(C_n(k)) \\
\leq\Delta(\pit(A))\cdot\sup_{k\geq0}\Delta(C_n(k)).
\end{multline*}
Taking $n\to\infty$ yields
$\lim_{k\to\omega}\Delta(A(k))\le \Delta(\pit(A))$, as required.
\end{proof}

The following result states that the Fuglede-Kadison determinant is preserved by the diagonal embedding
$\pit\circ\delta_\Fsc$ from Lemma~\ref{lem:diagemb}.
It is, in fact, easy to prove directly, too.

\begin{cor}\label{cor:Yt}
Let $Y\in\Lc_{\log}(\Mcal,\tau)$.
Then $\Delta(\pit\circ\delta_\Fsc(Y))=\Delta(Y)$.
\end{cor}
\begin{proof}
Using~\eqref{eq:Deltaeps} and Theorem~\ref{det converge}, we have
\begin{multline*}
\Delta(\pit\circ\delta_\Fsc(Y))^2=\lim_{\eps\to0}\Delta(|\pit\circ\delta_\Fsc(Y)|^2+\eps)
=\lim_{\eps\to0}\Delta(\pit\circ\delta_\Fsc(|Y|^2)+\eps) \\
=\lim_{\eps\to0}\Delta(|Y|^2+\eps)=\Delta(Y)^2.
\end{multline*}
\end{proof}

\section{Properties of Haagerup--Schultz projections for mild operators}
\label{mild section}

For mild operators with appropriate Lipschitz domains, Haagerup--Schultz projections corresponding to disks are given by Theorem \ref{hs integral}.
In this section, we construct Haagerup--Schultz projections for mild operators and certain closed sets
(namely, those closed sets $\mathcal{B}$ with $m(\partial\mathcal{B})=0$, where $m$ is Lebesgue measure).
Note that for a mild operator $T$, since $\nu_T$ is absolutely continuous with respect to Lebesgue measure, a Haagerup-Schultz projection for a set $\Bc$
remains a Haagerup-Schultz projection for any set that differs from $\Bc$ by a Lebesgue-null set.

We recall that mild operators $T$ satisfy $\Delta(T)>0$ and, in particular, must have zero kernel and, thus, be invertible in $\Sc(\Mcal,\tau)$.
In fact, by Lemma 2.4 of~\cite{HS1}, the inverse lies in $\Lc_{\log}(\Mcal,\tau)$.

\begin{lem}\label{lem:Moebius mild}
Suppose $T\in\Lc_{\log}(\Mcal,\tau)$ is mild$(\Dc)$ for some $\Dc\subseteq\Cpx$
and let $\gamma$ be a fractional linear transformation.
Then $\gamma(T)$ is mild$(\gamma(\Dc)\cap\Cpx)$.
\end{lem}
\begin{proof}
It will suffice to prove it in the two cases (i) $\gamma(z)=az+b$, $a\ne0$ and (ii) $\gamma(z)=z^{-1}$,
since each fractional linear transformation is a composition of some of these forms.
In case~(i), the conclusions follow immediately.
In case~(ii),
$\Delta(T^{-1})=\Delta(T)^{-1}\ne0$ and, for $\lambda\ne0$,
$$
\Delta(T^{-1}-\lambda)=\Delta(T)^{-1}|\lambda|\Delta(T-\lambda^{-1})\ne0,
$$
so part~\eqref{it:milda} of Definition~\ref{def:mild} holds for $T^{-1}$.
Since $T^{-1}\in\Lc_{\log}(\Mcal,\tau)$, we have
(see Proposition 2.16 in \cite{HS1}) that $\nu_{T^{-1}}$ is the push-forward measure of $\nu_T$ under $z\mapsto z^{-1}$.
Thus, $\nu_{T^{-1}}$ is Lebesgue absolutely continuous and part~\eqref{it:mildb} of Definition~\ref{def:mild} holds.
Let $p\in(\frac12,1)$ be such that the map $z\mapsto(z-T)^{-1}$ is locally Lipschitz from $\Dc$ into $\Lc_p(\Mcal,\tau)$.
Now from the identity
$$
(z-T^{-1})^{-1}=z^{-1}-z^{-2}(z^{-1}-T)^{-1},
$$
we see that the map $z\mapsto(z-T^{-1})^{-1}$ is locally Lipschitz from $\{z^{-1}\mid z\in\Dc\backslash\{0\}\}=\gamma(\Dc)\cap\Cpx$ to $\Lc_p(\Mcal,\tau)$.
This finishes the proof that $T^{-1}$ is mild$(\gamma(\Dc)\cap\Cpx)$.
\end{proof}

We will need the following elementary lemma.
For a projection $p\in\Mcal$, we make the obvious identification
$$ p\Sc(\Mcal,\tau)p=\Sc(p\Mcal p,\tau(p)^{-1}\tau\restrict_\Mcal).$$

\begin{lem}\label{lem:UTinvert}
Let $R\in\Sc(\Mcal,\tau)$, suppose $p\in\Mcal$ is an $R$-invariant projection
and suppose $R$ is invertible.
Then $p$ is $R^{-1}$-invariant and $R^{-1}p=(Rp)^{-1}$, where the latter inverse is in
$p\Sc(\Mcal,\tau)p$.
Similarly, $(1-p)R^{-1}=((1-p)R(1-p))^{-1}$.
\end{lem}
\begin{proof}
We write elements $T$ of $\Sc(\Mcal,\tau)$ as $2\times2$ matrices
$$T=\left(\begin{matrix} pTp& (1-p)Tp \\ pT(1-p) & (1-p)T(1-p)\end{matrix}\right).$$
Since $p$ is $R$-invariant, we have
$$R^{-1}=\left(\begin{matrix}x&y\\w&z\end{matrix}\right),\qquad R=\left(\begin{matrix}a&b\\0&c\end{matrix}\right).$$
From $R^{-1}R=\left(\begin{smallmatrix}p&0\\0&1-p\end{smallmatrix}\right)$, and using that left-invertible elements
are invertible, we find $x=a^{-1}$, $w=0$ and $z=c^{-1}$.
\end{proof}

\begin{lem}\label{inverse lemma}
Let $\mathcal{B}$ be the complement of the unit disk centered at $0$.
Let $T\in\Lc_{\log}(\Mcal,\tau)$ be a mild$(\partial B)$ operator.
Suppose $p$ is a Haagerup--Schultz projection for $T^{-1}$ corresponding to the set $\Cpx\backslash\Bc$.
Then $p$ is a Haagerup--Schultz projection for $T$ corresponding to the set $\Bc$.
\end{lem}
\begin{proof}
Since $T^{-1}p=pT^{-1}p$,
by Lemma~\ref{lem:UTinvert}
we have $Tp=pTp$ and that
the inverse of $Tp$ in the algebra $p\Mcal p$ equals $T^{-1}p$.
For every $\lambda\neq0,$ it follows from \eqref{det homomorphism} that
$$\Delta_{p\Mcal p}(Tp-\lambda p)=\Delta_{p\Mcal p}(\lambda Tp)\Delta_{p\Mcal p}(T^{-1}p-\lambda^{-1}p).$$
Therefore,
$$\log(\Delta_{p\Mcal p}(Tp-\lambda p))=\log(\Delta_{p\Mcal p}(\lambda Tp))+\int_{\mathbb{C}}\log(|\lambda^{-1}-z|)d\nu_{T^{-1}p}(z).$$
Since $p$ is a Haagerup--Schultz projection for $T^{-1}$ corresponding to $\Cpx\backslash\Bc$, we have
$$\nu_{T^{-1}p}=\frac1{\nu_{T^{-1}}(\Cpx\backslash\Bc)}\nu_{T^{-1}}|_{\Cpx\backslash\Bc}.$$
In addition to that, we have $d\nu_{T^{-1}}(z)=d\nu_T(z^{-1})$ and, hence, $\nu_{T^{-1}}(\Cpx\backslash\Bc)=\nu_T(\Bc).$
It follows that
\begin{align*}
\log(&\Delta_{p\Mcal p}(Tp-\lambda p))=\log(\Delta_{p\Mcal p}(\lambda Tp))+\frac1{\nu_T(\Bc)}\int_{\Cpx\backslash\Bc}\log(|\lambda^{-1}-z|)d\nu_T(z^{-1}) \\
&=\log(\Delta_{p\Mcal p}(\lambda Tp))+\frac1{\nu_T(\Bc)}\int_{\Bc}\log(|\lambda^{-1}-w^{-1}|)d\nu_T(w) \\
&=\log(|\lambda|)+\log(\Delta_{p\Mcal p}(Tp))+\frac1{\nu_T(\Bc)}\int_{\Bc}\log(|\lambda-w|)d\nu_T(w) \\
&\quad-\frac1{\nu_T(\Bc)}\int_{\Bc}\log(|\lambda|)d\nu_T(w)-\frac1{\nu_T(\Bc)}\int_{\Bc}\log(|w|)d\nu_T(w) \\
&=\frac1{\nu_T(\Bc)}\int_{\Bc}\log(|\lambda-w|)d\nu_T(w) \\
&\quad+\log(\Delta_{p\Mcal p}(Tp))-\frac1{\nu_T(\Bc)}\int_{\Bc}\log(|w|)d\nu_T(w)
\end{align*}
and we infer that
$$
\log(\Delta_{p\Mcal p}(Tp-\lambda p))-\frac1{\nu_T(\Bc)}\int_{\Bc}\log(|\lambda-w|)d\nu_T(w)
$$
is constant as $\lambda$ varies.
Taking the Laplacian, the constant disappears and we get
$$\nu_{Tp}=\frac1{\nu_T(\Bc)}\nu_T|_{\Bc}.$$
Now using Theorem~\ref{det nu matrix}, we get
$$\nu_{(1-p)T}=\frac1{\nu_T(\Cpx\backslash\Bc)}\nu_T|_{\Cpx\backslash\mathcal{B}}.$$
This completes the proof.
\end{proof}

We are now ready to define some specific Haagerup--Schultz projections $P(T,\Bc)$.
In the next remark, definition and in the proof of Lemma~\ref{conformal} below, by an {\em affine transformation} of $\Cpx$ we will mean
a transformation of the form $\gamma(z)=\alpha z+z_0$, for some $\alpha>0$ and $z_0\in\Cpx$.

\begin{rem}
Suppose $T\in\Lc_{\log}(\Mcal,\tau)$ and $p$ is a Haagerup-Schultz projection for $T$ and a Borel subset $\Bc\subseteq\Cpx$.
Let $\gamma$ be an affine transformation.
Then, as is apparent from the transformation rules for Brown measure under affine transformations,
$p$ is a Haagerup-Schultz projection for $\gamma(T)$ and $\gamma(\Bc)$.
\end{rem}

\begin{defi}\label{def:someHSproj}
Let $\Bc\subseteq\Cpx$ be either a disk or the complement of a disk.
Let $T\in\Lc_{\log}(\Mcal,\tau)$ be a mild$(\partial\Bc)$ operator.
We let $P(T,\Bc)$ be the Haagerup--Schultz projection for $T$ and $\Bc$ defined as follows:
\begin{enumerate}[{\rm (a)}]
\item\label{it:someHSproj-a} if $\Bc$ is a disk in $\Cpx$, then $P(T,\Bc)$ is the projection constructed by Haagerup and Schultz  in the proof of Theorem~\ref{hs integral};
\item\label{it:someHSproj-b} if $\Bc$ is the complement of the unit disk in $\Cpx$,
$P(T,\Bc)=P(T^{-1},\Cpx\backslash\Bc)$ (see Lemmas~\ref{lem:Moebius mild} and~\ref{inverse lemma});
\item\label{it:someHSproj-c} if $\Bc$ is the complement of any disk in $\Cpx$,
then we let $P(T,\Bc)=P(\gamma(T),\gamma(\Bc))$, where $\gamma$ is the affine transformation that maps $\Bc$ onto the complement of the unit disk in $\Cpx$.
\end{enumerate}
\end{defi}

\begin{lem}\label{conformal}
Let $\Bc$ be a disk in $\Cpx$ and let $T\in\Lc_{\log}(\Mcal,\tau)$ be a mild$(\partial\Bc)$ operator.
Then for every fractional linear transform $\gamma$ with $\infty\notin\gamma(\partial\mathcal{B}),$ we have
\begin{equation}\label{eq:Pgamma}
P(T,\mathcal{B})=P(\gamma(T),\gamma(\mathcal{B})).
\end{equation}
\end{lem}
\begin{proof}
If $\gamma$ is an affine transformation,
then it is clear from Haagerup and Schultz's
construction in Theorem~\ref{hs integral} that~\eqref{eq:Pgamma} holds.

Let $\mathcal{B}_0$ be the unit disk centered at $0$.
If $\gamma(\mathcal{B})$ is a disk, then $\gamma=\gamma_0\circ\gamma_1\circ\gamma_2$,
where $\gamma_2:\mathcal{B}\to\mathcal{B}_0$ and $\gamma_0:\mathcal{B}_0\to\gamma(\mathcal{B})$ are affine transformations
and $\gamma_1:\mathcal{B}_0\to\mathcal{B}_0$ is an affine transformation of the disk $\mathcal{B}_0$ onto itself.
We use Lemma~\ref{lem:Moebius mild} to see that $\gamma_1\circ\gamma_2(T)$ is mild$(\partial\Bc_0)$, {\em etc}.
Below, we will show that the lemma holds when $\gamma=\gamma_1$ and $\Bc=\Bc_0$.
Since $\gamma_2$ is affine, it will then follow that the lemma holds when $\Bc$ is a disk.

If $\gamma(\mathcal{B})$ is a complement of a disk, then $\gamma=\gamma'_0\circ\gamma'\circ\gamma_1\circ\gamma_2,$ where
$\gamma_1$ and $\gamma_2$ are as above, where $\gamma':z\mapsto z^{-1}$ and where $\gamma_0'$ is an affine transformation
mapping $\Cpx\backslash\Bc_0$ onto $\gamma(\Bc)$.
By Definition~\ref{def:someHSproj}\eqref{it:someHSproj-b}, the equality~\eqref{eq:Pgamma} holds if $\gamma=\gamma_1$ and $\Bc=\Bc_0$.
By Definition~\ref{def:someHSproj}\eqref{it:someHSproj-c}, the equality~\eqref{eq:Pgamma} holds if $\gamma$ is an affine transformation
and $\Bc$ is the complement of a disk.
Combined with the facts about $\gamma_2$ and $\gamma_1$ mentioned above, this will prove the lemma also in the case that $\gamma(\Bc)$
is the complement of a disk.

It remains to show that~\eqref{eq:Pgamma} holds when $\gamma=\gamma_1$ is of the form
$\gamma(z)=e^{i\theta}\frac{z-a}{1-\bar{a}z}$ with $|a|<1$, $\theta\in\Reals$, and when $\Bc=\Bc_0$.
By Theorem 6.2 in \cite{HS2}
and the change of variables formula (Lemma~\ref{lem:RIcov}),
we have
\begin{align*}
\int_{|z|=1}\frac{dz}{z-\gamma(T)}&=\int_{|w|=1}\frac{\gamma'(w)dw}{\gamma(w)-\gamma(T)} \\
&=\int_{|w|=1}\frac{(1-|a|^2)dw}{(1-\bar{a}w)^2}\cdot \frac{(1-\bar{a}T)(1-\bar{a}w)}{(w-a)(1-\bar{a}T)-(T-a)(1-\bar{a}w)} \\
&=\int_{|w|=1}\frac{dw}{1-\bar{a}w}\cdot\frac{1-\bar{a}T}{w-T} \\
&=\int_{|w|=1}\frac{dw}{1-\bar{a}w}\cdot\frac{(1-\bar{a}w)+\bar{a}(w-T)}{w-T} \\
&=\int_{|w|=1}\frac{dw}{w-T}+\bar{a}\int_{|w|=1}\frac{dw}{1-\bar{a}w}=\int_{|w|=1}\frac{dw}{w-T}.
\end{align*}
This implies that the idempotents $E(T,\Bc)$ and $E(\gamma(T),\Bc)$ used by Haagerup and Schultz (see Equation~\eqref{idempotent def})
agree, and, therefore, the projections $P(T,\Bc)$ and $P(\gamma(T),\Bc)$ onto their images are the same.
\end{proof}

\begin{lem}
\label{first order lemma}
Let $\mathcal{B}_1$ and $\mathcal{B}_2$ be disks in $\Cpx$.
Let $T\in\Lc_{\log}(\Mcal,\tau)$ be a mild$(\partial \Bc_1\cup\partial\Bc_2)$ operator.
If $\overline{\mathcal{B}_1}\subset{\rm int}(\mathcal{B}_2)$, then $P(T,\mathcal{B}_1)\leq P(T,\mathcal{B}_2)$.
\end{lem}
\begin{proof} Applying a fractional-linear transform and using Lemmas~\ref{lem:Moebius mild} and~\ref{conformal},
we may assume without loss of generality that, for some $r<1$,
$$
\mathcal{B}_1=\{z\mid |z|\leq r\},\qquad\mathcal{B}_2=\{z\mid |z|\leq 1\}.
$$

Set
$$E_n=\frac1n\sum_{k=1}^n\frac{re^{2\pi ik/n}}{re^{2\pi ik/n}-T},\quad F_n=\frac1n\sum_{l=1}^n\frac{e^{2\pi il/n}}{e^{2\pi il/n}-T}$$
to be the Riemann sums for the idempotents $E(T,\mathcal{B}_1)$ and $E(T,\mathcal{B}_2),$ respectively.
By Theorem 6.2 in \cite{HS2} (see also~\cite{TW68}),
we have that $\|E_n-E(T,\mathcal{B}_1)\|_p\to0$ and $\|F_n-E(T,\mathcal{B}_2)\|_p\to0$ as $n\to\infty$.

We claim that
\begin{equation}\label{ef main equality}
E_nF_n=\frac{1}{1-r^n}E_n-\frac{r^n}{1-r^n}F_n=F_nE_n.
\end{equation}

Indeed, we have
$$E_nF_n=\frac1{n^2}\sum_{k,l=1}^n\frac{re^{2\pi i(k+l)/n}}{(re^{2\pi ik/n}-T)(e^{2\pi il/n}-T)}=F_nE_n.$$
It is clear that
$$\frac1{(re^{2\pi ik/n}-T)(e^{2\pi il/n}-T)}=\frac1{e^{2\pi il/n}-re^{2\pi ik/n}}\cdot\Big(\frac1{re^{2\pi ik/n}-T}-\frac1{e^{2\pi il/n}-T}\Big).$$
Therefore,
\begin{align*}
E_nF_n&=\frac1{n^2}\sum_{k,l=1}^n\frac{re^{2\pi i(k+l)/n}}{e^{2\pi il/n}-re^{2\pi ik/n}}\cdot \frac1{re^{2\pi ik/n}-T} \\
&\quad-\frac1{n^2}\sum_{k,l=1}^n\frac{re^{2\pi i(k+l)/n}}{e^{2\pi il/n}-re^{2\pi ik/n}}\cdot \frac1{e^{2\pi il/n}-T} \displaybreak[2]\\
&=\frac1{n^2}\sum_{k=1}^n\frac1{re^{2\pi ik/n}-T}\Big(\sum_{l=1}^n\frac{re^{2\pi i(k+l)/n}}{e^{2\pi il/n}-re^{2\pi ik/n}}\Big) \\
&\quad-\frac1{n^2}\sum_{l=1}^n\frac1{e^{2\pi il/n}-T}\Big(\sum_{k=1}^n\frac{re^{2\pi i(k+l)/n}}{e^{2\pi il/n}-re^{2\pi ik/n}}\Big).
\end{align*}
It is clear that
\begin{align*}
\sum_{l=1}^n\frac{re^{2\pi i(k+l)/n}}{e^{2\pi il/n}-re^{2\pi ik/n}}&=re^{2\pi ik/n}\sum_{m=1}^n\frac{e^{2\pi im/n}}{e^{2\pi im/n}-r}, \\
\sum_{k=1}^n\frac{re^{2\pi i(k+l)/n}}{e^{2\pi il/n}-re^{2\pi ik/n}}&=e^{2\pi il/n}\sum_{m=1}^n\frac{re^{2\pi im/n}}{1-re^{2\pi im/n}}.
\end{align*}
It follows that
$$E_nF_n=\alpha_nE_n-\beta_nF_n.$$
where
$$\alpha_n=\frac1n\sum_{m=1}^n\frac{e^{2\pi im/n}}{e^{2\pi im/n}-r},\quad \beta_n=\frac1n\sum_{m=1}^n\frac{re^{2\pi im/n}}{1-re^{2\pi im/n}}.$$
We have
\begin{align*}
\alpha_n&=\frac1n\sum_{m=1}^n\frac{e^{2\pi im/n}}{e^{2\pi im/n}-r}=\frac1n\sum_{m=1}^n\frac1{1-re^{-2\pi im/n}} \\
&=\frac1n\sum_{m=1}^n\sum_{k=0}^{\infty}r^ke^{-2\pi ikm/n}=\frac1n\sum_{k=0}^{\infty}r^k\left(\sum_{m=1}^ne^{-2\pi ikm/n}\right).
\end{align*}
The internal sum is $n$ when $\frac{k}{n}$ is integer and is $0$ otherwise. Thus,
$$\alpha_n=\frac1{1-r^n}\mbox{ and, similarly, }\beta_n=\frac{r^n}{1-r^n}.$$
This proves the claim.

We have
\begin{align*}
\|E_nF_n&-E(T,\mathcal{B}_1)E(T,\mathcal{B}_2)\|_{p/2} \\
&\leq 2^{(2/p)-1}\Big(\|(E_n-E(T,\mathcal{B}_1))F_n\|_{p/2}+\|E(T,\mathcal{B}_1)(F_n-E(T,\mathcal{B}_2))\|_{p/2}\Big) \\
&\leq2^{(2/p)-1}\Big(\|E_n-E(T,\mathcal{B}_1)\|_p\|F_n\|_p+\|E(T,\mathcal{B}_1)\|_p\|F_n-E(T,\mathcal{B}_2)\|_p\Big),
\end{align*}
and this last upper bound tends to $0$ as $n\to\infty$. 
Similarly, we have
$$
\lim_{n\to\infty}\|F_nE_n-E(T,\mathcal{B}_2)E(T,\mathcal{B}_1)\|_{p/2}=0.
$$
On the other hand, it follows from \eqref{ef main equality} that
\begin{multline*}
\|E_nF_n-E(T,\mathcal{B}_1)\|_{p/2}\leq\|\frac1{1-r^n}E_n-\frac{r^n}{1-r^n}F_n-E(T,\mathcal{B}_1)\|_p \\
\leq2^{(1/p)-1}\Big(\frac1{1-r^n}\|E_n-E(T,\mathcal{B}_1)\|_p+\frac{r^n}{1-r^n}\|F_n-E(T,\mathcal{B}_1)\|_p\Big),
\end{multline*}
and this last upper bound tends to $0$ as $n\to\infty$.
Hence,
$$E(T,\mathcal{B}_1)E(T,\mathcal{B}_2)=E(T,\mathcal{B}_1)=E(T,\mathcal{B}_2)E(T,\mathcal{B}_1).$$

If $E_1$ and $E_2$ are two commuting idempotents such that $E_1E_2=E_1,$ then ${\rm range}(E_1)\leq{\rm range}(E_2).$ Applying this to the idempotents $E(T,\mathcal{B}_1)$ and $E(T,\mathcal{B}_2),$ we conclude $P(T,\mathcal{B}_1)\leq P(T,\mathcal{B}_2).$
\end{proof}

\begin{lem}\label{second order lemma}
Let $\mathcal{B}_1$ and $\mathcal{B}_2$ be disks in $\Cpx$.
Let $U$ be an open neighborhood of $\partial\Bc_1\cup\partial\Bc_2$ and
suppose $T\in\Lc_{\log}(\Mcal,\tau)$ is a mild$(U)$ operator.
\begin{enumerate}[{\rm (a)}]
\item If $\mathcal{B}_1\subset\mathcal{B}_2$, then $P(T,\mathcal{B}_1)\leq P(T,\mathcal{B}_2)$.
\item If $\mathcal{B}_1\subset\mathbb{C}\backslash\mathcal{B}_2$, then $P(T,\mathcal{B}_1)\leq P(T,\mathbb{C}\backslash\mathcal{B}_2)$.
\item If $\mathbb{C}\backslash\mathcal{B}_1\subset\mathbb{C}\backslash\mathcal{B}_2$,
then $P(T,\mathbb{C}\backslash\mathcal{B}_1)\leq P(T,\mathbb{C}\backslash\mathcal{B}_2)$.
\end{enumerate}
\end{lem}
\begin{proof} We show the first assertion.
Let $\mathcal{B}_1=B(\lambda,r)$ be the disk of radius $r$ centered at $\lambda$.
Let $k_0\in\Nats$ be such that $\partial B(\lambda,\frac{kr}{k+1})\subseteq U$ for all $k\ge k_0$.
By Lemma \ref{first order lemma} for $k\ge k_0$, we have $P(T,B(\lambda,\frac{kr}{k+1}))\leq P(T,\mathcal{B}_2)$.
Therefore,
$$\bigvee_{k\geq k_0}P(T,B(\lambda,\frac{kr}{k+1}))\leq P(T,\mathcal{B}_2).$$
By Lemma \ref{first order lemma}, the sequence of projections $\{P(T,B(\lambda,\frac{kr}{k+1}))\}_{k\geq k_0}$
is increasing and is bounded from above by $P(T,\mathcal{B}_1)$.
It follows (see Definition~\ref{def:hs proj}\eqref{hdefb}) that
$$\lim_{k\to\infty}\tau(P(T,B(\lambda,\frac{kr}{k+1})))=\lim_{k\to\infty}\nu_T(B(\lambda,\frac{kr}{k+1}))=\nu_T(\mathcal{B}_1)=\tau(P(T,\mathcal{B}_1)).$$
Therefore,
$$P(T,\mathcal{B}_1)=\bigvee_{k\geq0}P(T,B(\lambda,\frac{kr}{k+1}))\leq P(T,\mathcal{B}_2).$$
This proves the first assertion.

The second and third assertions follow by combining Lemma \ref{first order lemma} and Lemma \ref{conformal}.
\end{proof}

It is clear from Definition~\ref{def:mild} and the properties of the Fuglede-Kadison determinant and Brown measure, that if $T$ is a mild$(\Bc)$ operator
for a subset $\Bc$ of $\Cpx$, then $T^*$ is a mild$(\Cc)$ operator, where $\Cc=\{z\mid \bar{z}\in\mathcal{B}\}$.

\begin{lem}\label{haag adjoint}
Let $\Bc$ be a disk in $\Cpx$ and suppose
$T\in\Lc_{\log}(\Mcal,\tau)$ is a mild$(\partial\Bc)$ operator.
Let $\mathcal{C}=\{z\mid \bar{z}\in\mathcal{B}\}$.
Then $P(T,\mathbb{C}\backslash\mathcal{B})=1-P(T^*,\mathcal{C})$.
\end{lem}
\begin{proof} Applying Lemma \ref{conformal}, we may assume without loss of generality that $\mathcal{B}$
(and, thus, also $\mathcal{C}$) is the unit disk centered at $0$.
Recall: $P(T,\Cpx\backslash\Bc)=P(T^{-1},\Bc)$.

By Theorem 6.2 in \cite{HS2}
and the change of variables formula (Lemma~\ref{lem:RIcov}),
using the substitution $z=w^{-1},$ we infer that
\begin{align*}
E(T^{-1},\mathcal{B})&=\frac1{2\pi i}\int_{|z|=1}\frac{dz}{z-T^{-1}}=\frac1{2\pi i}\int_{|w|=1}\frac{dw}{w^2(w^{-1}-T^{-1})} \\
&=\frac1{2\pi i}\int_{|w|=1}\Big(\frac1w-\frac1{w-T}\Big)dw=1-E(T,\mathcal{B}).
\end{align*}
Also,
$$E(T,\mathcal{B})^*=-\frac1{2\pi i}\int_{|z|=1}\frac{d\bar{z}}{\bar{z}-T^*}=\frac1{2\pi i}\int_{|w|=1}\frac{dw}{w-T^*}=E(T^*,\mathcal{B}).$$

If $E_1$ and $E_2$ are idempotents such that $E_1=1-E_2^*,$ then ${\rm range}(E_1)={\rm range}(E_2)^{\perp}.$ Applying this to idempotents $E_1=E(T^{-1},\mathcal{B})$ and $E_2=E(T^*,\mathcal{B}),$ we conclude the proof.
\end{proof}

\begin{defi}\label{haag mild def}
Let $T\in\Lc_{\log}(\Mcal,\tau)$ be a mild operator.
Let $\mathcal{B}$ be a closed set.
If $\{\lambda_k\}_{k\geq0}$ is a dense subset in ${\rm int}(\mathcal{B})$
(respectively, $\{\lambda_k'\}_{k\geq0}$ is a dense subset in $\mathbb{C}\backslash\mathcal{B}$), then set
\begin{align*}
P_{(\downarrow)}(T,\mathcal{B})&=\bigvee_{k\geq0}P(T,B(\lambda_k,{\rm dist}(\lambda_k,\partial\mathcal{B}))), \\
P^{(\uparrow)}(T,\mathcal{B})&=\bigwedge_{k\geq0}P(T,\mathbb{C}\backslash B(\lambda'_k,{\rm dist}(\lambda_k',\partial\mathcal{B}))),
\end{align*}
where this means $P_{(\downarrow)}(T,\mathcal{B})=0$ if ${\rm int}(\mathcal{B})$ is empty, and
$P^{(\uparrow)}(T,\mathcal{B})=1$ if $\Bc=\Cpx$.
Note that $P_{(\downarrow)}(T,\mathcal{B})$ and $P^{(\uparrow)}(T,\mathcal{B})$ are $T$-invariant.
\end{defi}

\begin{lem}
Let $T\in\Lc_{\log}(\Mcal,\tau)$ be a mild operator.
The projections $P_{(\downarrow)}(T,\mathcal{B})$ and $P^{(\uparrow)}(T,\mathcal{B})$ are well defined (i.e., are independent of the particular choice of dense subsets $\{\lambda_k\}_{k\geq0}$ and $\{\lambda_k'\}_{k\geq0}$).
\end{lem}
\begin{proof} We prove the assertion for $P_{(\downarrow)}(T,\mathcal{B})$ (the proof for $P^{(\uparrow)}(T,\mathcal{B})$ is identical). Suppose $\{\lambda_k\}_{k\geq0}$ and $\{\mu_l\}_{l\geq0}$ are dense subsets in ${\rm int}(\mathcal{B})$.
Consider the corresponding projections
\begin{align*}
P_1&=\bigvee_{k\geq0}P(T,B(\lambda_k,{\rm dist}(\lambda_k,\partial\mathcal{B}))), \\
P_2&=\bigvee_{k\geq0}P(T,B(\mu_k,{\rm dist}(\mu_k,\partial\mathcal{B}))).
\end{align*}
Fix $k$ and $0<q<r_k:={\rm dist}(\lambda_k,\partial\mathcal{B}).$ Since $\{\mu_l\}_{l\geq0}$ is a dense subset of ${\rm int}(\mathcal{B}),$ it follows that there exists $l$ such that $|\lambda_k-\mu_l|<\frac12(r_k-q).$ It follows that
$$\overline{B(\lambda_k,q)}\subset B(\mu_l,\frac{r_k+q}{2})\subset B(\mu_l,{\rm dist}(\mu_l,\partial\mathcal{B})).$$
It follows from Lemma \ref{first order lemma} that
$$P(T,B(\lambda_k,q))\leq P(T,B(\mu_l,{\rm dist}(\mu_l,\partial\mathcal{B})))\leq P_2.$$
Since $q<r_k={\rm dist}(\lambda_k,\partial\mathcal{B})$ is arbitrary, 
and arguing
using convergence of traces
as in the proof of Lemma \ref{second order lemma}, we conclude that
$$
P(T,B(\lambda_k,{\rm dist}(\lambda_k,\partial\mathcal{B})))=
\bigvee_{\substack{q\in\mathbb{Q}\\ 0<q<{\rm dist}(\lambda_k,\partial\mathcal{B})}} P(T,B(\lambda_k,q))\leq P_2.$$
Taking the supremum over $k,$ we infer that $P_1\leq P_2$.
By symmetry, also $P_2\leq P_1$.
Thus, $P_2=P_1$.
\end{proof}

\begin{lem}\label{low less up} Let $T\in\Lc_{\log}(\Mcal,\tau)$ be a mild operator.
For every closed set $\mathcal{B},$ we have $P_{(\downarrow)}(T,\mathcal{B})\leq P^{(\uparrow)}(T,\mathcal{B})$.
\end{lem}
\begin{proof} Suppose that $\{\lambda_k\}_{k\geq0}$ is a dense subset of ${\rm int}(\mathcal{B})$ and $\{\lambda_l'\}_{l\geq0}$ is a dense subset of $\mathbb{C}\backslash\mathcal{B}.$ We have
$$B(\lambda_k,{\rm dist}(\lambda_k,\partial\mathcal{B}))\subset\mathcal{B}\subset\mathbb{C}\backslash B(\lambda'_l,{\rm dist}(\lambda_l',\partial\mathcal{B})).$$
It follows from Lemma \ref{second order lemma} that
$$P(T,B(\lambda_k,{\rm dist}(\lambda_k,\partial\mathcal{B})))\leq P(T,\mathbb{C}\backslash B(\lambda_l',{\rm dist}(\lambda_l',\partial\mathcal{B}))).$$
Taking the infimum in the right hand side, we infer that
$$P(T,B(\lambda_k,{\rm dist}(\lambda_k,\partial\mathcal{B})))\leq P^{(\uparrow)}(T,\mathcal{B}).$$
Taking the supremum in the left hand side, we conclude the proof.
\end{proof}

\begin{thm}\label{mild const} Let $T\in\Lc_{\log}(\Mcal,\tau)$ be a mild operator.
If $\mathcal{B}$ is a closed set such that $m(\partial\mathcal{B})=0,$ then
\begin{enumerate}[{\rm (a)}]
\item\label{mida} $P_{(\downarrow)}(T,\mathcal{B})=P^{(\uparrow)}(T,\mathcal{B}).$
\item\label{midb} $P^{(\uparrow)}(T,\mathcal{B})=1-P_{(\downarrow)}(T^*,\overline{\mathbb{C}\backslash\mathcal{C}}),$ where $\mathcal{C}=\{z\mid \bar{z}\in\mathcal{B}\}$
and $\overline{\mathbb{C}\backslash\mathcal{C}}$ denotes the closure of $\mathbb{C}\backslash\mathcal{C}.$
\item\label{midc} $\tau(P^{(\uparrow)}(T,\mathcal{B}))=\nu_T(\mathcal{B}).$
\item\label{midd} $\nu_{TP^{(\uparrow)}(T,\mathcal{B})}=\frac1{\nu_T(\mathcal{B})}\nu_T|_{\mathcal{B}}$ and $\nu_{(1-P^{(\uparrow)}(T,\mathcal{B}))T}=\frac1{\nu_T(\mathbb{C}\backslash\mathcal{B})}\nu_T|_{\mathbb{C}\backslash\mathcal{B}}.$
\end{enumerate}
\end{thm}
\begin{proof} We start by proving \eqref{midb}.
Letting $\{\lambda_k'\}_{k\ge0}$ be a dense subset of $\Cpx\backslash\Bc$,
it follows from Definition \ref{haag mild def} and Lemma \ref{haag adjoint} that
\begin{align*}
P^{(\uparrow)}(T,\mathcal{B})&=\bigwedge_{k\geq0}\Big(1-P(T^*,B(\overline{\lambda'_k},{\rm dist}(\lambda_k',\partial\mathcal{B})))\Big) \\
&=1-\bigvee_{k\geq0}P(T^*,B(\overline{\lambda'_k},{\rm dist}(\lambda_k',\partial\mathcal{B})))
=1-P_{(\downarrow)}(T^*,\overline{\mathbb{C}\backslash\mathcal{C}}).
\end{align*}
This proves \eqref{midb}.

We now show that the Brown measure of $TP^{(\uparrow)}(T,\mathcal{B})$ is supported in $\mathcal{B}$.
Suppose that $\{\lambda_k'\}_{k\geq0}$ is a dense subset of $\mathbb{C}\backslash\mathcal{B}$.
By Definition \ref{haag mild def}, we have
$$P^{(\uparrow)}(T,\mathcal{B})\leq P(T,\mathbb{C}\backslash B(\lambda'_k,{\rm dist}(\lambda_k',\partial\mathcal{B}))).$$
Since $P(T,\mathbb{C}\backslash B(\lambda'_k,{\rm dist}(\lambda_k',\partial\mathcal{B})))$ is a Haagerup--Schultz projection
for $T$ and the set $\mathbb{C}\backslash B(\lambda'_k,{\rm dist}(\lambda_k',\partial\mathcal{B}))$, we have
$${\rm supp}(\nu_{TP^{(\uparrow)}(T,\mathcal{B})})\subset \mathbb{C}\backslash B(\lambda'_k,{\rm dist}(\lambda_k',\partial\mathcal{B})).$$
Taking the intersection over $k\geq0,$ we conclude that
\begin{equation}\label{eq:supp in B}
{\rm supp}(\nu_{TP^{(\uparrow)}(T,\mathcal{B})})\subset\mathcal{B}.
\end{equation}

We now prove \eqref{midc}. It follows from Theorem \ref{det nu matrix} that
$$\nu_T(\mathcal{B})=\tau(P^{(\uparrow)}(T,\mathcal{B}))\nu_{TP^{(\uparrow)}(T,\mathcal{B})}(\mathcal{B})+\tau(1-P^{(\uparrow)}(T,\mathcal{B}))\nu_{(1-P^{(\uparrow)}(T,\mathcal{B}))T}(\mathcal{B}).$$
Since $\nu_{TP^{(\uparrow)}(T,\mathcal{B})}$ is concentrated in $\mathcal{B}$,
it follows that $\nu_{TP^{(\uparrow)}(T,\mathcal{B})}(\mathcal{B})=1$.
This proves $\tau(P^{(\uparrow)}(T,\mathcal{B}))\leq\nu_T(\mathcal{B})$.

Since $T^*$ is also mild, we may apply the conclusion  of the previous paragraph to $T^*$ and
the set $\overline{\mathbb{C}\backslash\mathcal{C}},$ to obtain
\begin{align*}
\tau(P^{(\uparrow)}(T,\mathcal{B}))\stackrel{\eqref{midb}}{=}&1-\tau(P_{(\downarrow)}(T^*,\overline{\mathbb{C}\backslash\mathcal{C}}))
\stackrel{\ref{low less up}}{\geq}1-\tau(P^{(\uparrow)}(T^*,\overline{\mathbb{C}\backslash\mathcal{C}})) \\
&\geq1-\nu_{T^*}(\overline{\mathbb{C}\backslash\mathcal{C}})
=1-\nu_T(\overline{\mathbb{C}\backslash\mathcal{B}})=\nu_T(\mathcal{B})-\nu_T(\partial\mathcal{B}).
\end{align*}
Since $\nu_T(\partial\mathcal{B})=0,$ it follows that $\tau(P^{(\uparrow)}(T,\mathcal{B}))\geq\nu_T(\mathcal{B}).$ This proves \eqref{midc}.

We now prove \eqref{mida}. Applying \eqref{midb} and
applying \eqref{midc} to $T^*$ and the set $\overline{\mathbb{C}\backslash\mathcal{C}}$,
we infer that
\begin{align*}
\tau(P_{(\downarrow)}(T,\mathcal{B}))\stackrel{\eqref{midb}}{=}&1-\tau(P^{(\uparrow)}(T^*,\overline{\mathbb{C}\backslash\mathcal{C}}))\stackrel{\eqref{midc}}{=}
1-\nu_{T^*}(\overline{\mathbb{C}\backslash\mathcal{C}}) \\
=&\nu_T(\mathcal{B})-\nu_T(\partial\mathcal{B})=\nu_T(\mathcal{B})
\stackrel{\eqref{midc}}{=}\tau(P^{(\uparrow)}(T,\mathcal{B})).
\end{align*}
Since $\tau$ is faithful, the assertion of \eqref{mida} follows now from Lemma \ref{low less up}.

Since the Brown measure of $TP^{(\uparrow)}(T,\mathcal{B})$ is supported in $\mathcal{B}$,
the assertion of \eqref{midd} follows by using the $T$-invariance of $P^{(\uparrow)}(T,\mathcal{B})$, equation~\eqref{eq:supp in B}, \eqref{midc} and Theorem \ref{det nu matrix}.
\end{proof}

\begin{cor}
With $T$ and $\Bc$ as in Theorem~\ref{mild const}, $P^{(\uparrow)}(T,\Bc)$ is a Haagerup--Schultz projection
for the operator $T$ and the set $\Bc$.
\end{cor}

\section{Construction of Haagerup--Schultz projections}
\label{hs construct}

In this section, we let $T\in\Lc_{\log}(\Mcal,\tau)$ be arbitrary.
For an arbitrary Borel set $\mathcal{B}\subseteq\mathbb{C}$,
we will construct a Haagerup--Schultz projection $P(T,\mathcal{B})$ in the larger ultrapower von Neumann algebra $\Mcal_\omega$.
In order to distinguish the projections constructed at various stages,
we will initially reserve the notation $P(T,\mathcal{B})$ for the projections
constructed in Definition~\ref{def:someHSproj} for mild operators.

The following is clear from Lemma~\ref{lem:Yt} and Corollary~\ref{cor:Yt}, but we state it here for emphasis.
\begin{lem}\label{lem:Tident}
Let $T_k=T+\frac1{k+1}Z$, where $Z$ is the operator discussed in~\S\ref{subsec:Z},
Then $\{T_k\}_{k\ge0}\in\Fsc(\Mcal,\tau)$.
We have the tracial von Neumann algebra $(\Mcal_\omega,\tau_\omega)$ constructed as an ultrapower, and we regard $\Lc_{\log}(\Mcal,\tau)$
as being embedded in $\Lc_{\log}(\Mcal_\omega,\tau_\omega)$ via the diagonal embedding of Lemma~\ref{lem:diagemb}.
Let
$\tilde T=\pit(\{T_k\}_{k\ge0})\in\Lc_{\log}(\Mcal_\omega,\tau_\omega)$,
where $\pit$ is from Theorem \ref{isomorphism thm}.
Then $\tilde T$ is identified with $T$.
\end{lem}

\begin{lem}\label{main hs lemma}
Assume $\Delta(T)\neq0$.
Take $T_k$ and $\tilde T$ as in Lemma~\ref{lem:Tident}.
Suppose for every $k\geq0$, $P_k$ is a projection in $\Mcal$ satisfying $T_kP_k=P_kT_kP_k$.
Let $P=\pit(\{P_k\}_{k\geq0})$
and assume $P\neq0$ and $P\neq1$.
Then
\begin{gather}
\lim_{k\to\omega}\Delta_{P_k\Mcal P_k}(T_kP_k)=\Delta_{P\Mcal_\omega P}(\tilde TP), \label{det tp1} \\
\lim_{k\to\omega}\Delta_{(1-P_k)\Mcal(1-P_k)}((1-P_k)T_k)=\Delta_{(1-P)\Mcal_\omega(1-P)}((1-P)\tilde T). \label{det tp2}
\end{gather}
\end{lem}
\begin{proof}
Denote the left hand sides of \eqref{det tp1} and \eqref{det tp2} by $a$ and $b$, respectively.
Since $P\neq 0,1$, we have
$$\lim_{k\to\omega}\tau(P_k)=\tau_\omega(P)\in(0,1),\quad \lim_{k\to\omega}\tau(1-P_k)=\tau_\omega(1-P)\in(0,1).$$

\noindent{\em Step 1.
Show
\begin{equation}\label{det equ}
a^{\tau_\omega(P)}b^{\tau_\omega(1-P)}=\Big(\Delta_{P\Mcal_\omega P}(\tilde TP)\Big)^{\tau_\omega(P)}\Big(\Delta_{(1-P)\Mcal_\omega(1-P)}((1-P)\tilde T)\Big)^{\tau_\omega(1-P)}
\end{equation}
and $a\ne0$, $b\neq0$.}

\smallskip\noindent
It is clear that
$$a^{\tau_\omega(P)}b^{\tau_\omega(1-P)}=\lim_{k\to\omega}\Big(\Delta_{P_k\Mcal P_k}(T_kP_k)\Big)^{\tau(P_k)}\Big(\Delta_{(1-P_k)\Mcal(1-P_k)}((1-P_k)T_k)\Big)^{\tau(1-P_k)}.$$
By Theorem \ref{det nu matrix}, we have that
$$\Big(\Delta_{P_k\Mcal P_k}(T_kP_k)\Big)^{\tau(P_k)}\Big(\Delta_{(1-P_k)\Mcal(1-P_k)}((1-P_k)T_k)\Big)^{\tau(1-P_k)}=\Delta(T_k).$$
Using Proposition 4.5 in \cite{HS2} (see the remarks found in~\S\ref{subsec:Z}), we have
$$a^{\tau_\omega(P)}b^{\tau_\omega(1-P)}=\lim_{k\to\omega}\Delta(T_k)=\left(\lim_{k\to\omega}\Delta\left(|T|^2+\frac1{(k+1)^2}\right)\right)^{1/2}=\Delta(T),$$
where to get the last equality, we invoked~\eqref{eq:Deltaeps}.
Thus, $a\neq0$ and $b\neq0$.
Let $\Zt=\{\frac1{k+1}Z\}_{k\ge0}$.
By Lemma~\ref{lem:Yt}, $\Zt\in\Fsc(\Mcal,\tau)$ and $\pit(\Zt)=0$.
Thus, $\pit(\Tt)=\pit(\{T\}_{k\ge0})$ and, by Corollary~\ref{cor:Yt}, $\Delta(\Tt)=\Delta(T)$.
By Theorem \ref{det nu matrix}, $\Delta(\Tt)$ equals the right-hand-side of~\eqref{det equ}.
This proves that~\eqref{det equ} holds.

\smallskip
\noindent{\em Step 2.
Show
\begin{align}
a&\leq\Delta_{P\Mcal_\omega P}(\tilde TP). \label{det inequ1} \\
b&\leq\Delta_{(1-P)\Mcal_\omega(1-P)}((1-P)\tilde T). \label{det inequ2}
\end{align}
}

\smallskip\noindent
It is clear that
$$\Delta_{P_k\Mcal P_k}(T_kP_k)^2=\Delta_{P_k\Mcal P_k}(|T_kP_k|^2)\leq\Delta_{P_k\Mcal P_k}(|T_kP_k|^2+\eps^2P_k).$$
Hence,
$$a^2\leq\lim_{k\to\omega}\Delta_{P_k\Mcal P_k}(|T_kP_k|^2+\varepsilon^2P_k).$$
By Theorem~\ref{det nu matrix}, we have that
$$\Big(\Delta_{P_k\Mcal P_k}(|T_kP_k|^2+\varepsilon^2P_k)\Big)^{\tau(P_k)}\varepsilon^{2\tau(1-P_k)}=\Delta(|T_kP_k|^2+\varepsilon^2).$$
Hence,
\begin{align*}
a^2&\leq \lim_{k\to\omega}\varepsilon^{-2\frac{\tau(1-P_k)}{\tau(P_k)}}\Big(\Delta(|T_kP_k|^2+\varepsilon^2)\Big)^{\frac1{\tau(P_k)}} \\
&=\varepsilon^{-2\frac{\tau_\omega(1-P)}{\tau_\omega(P)}}\Big(\lim_{k\to\omega}\Delta(|T_kP_k|^2+\varepsilon^2)\Big)^{\frac1{\tau_\omega(P)}}.
\end{align*}
The sequences $\{T\}_{k\geq0}$ and $\{\frac1{k+1}Z\}_{k\geq0}$ are both in the algebra $\Fsc(\Mcal,\tau).$ Hence, so is the sequence $\{|T_kP_k|^2+\varepsilon^2\}_{k\geq0}.$ It follows from Theorem \ref{det converge} that
$$a^2\leq\varepsilon^{-2\frac{\tau_\omega(1-P)}{\tau_\omega(P)}}\Delta\Big(\pit\Big(\Big\{|T_kP_k|^2+\varepsilon^2\Big\}_{k\geq0}\Big)\Big)^{\frac1{\tau_\omega(P)}}.$$
Since $\pit$ is a $*$-homomorphism, it follows that
$$\pit\Big(\Big\{|T_kP_k|^2+\varepsilon^2\Big\}_{k\geq0}\Big)=|\Tt P|^2+\varepsilon^2.$$
Hence, we have
$$a^2\leq\varepsilon^{-2\frac{\tau_\omega(1-P)}{\tau_\omega(P)}}\Delta(|\Tt P|^2+\varepsilon^2)^{\frac1{\tau_\omega(P)}}=\Delta_{P\Mcal_\omega P}(|\Tt P|^2+\varepsilon^2P),$$
where for the equality we have used Theorem~\ref{det nu matrix}.
Since, again using~\eqref{eq:Deltaeps}, we have
$$\lim_{\eps\to 0}\Delta_{P\Mcal_\omega P}(|\Tt P|^2+\varepsilon^2P)
=\Delta_{P\Mcal_\omega P}(|\Tt P|^2)=\big(\Delta_{P\Mcal_\omega P}(\Tt P)\big)^2,$$
\eqref{det inequ1} follows.
Repeating this argument {\it mutatis mutandis}, we obtain~\eqref{det inequ2}.
The result now follows by combining \eqref{det equ}, \eqref{det inequ1} and \eqref{det inequ2}, since $a\ne0$ and $b\ne0.$
\end{proof}

\begin{lem}\label{713} Let $T\in\Lc_{{\rm log}}(\Mcal,\tau)$ be arbitrary. If $\mathcal{B}$ is a disk or the complement of a disk in $\Cpx$
such that $\nu_T(\partial\mathcal{B})=0$, then
\begin{equation}\label{eq:limintlog}
\lim_{n\to\infty}\int_{\mathcal{B}}\log(|z-\lambda|)d\nu_{T+\frac1{n+1}Z}(z)=\int_{\mathcal{B}}\log(|z-\lambda|)d\nu_T(z)
\end{equation}
for every $\lambda\notin\partial\mathcal{B}$ such that $\Delta(T-\lambda)\neq0.$
\end{lem}
\begin{proof} {\em Step 1. Prove the assertion assuming $\lambda\in\Cpx\backslash\overline{\mathcal{B}}$.}

\noindent
Let $m\ge\min(2,1+|\lambda|)$ and let $\phi_m$ be a compactly supported continuous function such that $0\le \phi_m\le 1$ and
$\phi_m(z)=1$ for every $z$ with $|z|\leq m.$ The mapping
$$z\mapsto\phi_m(z)\log(|z-\lambda|),\quad z\in\mathcal{B}$$
is a compactly supported continuous function.
Note that, from Corollary 4.6 in~\cite{HS2}, $\nu_{T+\frac1{n+1}Z}$ is Lebesgue absolutely continuous.
Since $\nu_T(\partial\mathcal{B})=0$, it follows from
Corollary 4.8
in \cite{HS2} that
\begin{equation}\label{eq:limintphim}
\lim_{n\to\infty}\int_{\mathcal{B}}\phi_m(z)\log(|z-\lambda|)\,d\nu_{T+\frac1{n+1}Z}(z)=\int_{\mathcal{B}}\phi_m(z)\log(|z-\lambda|)d\nu_T(z).
\end{equation}
(Note that Corollaries 4.6 and 4.8 of~\cite{HS2} remain valid for $T\in\Lc_{\log}(\Mcal,\tau)$
--- see the remarks found in~\S\ref{subsec:Z}.)
On the other hand, we have
\begin{multline*}
\Big|\int_{\mathcal{B}}\phi_m(z)\log(|z-\lambda|)d\nu_{T+\frac1{n+1}Z}(z)
 -\int_{\mathcal{B}}\log(|z-\lambda|)d\nu_{T+\frac1{n+1}Z}(z)\Big| \\
\leq\int_{|z|\geq m}|\log(|z-\lambda|)|d\nu_{T+\frac1{n+1}Z}(z)\leq 2\int_{|z|\geq m}\log(|z|)d\nu_{T+\frac1{n+1}Z}(z) \\
\leq 4\int_{\mathbb{C}}\log^+\left(\frac{|z|}{m^{1/2}}\right)d\nu_{T+\frac1{n+1}Z}(z).
\end{multline*}
Applying Lemma 2.20 in \cite{HS1} and Lemma~\ref{lem:elemlog+}, we infer
\begin{multline}\label{eq:intphidiff}
\Big|\int_{\mathcal{B}}\phi_m(z)\log(|z-\lambda|)d\nu_{T+\frac1{n+1}Z}(z)
-\int_{\mathcal{B}}\log(|z-\lambda|)d\nu_{T+\frac1{n+1}Z}(z)\Big| \\
\leq4\tau(\log^+\left(\frac{|T+\frac1{n+1}Z|}{m^{1/2}}\right))
\le 4\tau(\log^+\left(\frac{2|T|}{m^{1/2}}\right))+4\tau(\log^+\left(\frac{2|Z|}{m^{1/2}(n+1)}\right))
\end{multline}
Similarly, we have
$$\Big|\int_{\mathcal{B}}\phi_m(z)\log(|z-\lambda|)d\nu_{T}(z)-\int_{\mathcal{B}}\log(|z-\lambda|)d\nu_{T}(z)\Big|
\leq 8\tau(\log^+\left(\frac{\mu(T)}{m^{1/2}}\right)).$$
Since the upper bound~\eqref{eq:intphidiff} is uniform in $n$, letting $m\to\infty$ in~\eqref{eq:limintphim} proves~\eqref{eq:limintlog}.
This finishes Step 1.

\medskip
\noindent
{\em Step 2. Prove the assertion assuming $\lambda\in\text{int}(\mathcal{B})$.}

\smallskip\noindent
Applying Step~1 in the case when $\Bc$ is replaced by $\Cpx\backslash\Bc$, we obtain
\begin{multline*}
\lim_{n\to\infty}\int_{\mathcal{B}}\log(|z-\lambda|)d\nu_{T+\frac1{n+1}Z}(z) \\
=\lim_{n\to\infty}\int_{\mathbb{C}}\log(|z-\lambda|)d\nu_{T+\frac1{n+1}Z}(z)
-\lim_{n\to\infty}\int_{\mathbb{C}\backslash\mathcal{B}}\log(|z-\lambda|)d\nu_{T+\frac1{n+1}Z}(z) \\
=\lim_{n\to\infty}\log(\Delta\left(T+\frac1{n+1}Z-\lambda\right))-\int_{\mathbb{C}\backslash\mathcal{B}}\log(|z-\lambda|)d\nu_T(z).
\end{multline*}
From Proposition 4.5 in \cite{HS2},
$$\lim_{n\to\infty}\Delta(T+\frac1{n+1}Z-\lambda)=\Delta(T-\lambda).$$
Hence,
$$\lim_{n\to\infty}\int_{\mathcal{B}}\log(|z-\lambda|)d\nu_{T+\frac1{n+1}Z}(z)
=\int_{\mathbb{C}}\log(|z-\lambda|)d\nu_T(z)-\int_{\mathbb{C}\backslash\mathcal{B}}\log(|z-\lambda|)d\nu_T(z).$$
This finishes the proof of Step 2, and of the lemma.
\end{proof}

Recall (see Theorem~\ref{thm:TZmild}) that $T+\frac1{k+1}Z$ is mild for every $k\ge0$.
In the next result $\pi$ is the quotient map as described in~\S\ref{subsec:ultraP}
and $P(\cdot,\cdot)$ is as given in Definition~\ref{def:someHSproj}.

\begin{lem}\label{general disk lemma} Let $\mathcal{B}$ be a disk or the complement of a disk
such that $\nu_T(\mathcal{B})\notin\{0,1\}$ and $\nu_T(\partial\mathcal{B})=0$.
Let $\tilde T$ be as in Lemma \ref{lem:Tident}.
Consider the projection
$$P_1(T,\mathcal{B}):=\pi(\{P(T+\frac1{k+1}Z,\mathcal{B})\}_{k\geq0})\in\Mcal_\omega.$$
Then:
\begin{enumerate}[{\rm (a)}]
\item\label{gbla} $P_1(T,\mathcal{B})$ is $\tilde T$-invariant.
\item\label{gblb} $\tau_\omega(P_1(T,\mathcal{B}))=\nu_T(\mathcal{B}).$
\item\label{gblc} $\nu_{\tilde TP_1(T,\mathcal{B})}=\frac1{\nu_T(\mathcal{B})}\nu_T|_{\mathcal{B}}$ and
$\nu_{(1-P_1(T,\mathcal{B}))\tilde T}=\frac1{\nu_T(\mathbb{C}\backslash\mathcal{B})}\nu_T|_{\mathbb{C}\backslash\mathcal{B}}$.
\end{enumerate}
\end{lem}
\begin{proof}
Part~\eqref{gbla} is clear.

From $\nu_T(\partial\Bc)=0$ and Corollary 4.8 of~\cite{HS2}, we have
\begin{equation}\label{eq:nuTB}
\lim_{k\to\infty}\nu_{T+\frac1{k+1}Z}(\Bc)=\nu_T(\Bc).
\end{equation}
Now part~\eqref{gblb} follows from $\tau(P(T+\frac1{k+1}Z,\mathcal{B}))=\nu_{T+\frac1{k+1}Z}(\Bc)$.

We will now show \eqref{gblc}.
To ease the notation, we write $P$ and $P_k$ instead of $P_1(T,\mathcal{B})$ and $P(T+\frac1{k+1}Z,\mathcal{B})$, respectively.
Applying Lemma \ref{main hs lemma} to the operator $\tilde T-\lambda,$ we infer that if $\Delta(T-\lambda)\ne0$, then
$$
\lim_{k\to\omega}\Delta_{P_k\Mcal P_k}((T+\frac1{k+1}Z-\lambda)P_k)=\Delta_{P\Mcal_\omega P}((\Tt-\lambda)P).
$$
On the other hand, since $P_k$ is a Haagerup--Schultz projection, we have
$$
\nu_{(T+\frac1{k+1}Z)P_k}=\frac1{\nu_{T+\frac1{k+1}Z}(\Bc)}\nu_{T+\frac1{k+1}Z}|_\Bc
$$
and then from~\eqref{br def} we get
$$
\log(\Delta_{P_k\Mcal P_k}((T+\frac1{k+1}Z-\lambda)P_k))=
\frac1{\nu_{T+\frac1{k+1}Z}(\Bc)}\int_\Bc\log(|z-\lambda|)\,d\nu_{T+\frac1{k+1}Z}(z).
$$
Thus, using~\eqref{eq:nuTB} and Lemma~\ref{713}, we get that if $\Delta(T-\lambda)\ne0$ and $\lambda\notin\partial\Bc$, then
\begin{equation}\label{eq:subharmon}
\log(\Delta_{P\Mcal_\omega P}((T-\lambda)P))=\frac1{\nu_T(\mathcal{B})}\int_{\mathcal{B}}\log(|z-\lambda|)d\nu_T(z).
\end{equation}
We will show that this equality holds at all values of $\lambda\in\Cpx$, by a sort of continuity argument.
First, recall that the subharmonic function $\log\Delta(T-\lambda)$ is not identically $-\infty$ (see~\S\ref{subsec:FK}).
Thus, the equality~\eqref{eq:subharmon} holds at almost every $\lambda\in\Cpx$ (see, e.g., Corollary 2.5.3 of~\cite{Ra95}).
Both sides of the equality~\eqref{eq:subharmon} are subharmonic functions of $\lambda$ in $\Cpx$.
If $f$ is a subharmonic function on $\Cpx$ and if $N\subseteq\Cpx$ is a set of Lebesgue measure zero,
then for every $w\in N$ there is a sequence $\{w_n\}_{n\ge0}$ in $\Cpx\backslash N$ such that $\lim_{n\to\infty}w_n=w$
and $\lim_{n\to\infty}f(w_n)=f(w)$.
In the case $f(w)=-\infty$, this follows by upper semicontinuity of $f$, while in the case $f(w)>-\infty$, it follows from
Corollary 5.4.4 of~\cite{Ra95}.
Thus, from the fact that the identity~\eqref{eq:subharmon} holds almost everywhere, we get that it holds for all $\lambda\in\Cpx$.

Now taking the Laplacian of both sides of~\eqref{eq:subharmon} yields the first equality in~\eqref{gblc}.
The second equality in \eqref{gblc} follows from the first one and Equation~\eqref{nuT} in Theorem \ref{det nu matrix}.
\end{proof}

\begin{lem}\label{general closed lemma}
Let $\tilde T$ be as in Lemma \ref{lem:Tident}.
Let $\mathcal{B}$ be a closed set such that $\nu_T(\mathcal{B})\notin\{0,1\}$, $\nu_T(\partial\mathcal{B})=0$ and $m(\partial\mathcal{B})=0$.
Then the projection
$$P_2(T,\mathcal{B}):=\pi(\{P^{(\uparrow)}(T+\frac1{k+1}Z,\mathcal{B})\}_{k\geq0})\in\Mcal_\omega$$
satisfies the conditions
\begin{enumerate}[{\rm (a)}]
\item\label{gclos0} $P_2(T,\mathcal{B})$ is $\tilde T$-invariant.
\item\label{gclosa} $\tau_\omega(P_2(T,\mathcal{B}))=\nu_T(\mathcal{B}).$
\item\label{gclosb} $\nu_{\tilde TP_2(T,\mathcal{B})}=\frac1{\nu_T(\mathcal{B})}\nu_T|_{\mathcal{B}}$
and $\nu_{(1-P_2(T,\mathcal{B}))\tilde T}=\frac1{\nu_T(\mathbb{C}\backslash\mathcal{B})}\nu_T|_{\mathbb{C}\backslash\mathcal{B}}.$
\end{enumerate}
\end{lem}
\begin{proof} Part \eqref{gclos0} holds because $P^{(\uparrow)}(T+\frac1{k+1}Z,\mathcal{B})$ is $(T+\frac1{k+1}Z)$-invariant.

We now prove \eqref{gclosa}. We have
$$\tau_\omega(P_2(\tilde T,\mathcal{B}))=\lim_{k\to\omega}\tau(P^{(\uparrow)}(T+\frac1{k+1}Z,\mathcal{B})).$$
Since $m(\partial\mathcal{B})=0$ and since the Brown measure of $T+\frac1{k+1}Z$ is absolutely continuous with respect to the Lebesgue measure, it follows from Theorem \ref{mild const} \eqref{midc} that
$$\tau(P^{(\uparrow)}(T+\frac1{k+1}Z,\mathcal{B}))=\nu_{T+\frac1{k+1}Z}(\mathcal{B}).$$
Since $\nu_T(\partial\mathcal{B})=0,$ it follows from \eqref{ultra trace def} and Corollary 4.8 in \cite{HS2} that
$$\tau_\omega(P_2(T,\mathcal{B}))=\lim_{k\to\omega}\nu_{T+\frac1{k+1}Z}(\mathcal{B})=\nu_T(\mathcal{B}).$$
This proves \eqref{gclosa}.

We now prove \eqref{gclosb}.
Suppose that $\{\lambda_n'\}_{n\geq0}$ is a dense subset of $\mathbb{C}\backslash\mathcal{B}$.
Fix $n\geq0$ and $r<{\rm dist}(\lambda_n',\partial\mathcal{B})$ such that $\nu_T(\partial B(\lambda_n',r))=0$.
By Definition \ref{haag mild def}, we have
$$P^{(\uparrow)}(T+\frac1{k+1}Z,\mathcal{B})\leq P(T+\frac1{k+1}Z,\mathbb{C}\backslash B(\lambda'_n,r)).$$
It follows from the definition of $P_2(T,\mathcal{B})$ in Lemma \ref{general closed lemma} and the definition of $P_1(T,\mathbb{C}\backslash B(\lambda'_n,r))$  as in Lemma \ref{general disk lemma} that
$$P_2(T,\mathcal{B})\leq\pi(\{P(T+\frac1{k+1}Z,\mathbb{C}\backslash B(\lambda'_n,r))\}_{k\geq0})=P_1(T,\mathbb{C}\backslash B(\lambda'_n,r)).$$
It follows from Corollary \ref{nu matrix cor} that $\nu_{\tilde TP_2(T,\mathcal{B})}$ is supported in $\mathbb{C}\backslash B(\lambda'_n,r).$ Since $r$ can be taken arbitrarily close to ${\rm dist}(\lambda_n',\partial\mathcal{B}),$ it follows that
$${\rm supp}(\nu_{\tilde TP_2(T,\mathcal{B})})\subset \mathbb{C}\backslash B(\lambda'_n,{\rm dist}(\lambda_n',\partial\mathcal{B})).$$
Taking the intersection over $n\geq0,$ we conclude that
$${\rm supp}(\nu_{TP(T,\mathcal{B})})\subset\mathcal{B}.$$
Combining the proved inequality with \eqref{gclosa}, we conclude the proof of \eqref{gclosb}.
\end{proof}

\begin{lem}\label{general monotone lemma} Let $\mathcal{B}_1,\mathcal{B}_2\subset\mathbb{C}$ be closed sets
such that $\nu_T(\partial\mathcal{B}_1)=\nu_T(\partial\mathcal{B}_2)=0$ and $m(\partial\mathcal{B}_1)=m(\partial\mathcal{B}_2)=0$
and $\mathcal{B}_1\subset\mathcal{B}_2$.
Then $P_2(T,\mathcal{B}_1)\leq P_2(T,\mathcal{B}_2).$
\end{lem}
\begin{proof}
We may without loss of generality assume $\Bc_2\ne\Cpx$.
If $\{\lambda_n'\}_{n\geq0}$ is a dense subset in $\mathbb{C}\backslash\mathcal{B}_1,$ then it follows from the definition of $P^{(\uparrow)}$ that
$$P^{(\uparrow)}(T+\frac1{k+1}Z,\mathcal{B}_1)=\bigwedge_{n\geq0}P(T+\frac1{k+1}Z,\mathbb{C}\backslash B(\lambda'_n,{\rm dist}(\lambda_n',\partial\mathcal{B}))).$$
Clearly, $\{\lambda_n'\}_{\substack{n\geq0\\\lambda_n'\notin\mathcal{B}_2}}$ is dense in $\mathbb{C}\backslash\mathcal{B}_2.$ It follows from the definition of $P^{(\uparrow)}$ that
$$P^{(\uparrow)}(T+\frac1{k+1}Z,\mathcal{B}_2)=\bigwedge_{\substack{n\geq0\\\lambda_n'\notin\mathcal{B}_2}}P(T+\frac1{k+1}Z,\mathbb{C}\backslash B(\lambda'_n,{\rm dist}(\lambda_n',\partial\mathcal{B}))).$$
It follows that
$$P^{(\uparrow)}(T+\frac1{k+1}Z,\mathcal{B}_1)\leq P^{(\uparrow)}(T+\frac1{k+1}Z,\mathcal{B}_2),\quad k\geq0.$$
The assertion follows from the definition of $P_2(T,\cdot)$ in Lemma \ref{general closed lemma}.
\end{proof}

We are now ready to prove Theorem \ref{construction theorem}.

\begin{proof}[Proof of Theorem \ref{construction theorem}]
We are identifying $T$ with $\tilde T$, as in Lemma~\ref{lem:Tident}.

\smallskip\noindent
{\em Step 1:  The case of closed subsets.}

\smallskip\noindent
Suppose that $\mathcal{B}$ is closed. Consider the larger closed sets
$$\mathcal{B}_t=\Big\{z\in\mathbb{C}\;\Big|\; {\rm dist}(z,\mathcal{B})\leq t\Big\},\quad t\in(0,1).$$
We claim that the projection
$$P(T,\mathcal{B})=\bigwedge_{\substack{t>0\\m(\partial\mathcal{B}_t)=0\\\nu_T(\partial\mathcal{B}_t)=0}}P_2(T,\mathcal{B}_t)$$
satisfies~\eqref{csta}, \eqref{cstb} and~\eqref{cstc}.
The functions $t\mapsto m(\mathcal{B}_t)$ and $t\mapsto\nu_T(\mathcal{B}_t)$ are monotone.
Hence, they have at most countably many discontinuity points
and $m(\partial\mathcal{B}_t)=0$ and $\nu_T(\partial\mathcal{B}_t)=0$ for all but countably many values of $t>0$.
Thus, in Lemma \ref{general closed lemma}, the projections $P_2(T,\mathcal{B}_t)$ have been defined
for all but countably many values of $t>0$.
By Lemma~\ref{general monotone lemma}, these projections are increasing in $t$.
Using Lemma~\ref{general closed lemma}, it follows that
\begin{equation}\label{proof of cstb}
\tau_\omega(P(T,\mathcal{B}))=\inf_{t>0}\tau_\omega(P_2(T,\mathcal{B}_t))=\inf_{t>0}\nu_T(\mathcal{B}_t)=\nu_T(\mathcal{B}).
\end{equation}
This proves \eqref{cstb}.

By Lemma \ref{general closed lemma}, every projection $P_2(T,\mathcal{B}_t)$ is $T$-invariant and, therefore,
$P(T,\mathcal{B})$ is also $T$-invariant.
This proves \eqref{csta}.

We now prove \eqref{cstc}. For every $t>0,$ it follows from Lemma \ref{general closed lemma} and Corollary \ref{nu matrix cor} that
$${\rm supp}(\nu_{TP(T,\mathcal{B})})\subset{\rm supp}(\nu_{TP(T,\mathcal{B}_t)})\subset\mathcal{B}_t.$$
Taking the intersection over $t>0,$ we infer
$${\rm supp}(\nu_{TP(T,\mathcal{B})})\subset\mathcal{B}.$$
It follows from \eqref{proof of cstb} and Equation~\eqref{nuT} of Theorem \ref{det nu matrix} that
$$\nu_T(\mathcal{B})=\nu_T(\mathcal{B})\nu_{TP(T,\mathcal{B})}(\mathcal{B})+(1-\nu_T(\mathcal{B}))\nu_{(1-P(T,\mathcal{B}))T}(\mathcal{B}).$$
Since the probability measure $\nu_{TP(T,\mathcal{B})}$ is supported in $\mathcal{B}$,
we must have $\nu_{TP(T,\mathcal{B})}(\mathcal{B})=1$ and, therefore, $\nu_{(1-P(T,\mathcal{B}))T}(\mathcal{B})=0$.
Hence, $\nu_{(1-P(T,\mathcal{B}))T}$ is concentrated in $\Cpx\backslash\Bc$.
Now \eqref{cstc} follows from Equation~\eqref{nuT} of Theorem \ref{det nu matrix}.

If $\mathcal{B}_1$ and $\mathcal{B}_2$ are closed sets such that $\mathcal{B}_1\subset\mathcal{B}_2,$ then $\mathcal{B}_{1t}\subset\mathcal{B}_{2t}$ for every $t>0.$ It follows from Lemma \ref{general monotone lemma} that $P(T,\mathcal{B}_1)\leq P(T,\mathcal{B}_2).$ This concludes the proof for closed sets.

\smallskip\noindent
{\em Step 2: Arbitrary Borel sets.}

\smallskip\noindent
Let $\mathcal{B}$ be an arbitrary Borel set.
Let $\Compact(\Bc)$ be the set of all compact subsets of $\Bc$ and let
\begin{equation}\label{eq:PTB}
P(T,\mathcal{B})=\bigvee_{\Kc\in\Compact(\Bc)}P(T,\mathcal{K}).
\end{equation}
By Step 1, every projection $P(T,\mathcal{K})$ is $T$-invariant and, therefore, $P(T,\mathcal{B})$ is also $T$-invariant. This proves \eqref{csta}.

$\Compact(\mathcal{B})$ is directed by inclusion.
By Step~1, the projections $\{P(T,\mathcal{K})\}_{\mathcal{K}\subset\mathcal{B}}$ form an increasing net.
Since the trace $\tau_\omega$ is normal, we get
$$\tau_\omega(P(T,\mathcal{B}))=\sup_{\Kc\in\Compact(\Bc)}\tau_\omega(P(T,\mathcal{K}))=\sup_{\Kc\in\Compact(\Bc)}\nu_T(\mathcal{K})=\nu_T(\mathcal{B}).$$
This proves \eqref{cstb}.

We now prove \eqref{cstc};
to this end, we assume $\nu_T(\mathcal{B})>0$.
Fix a compact set $\mathcal{K}\subset\mathcal{B}$.
By Step 1, we have $\nu_{TP(T,\mathcal{K})}(\mathcal{K})=1$ and $\nu_{(1-P(T,\mathcal{K}))T}(\mathcal{K})=0$.
It follows now from Corollary \ref{nu matrix cor} that
\begin{equation}\label{eq:nuK0}
\nu_{(1-P(T,\mathcal{K}))TP(T,\mathcal{B})}(\mathcal{K})=0,
\end{equation}
where this is the Brown measure taken with respect to the renormalized trace on the algebra $\Mcal$ compressed by the projection
$P(T,\mathcal{B})-P(T,\mathcal{K})$.
Again applying Equation~\eqref{nuT} of Theorem \ref{det nu matrix} to the algebra $P(T,\mathcal{B})\Mcal_\omega P(T,\mathcal{B})$,
we infer that
$$\nu_{TP(T,\mathcal{B})}=\frac{\tau_\omega(P(T,\mathcal{K}))}{\tau_\omega(P(T,\mathcal{B}))}\nu_{TP(T,\mathcal{K})}+\frac{\tau_\omega(P(T,\mathcal{B})-P(T,\mathcal{K}))}{\tau_\omega(P(T,\mathcal{B}))}\nu_{(1-P(T,\mathcal{K}))TP(T,\mathcal{B})}.$$
Applying the latter equality to the set $\mathcal{K}$ and using~\eqref{eq:nuK0}, we obtain
\begin{align*}
\nu_{TP(T,\mathcal{B})}(\mathcal{K})&=\frac{\tau_\omega(P(T,\mathcal{K}))}{\tau_\omega(P(T,\mathcal{B}))}\nu_{TP(T,\mathcal{K})}(\mathcal{K}) \\
&\qquad\qquad\qquad+\frac{\tau_\omega(P(T,\mathcal{B})-P(T,\mathcal{K}))}{\tau_\omega(P(T,\mathcal{B}))}\nu_{(1-P(T,\mathcal{K}))TP(T,\mathcal{B})}(\mathcal{K}) \\
&=\frac{\nu_T(\mathcal{K})}{\tau_\omega(P(T,\mathcal{B}))}.
\end{align*}
Letting $\mathcal{K}$ grow in $\Compact(\mathcal{B})$,
we have that $\nu_{TP(T,\mathcal{B})}(\mathcal{K})$ increases to $\nu_{TP(T,\mathcal{B})}(\mathcal{B})$ and $\nu_T(\mathcal{K})$
increases to $\nu_T(\mathcal{B})$.
Hence, $\nu_{TP(T,\mathcal{B})}(\mathcal{B})=1$.
Now arguing as in Step 1 proves \eqref{cstc} for general Borel sets.

The final assertion, that for Borel sets $\Bc_1\subset\Bc_2$ we have $P(T,\Bc_1)\le P(T,\Bc_2)$, follows from the corresponding
statement for compact sets and the definition~\eqref{eq:PTB}.
\end{proof}

\section{Faux Expectations}
\label{sec:FauxExp}

Let $\mathcal{D}$ be an abelian von Neumann subalgebra  of $\Mcal$ and assume that $\Dc$ has a separable predual.
We have the conditional expectation map $\Exp_{\Dc'}:\Mcal\to\Mcal\cap\Dc'$ (see~\S\ref{subsec:condexp})
and we would like to have an extension of $\Exp_{\Dc'}$
to a map defined from all of $\Lc_{\log}(\Mcal,\tau)$ onto $\Lc_{\log}(\Mcal,\tau)\cap\Dc'$,
but this will not, in general exist (see Appendix~\ref{sec:NoCondExp}).
However we will manage to construct some weaker sort of conditional expectation of special sorts of operators in $\Lc_{\log}(\Mcal,\tau)$.

Let $S\in\Lc_{\log}(\Mcal,\tau)$.
Let $\Cpx=\Dc_0\subset\Dc_1\subset\cdots$ be an increasing sequence of finite dimensional $*$-subalgebras of $\Dc$
whose union is weakly dense in $\Dc$.
Let $p_1^{(m)},\ldots,p_{n(m)}^{(m)}$ be the minimal projections of $\Dc_m$.
We use the notation
\begin{equation}\label{eq:ExpDm'}
\Exp_{\Dc_m'}(S)=\sum_{k=1}^{n(m)}p_k^{(m)}Sp_k^{(m)}.
\end{equation}
Note that $\Exp_{\Dc_m'}(S)\in\Lc_{\log}(\Mcal,\tau)\cap\Dc_m'$.

Here is a standing hypothesis that applies in this section:
\begin{hypoth}\label{hyp:Sinv}
Suppose $S\in\Lc_{\log}(\Mcal,\tau)$ and that for every $m$ and every $k\in\{1,\ldots,n(m)\}$
the projection $p_1^{(m)}+\cdots+p_k^{(m)}$ is $S$-invariant.
\end{hypoth}

\begin{lem}\label{lem:ExpinF}
We have
$$\{\Exp_{\mathcal{D}_m'}(S)\}_{m\geq0}\in\Fsc(\Mcal,\tau).$$
\end{lem}
\begin{proof}
The paper~\cite{DSZ} was about bounded operators, but the proof of Lemma 18 in \cite{DSZ}
applies equally well to elements of $\Lc_{\log}$, and from this we obtain
$$\Exp_{\mathcal{D}_m'}(S)\prec\prec_{\log}S,\qquad(m\in\Nats).$$
Similarly, the proof of Lemma~25 in \cite{DSZ} applies to elements of $\Lc_{\log}$ and we have
$$\tau(\log^+(\frac1n|\Exp_{\mathcal{D}_m'}(S)|))\leq\tau(\log^+(\frac1n|S|)),\qquad(m,n\in\Nats).$$
Since $\lim_{n\to\infty}\tau(\log^+(\frac1n|S|))=0$, the assertion follows from Lemma~\ref{lem:Fchar}.
\end{proof}

Theorem \ref{isomorphism thm} permits now the following definition.

\begin{defi}\label{def:FauxExp}
We let
$$\FauxExp_{\mathcal{D}'}(S)=\pit(\{\Exp_{\mathcal{D}_m'}(S)\}_{m\geq0})\in\Lc_{\log}(\Mcal_\omega,\tau_\omega). $$
\end{defi}
Though not evident from the notation, the above operator depends on the choice of projections $p^{(m)}_k$ satisfying Hypothesis~\ref{hyp:Sinv}.

Below is an easy and general observation about commutants.
Recall that a set is said to be self-adjoint if it is close under taking adjoints.
\begin{lem}\label{lem:commutant}
If $\Xc$ is a self-adjoint subset of $\Sc(\Mcal,\tau)$, then the relative commutant $\Mcal\cap\Xc'$ is a von Neumann subalgebra
of $\Mcal$.
\end{lem}
\begin{proof}
For $S\in\Xc$ and $a\in\Mcal$, we have $aS=Sa$ and $aS^*=S^*a$ if and only if, letting $S=U|S|$ be the polar decomposition
of $S$, the element $a$ commutes with $U$, with $U^*$ and with all spectral projections of $|S|$.
Thus, $\Mcal\cap\Xc'=\Mcal\cap\Yc'$ for a self-adjoint set $\Yc\subset\Mcal$.
But $\Mcal\cap\Yc'$ is a von Neumann subalgebra of $\Mcal$.
\end{proof}

\begin{lem}\label{lem:FauxExp commutes}
If we identify $\Dc\subset\Mcal$ with its image in $\Mcal_\omega$ under the diagonal map (see~\S\ref{subsec:ultraP}),
then $\FauxExp_{\Dc'}(S)$ lies
in the relative commutant $\Lc_{\log}(\Mcal_\omega,\tau_\omega)\cap\Dc'$.
\end{lem}
\begin{proof}
Given $m(0)$, for every $m\ge m(0)$ the operator $\Exp_{\Dc_m'}(S)$ commutes with every element of $\Dc_{m(0)}$.
From this, we have
$$
\Dc_{m(0)}\subset\Mcal_\omega\cap\{\FauxExp_{\Dc'}(S),\FauxExp_{\Dc'}(S)^*\}'.
$$
Using Lemma~\ref{lem:commutant}, we get $\Dc\subset\Mcal_\omega\cap\{\FauxExp_{\Dc'}(S),\FauxExp_{\Dc'}(S)^*\}'$.
\end{proof}

\begin{lem}\label{det expectation lemma}
We have
$$\Delta(\FauxExp_{\mathcal{D}'}(S))=\Delta(S).$$
\end{lem}
\begin{proof} Consider the double sequence
$$a_{n,m}=\Delta(|\Exp_{\mathcal{D}_m'}(S)|^2+\frac1n),\quad n,m\in\mathbb{N}.$$
It follows from Lemma 18 in \cite{DSZ} that $a_{n,m}$ decreases in $m$.
Clearly, it also decreases in $n$.
For every such positive bi-monotone sequence, we have
$$\lim_{n\to\infty}\lim_{m\to\infty}a_{n,m}=\lim_{m\to\infty}\lim_{n\to\infty}a_{n,m}.$$

It is clear that
$$\lim_{n\to\infty}a_{n,m}=\Delta(|\Exp_{\mathcal{D}_m'}(S)|^2)=\Delta(\Exp_{\mathcal{D}_m'}(S))^2.$$
Since each projection $p_1^{(m)}+\cdots+p_k^{(m)},$ $k\in\{1,\ldots,n(m)\}$, is $S$-invariant,
it follows from Theorem \ref{det nu matrix} that
$\Delta(\Exp_{\mathcal{D}_m'}(S))=\Delta(S).$

On the other hand, it follows from Theorem \ref{det converge} that
$$\lim_{m\to\omega}a_{n,m}=\Delta\Big(\pit\Big(\Big\{|\Exp_{\mathcal{D}_m'}(S)|^2+\frac1n\Big\}_{m\geq0}\Big)\Big)=\Delta(|\FauxExp_{\mathcal{D}'}(S)|^2+\frac1n).$$
Hence,
$$\lim_{n\to\infty}\lim_{m\to\omega}a_{n,m}=\Delta(\FauxExp_{\mathcal{D}'}(S))^2.$$
Combining these equalities yields the assertion.
\end{proof}

\begin{thm}\label{brown expectation thm}
The Brown measures of $\FauxExp_{\mathcal{D}'}(S)$ and $S$ are the same.
\end{thm}
\begin{proof}
Of course, for all $\lambda\in\Cpx$, also $S-\lambda$ satisfies Hypothesis~\ref{hyp:Sinv} (for the same projections)
and
$$
\FauxExp_{\mathcal{D}'}(S-\lambda)=\FauxExp_{\mathcal{D}'}(S)-\lambda.
$$
So applying Lemma~\ref{det expectation lemma} to $S-\lambda$ yields
$$
\Delta(\FauxExp_{\mathcal{D}'}(S)-\lambda)=\Delta(S-\lambda).
$$
Now
taking the Laplacian completes the proof.
\end{proof}

The next result belongs here but will be used in Section~\ref{sec:pfMain}.
\begin{lem}\label{lem:FESq}
Suppose a projection $q\in\Mcal$ commutes with every element of $\Dc$ and is a Haagerup-Schultz projection for $S$
and some Borel set $\Bc\subseteq\Cpx$.
Then, identifying $\Mcal$ with the image of the diagonal subalgebra in $\Mcal_\omega$ (see~\S\ref{subsec:ultraP}),
$q$ is a Haagerup-Schultz projection for $\FauxExp_{\Dc'}(S)$ and $\Bc$.
\end{lem}
\begin{proof}
The projection $q$ is invariant under $\Exp_{\Dc_m'}(S)$ for all $m$.
This implies that $q$ is invariant under $\FauxExp_{\Dc'}(S)$ as well.

We consider the abelian subalgebra $\Dc q$ and the increasing sequence $\Dc_mq$ of finite dimensional subalgebras, whose
union is dense in $\Dc q$.
For each $m$ and for each $\lambda\in\Cpx$,
the ordered list of projections obtained from $p_1^{(m)}q,\ldots,p^{(m)}_{n(m)}q$ by dropping those that are equal to zero
is a list of the minimal projections in $\Dc_mq$ and for each $k$, the projection
$p_1^{(m)}q+\cdots+p_k^{(m)}q$ is $(Sq-\lambda q)$-invariant.
Thus, we find that Hypothesis~\ref{hyp:Sinv} is satisfied for $Sq-\lambda q$ and the lists of minimal projections described above.
We let
$$
\FauxExp_{\Dc'q}(Sq-\lambda q)\in\Lc_{\log}(q\Mcal_\omega q,\tau(q)^{-1}\tau_\omega\restrict_{q\Mcal_\omega q})
$$
be the operator resulting from these minimal projections
by the construction in Definition~\ref{def:FauxExp}.
We easily see
$$
\Exp_{\Dc_m'q}(Sq-\lambda q)=\Exp_{\Dc_m'}(S)q-\lambda q
$$
for every $m$ and, thus, we get
\begin{equation}\label{eq:FESq}
\FauxExp_{\Dc'q}(Sq-\lambda q)=\FauxExp_{\Dc'}(S)q-\lambda q.
\end{equation}

Let $\Delta_{(q)}$ denote the Fuglede-Kadison determinant on $q(\Mcal_\omega)q$ associated to the trace
$\frac1{\tau(q)}\tau_\omega\restrict_{q\Mcal_\omega q}$.
Using~\eqref{eq:FESq} and
invoking Lemma~\ref{det expectation lemma}, we have
$$
\Delta_{(q)}(\FauxExp_{\Dc'}(S)q-\lambda q)
=\Delta_{(q)}(\FauxExp_{\Dc'q}(Sq-\lambda q))
=\Delta_{(q)}(Sq-\lambda q)
$$
Taking the Laplacian, we get that the Brown measure of $\FauxExp_{\Dc'}(S)q-\lambda q$ agrees with the Brown measure of
$Sq-\lambda q$.
In a similar manner, we get that the Brown measure of
$(1-q)\FauxExp_{\Dc'}(S)-\lambda(1-q)$ agrees with the Brown measure of
$(1-q)S-\lambda(1-q)$.
\end{proof}

\section{Proof of the main result}
\label{sec:pfMain}

In this section, we prove Theorem \ref{decomposition theorem}.
We will draw on reduction theory for von Neumann algebras and results from~\cite{DNSZ}.

Suppose $(\Mcal,\tau)$ is a tracial von Neumann algebra
where $\Mcal$ has separable predual and is, therefore, countably generated.
Let $T\in\Lc_{\log}(\Mcal,\tau)$.
Let $\rho:[0,\infty)\to\Cpx$ be a surjective, continuous function.
By Theorem~\ref{construction theorem}, there is 
an increasing family $(q_t)_{t\ge0}$ of Haagerup-Schultz projection $q_t=P(T,\rho([0,t)))$ in $(\Mcal_\omega,\tau_\omega)$
for the operator $T$ and the Borel set $\rho([0,t))\subseteq\Cpx$.
Let $\Dc$ be the von Neumann subalgebra of $\Mcal_\omega$ generated by $\{q_t\mid t\ge0\}$.
Then $\Dc$ is the WOT closure of the image of the $*$-homomorphism $\phi:C_0([0,\infty))\to\Mcal_\omega$ given by
$$
\phi(f)=\int_{[0,\infty)}f\,dq_t.
$$
Thus, $\Dc$ is countably generated.
Let $\Mcal_1$ be the von Neumann algebra generated by
$\Dc\cup\Mcal$.
Then $\Mcal_1$ is also countably generated, so has separable predual --- see the discussion in \S\ref{subsec:II1}.

Choose an increasing family $\Dc_1\subseteq\Dc_2\subseteq\cdots$ of finite dimensional $*$-subalgebras of $\Dc$
whose union is weak-operator-topology dense in $\Dc$ and where, for each $m$, $\Dc_m$ is 
the linear span of finitely many of the projections from the set $\{q_t\mid t\ge0\}$.
Then taking differences of these projections, we find minimal projections $p^{(m)}_1,\ldots,p^{(m)}_{n(m)}$
for $\Dc_m$ that span $\Dc_m$ and that satisfy Hypotheses~\ref{hyp:Sinv} with $S$ replaced by $T$ and, of course, $\Mcal$ replaced by $\Mcal_1$.

Now let
\begin{equation}\label{eq:FauxET}
\FauxExp_{\Dc'}(T)\in(\Mcal_1)_\omega
\end{equation}
be the operator constructed as described in Definition~\ref{def:FauxExp}
for this choice of projections $p_k^{(m)}$.
Let $\Mcal_2$ be the von Neumann algebra generated by $\Mcal_1\cup\{\FauxExp_{\Dc'}(T)\}$
and let $\tau_2$ be the restriction of the trace $(\tau_1)_\omega$ to $\Mcal_2$.

Then $\Mcal_2$ is countably generated and, thus, has separable predual.
Taking the standard representation, we have $\Mcal_2\subseteq B(\HEu)$ where $\HEu=L^2(\Mcal_2,\tau_2)$ is a separable Hilbert space.
Since $\HEu$ is separable, there exists (by, for example, Proposition 9.5.3 of~\cite{KR1}) a Borel probability measure $\varsigma$ on $[0,\infty)$
and a $*$-isomorphism $\phibar$ from $L^\infty(\varsigma)$ onto $\Dc$ that extends $\phi$ and is continuous from the weak$^*$-topology on $L^\infty(\varsigma)$
to the weak operator topology on $\Dc$.
Thus, we have $q_t=\phibar(1_{[0,t)})$.

We will use the reduction theory (i.e., the theory of direct integrals) for Hilbert spaces, for von Neumann algebras and for unbounded
operators affiliated to von Neumann algebras.
This theory was developed by many mathematicians over the years;
see Dixmier's book~\cite{Dix81} for the basic theory and Nussbaum~\cite{Nu64} for some aspects of the theory for unbounded operators,
as well as~\cite{DNSZ} for further aspects.

In particular, the Hilbert space $\HEu$ may be written as a direct integral
\begin{equation}\label{eq:Hdirectint}
\HEu=\int^\oplus_{[0,\infty)}\HEu(t)\,d\varsigma(t)
\end{equation}
of Hilbert spaces $\HEu(t)$,
such that $\Dc$ is the set of diagonal operators.
Moreover, for every bounded Borel function $f:[0,\infty)\to\Cpx$, we have
$$
\phibar(f)=\int_{[0,\infty)}f(t)I_{\HEu(t)}\,d\varsigma(t),
$$
where $I_{\HEu(t)}$ denotes the identity operator on $\HEu(t)$.

Consider the diagonal operator
\begin{equation}\label{eq:ExpT}
\Exp_\Dc(T):=\int^\oplus_{[0,\infty)}\rho(t)I_{\HEu(t)}\,d\varsigma(t).
\end{equation}
We use this suggestive notation out of analogy with the situation for bounded operators,
even though (see Appendix~\ref{sec:NoCondExp}), there is little hope of there being an actual conditional expectation
mapping
that is applicable to all operators in $\Lc_{\log}(\Mcal,\tau)$.

Let $\Nc=\Mcal_2\cap\Dc'$ be the relative commutant of $\Dc$ in $\Mcal_2$.
Then
(see~\cite{Dix81} or the section of preliminaries of~\cite{DNSZ})
$\Nc$ is the direct integral
$$
\Nc=\int^\oplus_{[0,\infty)}\Nc(t)\,d\varsigma(t)
$$
of von Neumann algebras $\Nc(t)\subseteq B(\HEu(t))$ and the trace $\tau\restrict_\Nc$
is the direct integral
$$
\tau\restrict_\Nc=\int^\oplus_{[0,\infty)}\tau_t\,d\varsigma(t)
$$
of (for $\varsigma$-almost every $t$) normal, faithful, tracial states $\tau_t$ on $\Nc(t)$.
By Corollary~4 of~\cite{Nu64}, and Proposition 4.4 and Lemma~5.4 of~\cite{DNSZ}, the following holds:
\begin{prop}\label{prop:Sdirectint}
For $S\in\Sc(\Mcal_2,\tau)$, the following are equivalent:
\begin{enumerate}[{\rm (i)}]
\item\label{it:ScomD} $S$ commutes with every element of $\Dc$,
\item $S$ is affiliated to $\Nc$
\item $S$ is decomposable in the sense of Nussbaum~\cite{Nu64} with respect to the direct integral decomposition~\eqref{eq:Hdirectint}
\item\label{it:Sdirectint} $S=\int^\oplus_{[0,\infty)}S(t)\,d\varsigma(t)$,
where $S(t)\in\Sc(\Nc(t),\tau_t)$ for $\varsigma$-almost every $t\in[0,\infty)$.
\end{enumerate}
\end{prop}

Theorem 5.6 of~\cite{DNSZ} yields the following theorem.
\begin{thm}\label{thm:diBrown}
If $S\in\Lc_{\log}(\Mcal_2,\tau)$
and if the conditions of Proposition~\ref{prop:Sdirectint} hold,
then we have $S(t)\in\Lc_{\log}(\Nc(t),\tau_t)$
for $\varsigma$-almost every $t\in[0,\infty)$
and the Brown measure $\nu_S$ of $S$ is given by
$$
\nu_S(B)=\int_{[0,\infty)}\nu_{S(t)}(B)\,d\varsigma(t)
$$
for every Borel subset $B\subseteq\Cpx$,
where $\nu_{S(t)}$ is the Brown measure of $S(t)$ taken with respect to the trace $\tau(t)$ of $\Nc(t)$.
\end{thm}

We can now prove our main result.
\begin{proof}[Proof of Theorem \ref{decomposition theorem}]
For convenience, let us write $R=\FauxExp_{\Dc'}(T)$.
By Lemma \ref{lem:FauxExp commutes}, we have
\begin{equation}\label{eq:RinD'}
R\in\Lc_{\log}(\Mcal_2,\tau_2)\cap\Dc'.
\end{equation}
We will show:
\begin{enumerate}[{\rm (a)}]
\item\label{it:TR} $T$ and $R$ have the same Brown measure,
\item\label{it:ExpexpL1} $\Exp_\Dc(T)\in\Lc_{\log}(\Mcal_2,\tau_2)$,
\item\label{it:RN} $R$ and $\Exp_\Dc(T)$ have the same Brown measure,
\item\label{it:R-N0} the Brown measure of $R-\Exp_\Dc(T)$ is $\delta_0$.
\item\label{it:TR-N} $T-\Exp_\Dc(T)$ and $R-\Exp_\Dc(T)$ have the same Brown measure
\end{enumerate}
Then, taking $N=\Exp_\Dc(T)$ and $Q=T-N$, these facts provide proof of the theorem.

The truth of~\eqref{it:TR} follows directly from Theorem~\ref{brown expectation thm}.

From~\eqref{eq:RinD'} and Proposition~\ref{prop:Sdirectint}, we get a direct integral decomposition
\begin{equation}\label{eq:Rdi}
R=\int^\oplus_{[0,\infty)}R(t)\,d\varsigma(t),
\end{equation}
where $R(t)\in\Lc_{\log}(\Nc(t),\tau_t)$ for $\varsigma$-almost every $t$.
We will now show
\begin{equation}\label{eq:nuRt}
\nu_{R(t)}=\delta_{\rho(t)}\text{ for $\varsigma$-almost every }t\in[0,\infty).
\end{equation}
By Proposition~\ref{lem:FESq}, for every $t\in[0,\infty)$, $q_t$ is a Haagerup-Schultz projection for
$R$ and the set $\rho([0,t])$.
Suppose $0\le t(1)<t(2)<\infty$ and $q_{t(1)}\ne q_{t(2)}$.
Then the Brown measure of $Rq_{t(2)}$ is the renormalized restriction of $\nu_R$ to $\rho([0,t(2)))$.
Since $q_{t(1)}$ is $Rq_{t(2)}$-invariant and either $q_{t(1)}=0$ or the Brown measure of $Rq_{t(1)}$ is
the renormalized restriction of $\nu_R$ to $\rho([0,t(1)))$, it follows from Theorem~\ref{det nu matrix}
that the Brown measure of $(q_{t(2)}-q_{t(1)})R(q_{t(2)}-q_{t(1)})$ is the renormalized
restriction of $\nu_R$ to $\rho([0,t(2)))\backslash\rho([0,t(1)))\subseteq\rho([t(1),t(2)))$.
However the direct integral decomposition of $R$ restricts to give the direct integral decomposition
$$
(q_{t(2)}-q_{t(1)})R(q_{t(2)}-q_{t(1)})=\int^\oplus_{[t(1),t(2))}R(t)\,d\varsigma(t)
$$
and using Theorem~\ref{thm:diBrown}, we have
$$
\nu_{(q_{t(2)}-q_{t(1)})R(q_{t(2)}-q_{t(1)})}(B)=\frac1{\varsigma((t(1),t(2)])}\int^\oplus_{[t(1),t(2))}\nu_{R(t)}(B)\,d\varsigma(t)
$$
for every Borel set $B$.
In particular, the measure $\nu_{R(t)}$ is concentrated in $\rho([t(1),t(2)))$ for almost every $t\in[t(1),t(2))$.

Thus, we find a $\varsigma$-null set $N\subseteq(0,\infty)$ so that 
for all $t\in N^c$ and for all rational $t(1)$ and $t(2)$ with $0\le t(1)\le t<t(2)$,
the Brown measure $\nu_{R(t)}$ is concentrated in $\rho([t(1),t(2)))$.
By continuity of $\rho$, we have $\nu_{R(t)}=\delta_{\rho(t)}$ for all $t\in N^c$.
This proves~\eqref{eq:nuRt}.

From~\eqref{eq:nuRt}, we get~\eqref{it:ExpexpL1}, because we have
\begin{multline}\label{eq:taulogRbd}
\tau(\log^+\big(|\Exp_\Dc(T)|\big))=\int_{[0,\infty)}\log^+(|\rho(t)|)\,d\varsigma(t) \\
=\int_{[0,\infty)}\int\log^+(|z|)\,d\nu_{R(t)}(z)\,d\varsigma(t)=\int\log^+(|z|)\,d\nu_R(z)\le\tau(\log^+(|R|)),
\end{multline}
Indeed, the first equality in~\eqref{eq:taulogRbd} is from Theorem~\ref{thm:diBrown} and the direct integral decomposition~\eqref{eq:ExpT},
the last equality in~\eqref{eq:taulogRbd} is from Theorem~\ref{thm:diBrown} applied to~\eqref{eq:Rdi}
and approximation of the function $z\mapsto\log^+(|z|)$ from below by simple functions, together with the Monotone Convergence Theorem;
the inequality in~\eqref{eq:taulogRbd} is from Lemma 2.20 of~\cite{HS1}.

By Theorem~\ref{thm:diBrown} applied to $R$ and $\Exp_\Dc(T)$ and using~\eqref{eq:nuRt},
we will have, for every Borel set $\Bc\subseteq\Cpx$,
$$
\nu_R(\Bc)=\int_{[0,\infty)}\nu_{R(t)}(\Bc)\,d\varsigma(t)=\int_{[0,\infty)}\delta_{\rho(t)}(\Bc)\,d\varsigma(t)
=\nu_{\Exp_\Dc(T)}(\Bc).
$$
This proves~\eqref{it:RN}.

Similarly, from~\eqref{eq:nuRt} we have $\nu_{R(t)-\rho(t)I_\HEu(t)}=\delta_0$ for $\varsigma$-almost every $t$, and,
by invoking Theorem~\ref{thm:diBrown} again,
since
$$
R-\Exp_\Dc(T)=\int^\oplus_{[0,\infty)}\big(R(t)-\rho(t)I_{\HEu(t)}\big)\,d\varsigma(t),
$$
we get
$$
\nu_{R-\Exp_\Dc(T)}(B)=\int_{[0,\infty)}\nu_{R(t)-\rho(t)I_{\HEu(t)}}(B)\,d\varsigma(t)=\delta_0(B)
$$
for every Borel set $B\subseteq\Cpx$.
This proves~\eqref{it:R-N0}

To show~\eqref{it:TR-N}, note that, since $\Exp_\Dc(T)$ commutes with $q_t$, the projection $q_t$ is $(T-\Exp_\Dc(T))$-invariant for all $t$.
This implies that the projections $p_k^{(m)}$ described above Equation~\eqref{eq:FauxET} satisfy Hypothesis~\ref{hyp:Sinv}
when $S$ is replaced by $T-\Exp_\Dc(T)$.
Therefore, by Theorem~\ref{brown expectation thm}, the Brown measures
of $T-\Exp_\Dc(T)$ and $\FauxExp(T-\Exp_\Dc(T))$ are the same.
But we clearly have
$$
\Exp_{\Dc_m'}\big(T-\Exp_\Dc(T)\big)=\Exp_{\Dc_m'}(T)-\Exp_\Dc(T),
$$
where $\Exp_{\Dc_m'}$ is the conditional expectation described at~\eqref{eq:ExpDm'} at the start of Section~\ref{sec:FauxExp}.
Applying $\pit$ shows
$$
\FauxExp(T-\Exp_\Dc(T))=R-\Exp_\Dc(T),
$$
This finishes the proof of~\eqref{it:TR-N}.
\end{proof}

\section{Spectrality of traces}
\label{sec:spectralityTr}

The goal of this section is to prove Theorem~\ref{spectral trace thm}.
See the last paragraph of Section~\ref{sec:strategy} for a description of the strategy.

For every $T\in \Lc_{\log}(\Mcal,\tau),$ we set
$$
\Phi(s,T)=\int_{|z|\leq s}zd\nu_T(z),\quad s>0.
$$

\begin{lem}\label{dk hardest estimate}
If $Q\in\Lc_{\log}(\Mcal,\tau)$ is such that $\nu_Q=\delta_0,$ then, for every $s>0,$ we have
$$
\Big|\int_0^{2\pi}\Re(\Phi(s,Q+e^{i\theta}Q^*))d\theta\Big|\leq 
400\pi s\,\tau(\log^+(\frac{2e|Q|}{s}))
$$
\end{lem}
\begin{proof} Let $u\in C^2(\mathbb{C})$ be a subharmonic function
constructed by Kalton \cite{Kalton} (see Lemma 5.5.2 in \cite{LSZ} for details)
that satisfies
\begin{align}
|u(z)+\Re(z)1_{(1,\infty)}(|z|)|\leq 100\log^+(e|z|),\quad&(z\in\mathbb{C}) \label{dkh0} \\
u(z)=0,\quad&(|z|<1) \notag \\
u(z)=c_1+c_2\log(|z|)-\Re(z),\quad&(|z|>e). \notag
\end{align}
Thus,  $u$ is harmonic outside of the annulus with radii $1$ and $e$.

Fix $s>0$.
The function $z\mapsto\Re(z)+su(\frac{z}{s}),$ $z\in\mathbb{C}$ satisfies the 
conditions of Lemma 6.4 in \cite{DK-fourier},
except that it does not vanish in a neighborhood of $0$;
however, this is inessential because it is harmonic in a neighborhood of $0$.
For every $T\in \Lc_{\log}(\Mcal,\tau),$ we set
$$\Psi_s(T)=\int_{\mathbb{C}}\Big(\Re(z)+su(\frac{z}{s})\Big)d\nu_T(z),\quad \Theta_s(T)=\int_{\mathbb{C}}\Big(\Re(z)1_{(s,\infty)}(|z|)+su(\frac{z}{s})\Big)d\nu_T(z),$$
so that $\Psi_s(T)-\Theta_s(T)=\Re\Phi(s,T)$.
The integrals are convergent for every $T\in \Lc_{\log}(\Mcal,\tau)$.
Indeed, using~\eqref{dkh0}
and Lemma 2.20 in \cite{HS1}, we infer that
\begin{multline}\label{dkh4}
|\Theta_s(T)|\leq100s\int_{\mathbb{C}}\log^+(\frac{e|z|}{s})\,d\nu_T(z)
=100s\int_\Cpx\log^+(|w|)\,d\nu_{\frac{eT}s}(w) \\
\leq 100s\,\tau(\log^+(\frac{e|T|}{s})).
\end{multline}

Taking into account that $\nu_Q=\delta_0$ and using Lemma 6.4 in \cite{DK-fourier}, we infer that
$$0=\Psi_s(Q)\leq\frac1{2\pi}\int_0^{2\pi}\Psi_s(Q+e^{i\theta}Q^*)d\theta.$$
Thus,
\begin{equation}\label{dkh1}
0\leq \int_0^{2\pi}\Re(\Phi(s,Q+e^{i\theta}Q^*))d\theta+\int_0^{2\pi}\Theta_s(Q+e^{i\theta}Q^*)d\theta.
\end{equation}
By~\eqref{dkh4} and Lemma \ref{lem:elemlog+}, we have
\begin{equation}\label{dkh2}
|\Theta_s(Q+e^{i\theta}Q^*)|\leq 100s\,\tau(\log^+(\frac{e|Q+e^{i\theta}Q^*|}{s}))
\le 200s\,\tau(\log^+(\frac{2e|Q|}{s})).
\end{equation}
Combining \eqref{dkh1} and \eqref{dkh2}, we infer that
\begin{equation}\label{dkh3}
\int_0^{2\pi}\Re(\Phi(s,Q+e^{i\theta}Q^*))d\theta\geq -400\pi s\,\tau(\log^+(\frac{2e|Q|}{s})).
\end{equation}
Substituting $-Q$ instead of $Q$ in \eqref{dkh3} and combining with \eqref{dkh3}, we finish the proof.
\end{proof}

\begin{lem}\label{dk improved}
If $Q\in\Lc_{\log}(\Mcal,\tau)$ is such that $\nu_Q=\delta_0,$ then, for every $s>0,$ we have
$$|\tau(AE_{|A|}[0,s])|,\,|\tau(BE_{|B|}[0,s])|\leq300s\,\tau(\log^+\left(\frac{2e|Q|}s\right)),$$
where $A=\Re(Q)$ and $B=\Im(Q)$.
\end{lem}
\begin{proof} A direct verification shows that $Q+e^{i\theta}Q^*$ is normal and that
$$\Re(Q+e^{i\theta}Q^*)=(1+\cos(\theta))A+\sin(\theta)B.$$
Setting $T=Q+e^{i\theta}Q^*$ in Proposition 6.2 (3) in \cite{DK-fourier} and letting $r\to0$, we infer
\begin{multline*}
|\Re(\Phi(s,Q+e^{i\theta}Q^*))-\Phi(s,(1+\cos(\theta))A+\sin(\theta)B)| \\
\leq s\,\tau(E_{|Q+e^{i\theta}Q^*|}((s,\infty)))
\le s\,\tau(\log^+\left(\frac{e|Q+e^{i\theta}Q^*|}s\right))\le 2s\,\tau(\log^+\left(\frac{2e|Q|}s\right)).
\end{multline*}
where for the penultimate inequality above, we have used $1_{(s,\infty)}(t)\le\log^+(\frac{et}s)$ and for the last inequality we have again used  Lemma~\ref{lem:elemlog+}.

From Proposition 6.2(1) of \cite{DK-fourier} and $\Phi(s,-T)=-\Phi(s,T)$, we get
\begin{align*}
\big|\Phi&(s,(1+\cos\theta)A+(\sin\theta)B)-\Phi(s,A)-\Phi(s,(\cos\theta)A)-\Phi(s,(\sin\theta)B)\big| \\
&\le8s\tau\big(E_{|(1+\cos\theta)A+(\sin\theta)B|}(s,\infty)+E_{|A|}(s,\infty)
\begin{aligned}[t]&+E_{|(\cos\theta)A|}(s,\infty) \\
&\quad+E_{|(\sin\theta)B|}(s,\infty)\big)\end{aligned} \\
&\le8s\tau\big(E_{|(1+\cos\theta)A+(\sin\theta)B|}(s,\infty)+2E_{|A|}(s,\infty)+E_{|B|}(s,\infty)\big)
\end{align*}
while  from Proposition 6.2(2) of \cite{DK-fourier} we have
\begin{align*}
\big|\Phi(s,(\cos\theta)A)-(\cos\theta)\Phi(s,A)\big|&\le s\tau(E_{|A|}(s,\infty)) \\
\big|\Phi(s,(\sin\theta)B)-(\sin\theta)\Phi(s,B)\big|&\le s\tau(E_{|B|}(s,\infty)).
\end{align*}
So we get
\begin{align*}
|\Phi&(s,(1+\cos(\theta))A+\sin(\theta)B)-(1+\cos(\theta))\cdot\Phi(s,A)-\sin(\theta)\cdot\Phi(s,B)| \\
&\le s\tau\big(8E_{|(1+\cos\theta)A+(\sin\theta)B|}(s,\infty)+17E_{|A|}(s,\infty)+9E_{|B|}(s,\infty)\big) \\
&\le s\tau\big(8\log^+\left(e\left|\frac{(1+e^{-i\theta})Q+(1+e^{i\theta})Q^*}{2s}\right|\right)
\begin{aligned}[t]&+17\log^+\left(e\left|\frac{Q+Q^*}{2s}\right|\right) \\
&\quad+9\log^+\left(e\left|\frac{Q-Q^*}{2s}\right|\right)\big)\end{aligned} \\
&\le88s\tau(\log^+\left(\frac{2e|Q|}s\right)),
\end{align*}
where for the last inequality we used Lemma~\ref{lem:elemlog+}.

Combining these inequalities, we infer
$$|\Re(\Phi(s,Q+e^{i\theta}Q^*))-(1+\cos(\theta))\cdot\Phi(s,A)-\sin(\theta)\cdot\Phi(s,B)|
\leq90s\,\tau(\log^+\left(\frac{2e|Q|}s\right)).$$
From this and Lemma \ref{dk hardest estimate} we get
$$\Big|\int_0^{2\pi}\Big((1+\cos(\theta))\cdot\Phi(s,A)+\sin(\theta)\cdot\Phi(s,B)\Big)d\theta\Big|
\leq600\pi s\,\tau(\log^+\left(\frac{2e|Q|}s\right)).$$
Computing the integral and taking into account that the Brown measure of a self adjoint operator is the trace of its spectral measure, we conclude
$$|\tau(AE_{|A|}[0,s])|=|\Phi(s,A)|\leq300s\,\tau(\log^+\left(\frac{2e|Q|}s\right)).$$
The other inequality follows by replacing $Q$ with $iQ.$
\end{proof}

The following (non-linear) operator is a close analogue of the operator defined in formula (9) of \cite{SZ-AiM}.
Given $T\in\Sc(\Mcal,\tau)$ we consider the function $\mathbf{S}(T)$ on $(0,1)$ defined by
\begin{equation}\label{bolds def}
\mathbf{S}(t,T)=\mu(t,T)\Big(1+\frac1t\tau(\log^+(\frac{|T|}{\mu(t,T)}))\Big),\quad t\in(0,1).
\end{equation}

For a decreasing sequence
 $a=(a(0),a(1),\ldots,a(p-1))$, we let $\sigma_2a$ be its
dilation by a factor of $2$,
namely, the sequence $\{a(\lfloor\frac k2\rfloor)\}_{k=0}^{p-1}$.
For a decreasing function $x:(0,1]\to[0,\infty)$, we let $\sigma_2x$ be the function $\sigma_2x(t)=x(t/2)$, $t\in(0,1)$.

\begin{lem}\label{sz similar} For every $T\in\Lc_{\log}(\Mcal,\tau),$ we have $\mathbf{S}(T)\prec\prec_{\log}4\sigma_2\mu(T).$
\end{lem}
\begin{proof} Fix $n\in\mathbb{N}$ and consider sequences $a_n,b_n$ of length $2^n$ defined by the setting
$$a_n(k)=\mu(\frac{k+1}{2^n},T),\quad b_n(k)=a_n(k)(1+\frac{1}{k+1}\sum_{i=0}^k\log(\frac{a_n(i)}{a_n(k)})),\quad 0\leq k<2^n.$$
Clearly, for each $n$ and $k>0$, $a_n(k)\le a_n(k-1)$, i.e., $a_n$ is decreasing.
Thus, $b_n(k)\ge0$ for all $k$.
It follows from Lemma 21 in \cite{SZ-AiM} that also $b_n$ is decreasing.
Define decreasing functions $x_n$ and $y_n$ on $(0,1]$ by setting
$$x_n(t)=a_n(k),\quad y_n(t)=b_n(k),\quad t\in\left(\frac{k}{2^n},\frac{k+1}{2^n}\right],\quad 0\leq k<2^n.$$
Clearly, $x_n\le\mu(T)$.
It is clear that $\lim_{n\to\infty}y_n=\mathbf{S}(T)$ almost everywhere.
Thus, $\mathbf{S}(T)$ is a decreasing function.
Since
$b_n(k)\leq\mathbf{S}(\frac{k+1}{2^n},T)$,
it follows that $0\le y_n\leq\mathbf{S}(T).$

We claim that
\begin{equation}\label{sz similar main}
\int_0^t\log(y_n(s))ds\leq 8\int_0^{t/2}\log(x_n(s)))ds,\quad t>0.
\end{equation}
Since both sides of the inequality~\eqref{sz similar main} are piecewise linear in $t$,
it suffices to verify \eqref{sz similar main} for $t=\frac{k+1}{2^n},$ $0\leq k<2^n$.
It follows from Theorem 24 in \cite{SZ-AiM} that $b_n\prec\prec_{\log}4\sigma_2a_n$.
Hence, we have
$$\int_0^t\log(y_n(s))ds=2^{-n}\sum_{i=0}^k\log(b_n(i))\leq 2^{2-n}\sum_{i=0}^k\log((\sigma_2a_n)(i)).$$
If $k=2m+1,$ then
$$\sum_{i=0}^k\log((\sigma_2a_n)(i))=2\sum_{i=0}^m\log(a_n(i))=2^{n+1}\int_0^{t/2}\log(x_n(s))ds.$$
If $k=2m+2,$ then
$$\sum_{i=0}^k\log((\sigma_2a_n)(i))=\sum_{i=0}^m\log(a_n(i))+\sum_{i=0}^{m+1}\log(a_n(i))=2^{n+1}\int_0^{t/2}\log(x_n(s))ds.$$
This proves the claim~\eqref{sz similar main}.

From the claim and the fact that $x_n\le\mu(T)$, we have
$$\int_0^t\log(y_n(s))ds\leq 8\int_0^{t/2}\log(\mu(s,T))ds,\quad t>0.$$
Invoking Fatou's Lemma we get
$$
\int_0^t\log(\mu(s,\mathbf{S}(T)))ds\le\liminf_{n\to\infty}\int_0^t\log(y_n(s))ds
\leq 8\int_0^{t/2}\log(\mu(s,T))ds,
$$
which proves the lemma.
\end{proof}

By the {\em distribution} of a normal operator, we mean
the trace composed with spectral measure of the operator,
(which, of course, coincides with its Brown measure).
\begin{lem}\label{normal phi lemma}
Let $\Mcal$ be a finite factor, let $\Bfr(\Mcal,\tau)\subseteq\Sc(\Mcal,\tau)$
be an $\Mcal$-bimodule and let $\varphi:\Bfr\to\Cpx$ be a trace.
If $N\in\mathfrak{B}(\Mcal,\tau)$ is a normal operator, then $\varphi(N)$ depends only on the distribution of $N$.
\end{lem}
\begin{proof} By Theorem 2.3 in \cite{FH}, we have
$\varphi|_{\Mcal}=c_{\varphi}\tau|_{\Mcal}$ for a constant $c_\varphi$.

Suppose $N_1,N_2\in\mathfrak{B}(\Mcal,\tau)$ are normal operators and have the same distribution.
Let $\lfloor\cdot\rfloor$ be the integer part and let $\{\cdot\}$ be a fractional part.
We set $f(\lambda)=\lfloor\Re(\lambda)\rfloor+i\lfloor\Im(\lambda)\rfloor$
and $g(\lambda)=\{\Re(\lambda)\}+i\{\Im(\lambda)\},$ $\lambda\in\mathbb{C}$.
Clearly, $N_1=f(N_1)+g(N_1)$ and $N_2=f(N_2)+g(N_2)$.
By the preceding paragraph,
$$\varphi(N_1)=\varphi(f(N_1))+c_{\varphi}\tau(g(N_1)),\quad \varphi(N_2)=\varphi(f(N_2))+c_{\varphi}\tau(g(N_2)).$$
Normal operators $f(N_1)$ and $f(N_2)$ have only discrete spectrum and have the same distribution. Since $\Mcal$ is a factor, these operators are unitarily equivalent.
Thus, $\varphi(f(N_1))=\varphi(f(N_2))$.
Since $g(N_1)$ and $g(N_2)$ have the same distribution, we immediately see $\tau(g(N_1))=\tau(g(N_2))$.
This concludes the proof.
\end{proof}

\begin{proof}[Proof of Theorem \ref{spectral trace thm}]
Let $T=N+Q$ be the decomposition realized in Theorem \ref{decomposition theorem}
in the larger tracial von Neumann algebra $(\Mcal_2,\tau_2)$.
Thus, $N$ is normal with $\nu_N=\nu_T$ and $\nu_Q=\delta_0$.
By enlarging $\Mcal_2$ if necessary
(e.g., by taking the free product of $\Mcal_2$ with a diffuse tracial von Neumann algebra~\cite{D94})
we may assume $\Mcal_2$ is a II$_1$-factor.
We consider the extension of $\varphi$ to $\Bfr(\Mcal_2,\tau_2)$ as described in Theorem~\ref{ks extension thm}.

Since $\nu_N=\nu_T,$ it follows from Lemma 2.20 in \cite{HS1} that
$$\tau(\log^+(\frac{|N|}{t}))\leq\tau(\log^+(\frac{|T|}{t})),\quad t>0.$$
For a given $u\in(0,1),$ set $t=\mu(u,T).$ We have
\begin{multline*}
\int_0^u\log(\frac{\mu(s,N)}{t})ds
\le\int_0^u\log^+(\frac{\mu(s,N)}{t})ds
\leq\tau(\log^+(\frac{|N|}{t})) \\
\leq\tau(\log^+(\frac{|T|}{t}))=\int_0^u\log(\frac{\mu(s,T)}{t})ds,
\end{multline*}
where the equality above is becuase $\mu(u,T)=t$.
Hence, $N\prec\prec_{\log}T.$
Thus, since $T\in\mathfrak{B}(\Mcal,\tau)$, we have $N\in\mathfrak{B}(\Mcal_2,\tau_2)$,
because the bimodule is assumed to be closed with respect to logarithmic submajorization.
Hence, $Q=T-N\in\mathfrak{B}(\Mcal_2,\tau_2).$

We claim that $\nu_Q=\delta_0$ implies $\varphi(Q)=0$.
Set $A=\Re(Q)$ and $B=\Im(Q)$.
Note that
$$\mu(t,A)=\mu(t,\frac{Q+Q^*}{2})\leq\frac12\mu(\frac{t}{2},Q)+\frac12\mu(\frac{t}{2},Q^*)=\mu(\frac{t}{2},Q)\leq\mu(\frac{t}{2},2eQ).$$
Setting $s=\mu(\frac{t}{2},2eQ)$ in Lemma \ref{dk improved}, we arrive at
$$
|\tau(AE_{|A|}[0,s)|\leq600s\tau(\log^+(\frac{2e|Q|}{s}))
\stackrel{\eqref{bolds def}}{\leq}300t\mathbf{S}(\frac{t}{2},2eQ).$$
Thus,
\begin{multline*}
|\tau(AE_{|A|}[0,\mu(t,A)])|\leq|\tau(AE_{|A|}[0,\mu(\frac{t}{2},2eQ)])|+|\tau(AE_{|A|}(\mu(t,A),\mu(\frac{t}{2},2eQ)])| \\
\leq 300t\mathbf{S}(\frac{t}{2},2eQ)+\mu(\frac{t}{2},2eQ)\tau(E_{|A|}(\mu(t,A),\mu(\frac{t}{2},2eQ)]) \\
\leq 300t\mathbf{S}(\frac{t}{2},2eQ)+t\mu(\frac{t}{2},2eQ)\leq301t\mathbf{S}(\frac{t}{2},2eQ),
\end{multline*}
where the last inequality follows from the definition of the operator $\mathbf{S}$.
Taking into account that $Q\in\mathfrak{B}(\Mcal_2,\tau_2)$ and that bimodule $\mathfrak{B}(\Mcal_2,\tau_2)$ is closed with respect to logarithmic submajorization
we infer from Lemma \ref{sz similar} that $\mathbf{S}(2eQ)\in \mathfrak{B}(\Mcal_2,\tau_2)$
(or, rather, that this is the singular number function of an element of $\mathfrak{B}(\Mcal_2,\tau_2)$).
It follows now from Theorem 4.6 in \cite{DK-fourier} that $\Re(Q)=A\in[\mathfrak{B}(\Mcal_2,\tau_2),\Mcal_2].$ Thus, $\varphi(\Re(Q))=0.$ Similarly, we have $\varphi(\Im(Q))=0.$
Thus, $\varphi(Q)=0$.

It follows from above that $\varphi(T)=\varphi(N)$.
By Lemma \ref{normal phi lemma}, $\varphi(N)$ depends on the spectral measure of $N,$ that is, on the Brown measure of $T$.
\end{proof}

We prove the following result in order to distinguish Theorem~\ref{spectral trace thm} from the results of~\cite{DK-fourier}.

\begin{lem} There are plenty of bimodules which are closed with respect to logarithmic submajorization and which fail to be geometrically stable in the sense of \cite{DK-fourier}.
\end{lem}
\begin{proof} Consider bimodule $\mathfrak{B}_0$ which is closed with respect to the Hardy-Littlewood submajorization and which fails to be Cesaro-invariant. There is a large class of such bimodules.

An easy example can be found among Lorentz spaces.
$$\mathfrak{M}_{\psi}(\mathcal{M},\tau)=\Big\{A\in\mathcal{S}(\mathcal{M},\tau)\;\Big|\; \sup_{t\in(0,1)}\frac{1}{\psi(t)}\int_0^t\mu(s,A)ds<\infty\Big\}.$$
It is immediate that $\mathfrak{M}_{\psi}$ is closed with respect to the Hardy-Littlewood submajorization. However, the condition
$$\liminf_{t\to0}\frac{t\psi'(t)}{\psi(t)}=0$$
guarantees the failure of condition (6.24) in Chapter II in \cite{KPS}. By Theorem II.6.6 in \cite{KPS}, we have that $\mathfrak{M}_{\psi}$ is not Cesaro invariant.

We now set
$$\mathfrak{B}(\mathcal{M},\tau)=\{A\in\mathcal{S}(\mathcal{M},\tau):\ \log^+(|A|)\in\mathfrak{B}_0(\mathcal{M},\tau)\}.$$
Since
$$\mu(A+B)\leq\sigma_2\mu(A)+\sigma_2\mu(B)\leq 2\max\{\sigma_2\mu(A),\sigma_2\mu(B)\},$$
it follows that
$$\log_+(\mu(A+B))\leq 2\log_+(\sigma_2\mu(A))+2\log_+(\sigma_2\mu(B)).$$
This proves that $\mathfrak{B}(\mathcal{M},\tau)$ is a linear subspace in $\mathcal{S}(\mathcal{M},\tau).$ It is now immediate that $\mathfrak{B}(\mathcal{M},\tau)$ is a bimodule over $\mathcal{M}.$

We claim that $\mathfrak{B}(\mathcal{M},\tau)$ is closed with respect to logarithmic submajorization. Indeed, if $A\in \mathfrak{B}(\mathcal{M},\tau)$ and $B\in \mathcal{S}(\mathcal{M},\tau)$ are such that $B\prec\prec_{\log}A,$ then $\log^+(|B|)\prec\prec\log^+(|A|).$ Since $\mathfrak{B}_0(\mathcal{M},\tau)$ is closed with respect to the Hardy-Littlewood submajorization and since $\log^+(|A|)\in\mathfrak{B}_0(\mathcal{M},\tau),$ it follows that also $\log^+(|B|)\in\mathfrak{B}_0(\mathcal{M},\tau).$ This proves the claim.

Now, take $x=\mu(x)\in\mathfrak{B}_0$ such that $Cx\notin \mathfrak{B}_0.$ Taking $A\in\mathfrak{B}(\mathcal{M},\tau)$ such that $\mu(A)=\exp(x)$ and setting
$$\mu(t,B)=\exp(\frac1t\int_0^t\log(\mu(s,A))ds),\quad t\in(0,1),$$
we conclude that $\log_+(|B|)=\log(|B|)\notin\mathfrak{B}_0(\mathcal{M},\tau).$ Thus, $\mathfrak{B}(\mathcal{M},\tau)$ is not geometrically stable.
\end{proof}

\appendix

\section{Hyperinvariant subspaces of unbounded operators}
\label{app:hyperinv}

The Haagerup-Schultz projections of a {\em bounded} operator $T\in\Mcal\subseteq B(\HEu)$, as found by Haagerup and Schultz in~\cite{HS2},
are, as far as we know,
better than those found in Section~\ref{hs construct} for $T\in\Lc_{\log}(\Mcal,\tau)$.
Indeed, Haagerup and Schultz's operators lie in $\Mcal$ itself, they are unique and they are $T$-hyperinvariant, namely,
they are invariant under all operators in $B(\HEu)$ that commute with $T$.
A key aspect of Haagerup and Schultz's proof of their result was an alternative characterization of their Haagerup--Schultz projections
corresponding to closed balls in the complex plane, whereby it is clear that the projections are $T$-hyperinvariant.
In this appendix, we explore the notion of hyperinvariant subspaces for unbounded operators on Hilbert space,
in order to point out in what way we are not (yet) able to extend the full results of Haagerup an Schultz to unbounded operators.
We have more questions than results of note.

Let $T$ be a closed, possibly unbounded operator on Hilbert space $\HEu$.
Let $\Dc(T)$ denote the domain of $T$ and let $T=V|T|$ be the polar decomposition of $T$.
The {\em von Neumann algebra generated by $T$} is the von Neumann algebra generated by the set consisting of $V$ together with
all the spectral projections of $|T|$;
we denote this von Neumann algebra by $W^*(T)$.
We say that a bounded operator $S$ on $\HEu$ {\em permutes} with $T$ (this seems to be standard, though old-fashioned, notation)
if $ST\subseteq TS$, namely,
if $S(\Dc(T))\subseteq\Dc(T)$ and $TSx=STx$ for all $x\in\Dc(T)$.
For future use, we make an easy (and well known) observation:
\begin{lem}\label{lem:permutes}
If a bounded operator $S$ permutes with $T$, then $S$ permutes with $T^n$ for every $n\in\Nats$.
\end{lem}
\begin{proof}
We use induction on $n$.
The case $n=1$ is a tautology.
Assume $n\ge2$.
If $x\in\Dc(T^n)$, then $x\in\Dc(T^{n-1})$ and by the induction hypothesis $Sx\in\Dc(T^{n-1})$ and $T^{n-1}Sx=ST^{n-1}x$.
But $T^{n-1}x\in\Dc(T)$, so $ST^{n-1}x\in\Dc(T)$ and
$$
T(T^{n-1}Sx)=TST^{n-1}x=ST(T^{n-1}x).
$$
Thus, $Sx\in\Dc(T^n)$ and $T^nSx=ST^nx$.
\end{proof}

\begin{prop}\label{prop:W*(T)'}
If $S\in B(\HEu)$ belongs to the commutant of $W^*(T)$, then $S$ permutes with $T$.
\end{prop}
\begin{proof}
A vector
$x\in\HEu$ belongs to $\Dc(T)$ if and only if $\lim_{r\to\infty}|T|E_{|T|}([0,r])x$ converges in $\HEu$, where $E_{|T|}([0,r])$ denotes
the spectral projection of $|T|$ for the interval $[0,r]$, and then the value of the limit is $|T|x$.
Since $S$ commutes with each $|T|E_{|T|}([0,r])$, we have, for $x\in\Dc(T)$,
$|T|E_{|T|}([0,r])Sx=S|T|E_{|T|}([0,r])x$, and, thus $Sx\in\Dc(T)$ and $|T|Sx=S|T|x$.
Since $S$ also commutes with $V$, we have $TSx=STx$.
\end{proof}

\begin{rem}
Let $\Gfr_T$ denote the graph of $T$:
$$
\Gfr_T=\{(x,Tx)\mid x\in\Dc(T)\}\subseteq\HEu\oplus\HEu.
$$
Let $S\in B(\HEu)$.
Then $T$ permutes with $S$ iff $S\oplus S$ leaves $\Gfr_T$ invariant.
\end{rem}

\begin{defi}
Let $\Vc$ be a closed subspace of $\HEu$.
We say that $\Vc$ is {\em $T$-hyperinvariant} if $\Vc$ is invariant under every bounded operator $S$ that permutes with $T$.
\end{defi}

\begin{prop}
Let $\Vc$ be a $T$-hyperinvariant subspace and let $P$ be the orthogonal projection from $\HEu$ onto $\Vc$.
Then $P\in W^*(T)$.
\end{prop}
\begin{proof}
By Proposition~\ref{prop:W*(T)'}, for every $S\in W^*(T)'$, we have $SP=PSP$.
Taking adjoints, we have $PS^*=PS^*P$.
Since $W^*(T)$ is closed under taking adjoints, we conclude $SP=PS$ for every $S\in W^*(T)'$.
Thus, $P\in W^*(T)''=W^*(T)$.
\end{proof}

In~\cite{AV03}, Albrecht and Vasilescu defined a closed subspace $\Lc$ of $\HEu$ to be {\em invariant} under $T$
(and we will also say $T$-invariant for this) if 
$$
\Dc_0(T,\Lc):=\Dc(T)\cap\Lc
$$
is dense in $\Lc$ and $T(\Dc_0(T,\Lc))\subseteq\Lc$.
They observed that in this case, the restriction of $T$ to $\Dc_0(T,\Lc)$ yields a closed operator on $\Lc$.
They also defined $\Lc$ to be {\em quasi-invariant} under $T$ if
$$
\Dc(T;\Lc):=\{x\in\Dc_0(T,\Lc)\mid Tx\in\Lc\}
$$
is dense in $\Lc$.
They observed that in this case, the restriction of $T$ to $\Dc(T,\Lc)$ yields a closed operator on $\Lc$.

Clearly, if $\Lc$ is invariant under $T$, then it is quasi-invariant, and the reverse implication holds
if $T$ is bounded.
Albrecht and Vasilescu found an easy example of an unbounded $T$ and a $T$-quasi-invariant subspace that is not $T$-invariant.

\begin{ques}
If $P$ is $T$-hyperinvariant, must the subspace $P\HEu$ be $T$-invariant (or $T$-quasi-invariant)?
What if $W^*(T)$ is a finite von Neumann algebra?
\end{ques}

\begin{prop}
Let $\Mcal\subseteq B(\HEu)$ be a finite von Neumann algebra.
If $T$ is an unbounded operator, if $T$ is affiliated to $\Mcal$ and if $p\in\Mcal$ is a projection, then $p(\HEu)\cap\Dc(T)$ is dense in $p(\HEu)$.
\end{prop}
\begin{proof}
Let $\tau$ be a normal, faithful tracial state on $\Mcal$ and let $k\in\Nats$.
There is a projection $q_k\in W^*(T)\subseteq\Mcal$ such that $q_k(\HEu)\subseteq\Dc(T)$ and $\tau(q_k)>1-\frac1k$.
Then $\tau(p\wedge q_k)>\tau(p)-\frac1k$.
This implies that $\bigcup_{k=1}^\infty (p\wedge q_k)(\HEu)$ is a subspace of $p(\HEu)\cap\Dc(T)$ and is dense in $p(\HEu)$.
\end{proof}

We now look at some subspaces that are similar to those described in Section~3 of~\cite{HS2}.

\begin{defi}
Let $\Mcal\subseteq B(\HEu)$ be a finite von Neumann algebra and let $T$ be an unbounded operator that is affiliated to $\Mcal$ and
let $r>0$.
\begin{enumerate}[{\rm (i)}]
\item Let $E(T,r)$ be the set of all $\xi\in\HEu$ such that there exists a sequence $(\xi_n)_{n=1}^\infty$ in $\HEu$ satisfying
\begin{equation}\label{eq:xin}
\xi_n\in\Dc(T^n),\qquad\lim_{n\to\infty}\|\xi_n-\xi\|=0,\qquad\limsup_{n\to\infty}\|T^n\xi_n\|^{1/n}\le r.
\end{equation}
\item Let $F(T,r)$ be the set of all $\eta\in\HEu$ such that there exists a sequence $(\eta_n)_{n=1}^\infty$ in $\HEu$ satisfying
$$
\eta_n\in\Dc(T^n),\qquad\lim_{n\to\infty}\|T^n\eta_n-\eta\|=0,\qquad\limsup_{n\to\infty}\|\eta_n\|^{1/n}\le \frac1r.
$$
\end{enumerate}
\end{defi}

\begin{prop}
The subspaces $E(T,r)$ and $F(T,r)$ are $T$-hyperinvariant.
\end{prop}
\begin{proof}
It is not difficult to see that the subspaces are closed.
Suppose $S$ is bounded and permutes with $T$.
Let $\xi\in E(T,r)$ and let $\xi_n\to\xi$ be as in~\eqref{eq:xin}.
Then, by Lemma~\ref{lem:permutes}, $S\xi_n\in\Dc(T^n)$.
We see
$\lim_{n\to\infty}\|S\xi_n-S\xi\|=0$ and
\begin{multline*}
\limsup_{n\to\infty}\|T^nS\xi_n\|^{1/n}=\limsup_{n\to\infty}\|ST^n\xi_n\|^{1/n}\le\limsup_{n\to\infty}\|S\|^{1/n}\,\|T^n\xi_n\|^{1/n} \\
=\limsup_{n\to\infty}\|T^n\xi_n\|^{1/n}\le r.
\end{multline*}
Thus $S\xi\in E(T,r)$.
This shows that $E(T,r)$ is $T$-hyperinvariant.
The case of $F(T,r)$ follows similarly easily.
\end{proof}

\begin{ques}
Are the subspaces $E(T,r)$ and $F(T,r)$ always $T$-invariant?
What if $W^*(T)$ is a finite von Neumann algebra?
\end{ques}

A positive answer to the above question could be useful in attempts to prove better properties of the Haagerup--Schultz projections
for elements $\Lc_{\log}(\Mcal,\tau)$, when $(\Mcal,\tau)$ is a finite von Neumann algebra.

\section{Lack of conditional expectations}
\label{sec:NoCondExp}

The purpose of this appendix is to show that the extraordinary length gone through
in Section~\ref{sec:FauxExp} to construct something like the conditional expectation of a very special unbounded operator
onto the relative commutant of an abelian algebra,
is in some sense justified.
We do this by showing that there cannot in general be a conditional expectation from $\Lc_{\log}(\Mcal,\tau)$
onto $\Lc_{\log}(\Mcal,\tau)\cap\Dc'$ for a particular case of an abelian algebra $\Dc\subseteq\Mcal$.
(In fact, we show more than this).

Suppose $\mathcal{M}$ is the hyperfinite II$_1-$factor, which is the appropriate closure
of the tensor product $\bigotimes_1^\infty M_2(\Cpx)$.
Let $\tau$ denote the unique tracial state on $\Mcal$.
Let $\mathcal{D}$ be the closure of the tensor product of the diagonal subalgebras.
This is maximal abelian self-adjoint algebra (see~\cite{SS08}) which implies $\Mcal\cap\Dc'=\Dc$.
There is a unique $\tau$-preserving projection of norm $1$, $E:\Mcal\to\Dc$,
which, then, automatically satisfies the conditional expectation property:
$$
E(D_1AD_2)=D_1E(A)D_2,\qquad(A\in\Mcal,\,D_1,D_2\in\Dc).
$$

\begin{prop}
Suppose $\Bc\subseteq\Sc(\Mcal,\tau)$ is an $\Mcal$-bimodule that is not contained in $\Lc_1(\Mcal,\tau)$.
Then there is no linear map $\Et:\Bc\to\Sc(\Mcal,\tau)$ that is positive and whose restriction to $\Mcal$ is $E$.
\end{prop}
\begin{proof}
Let $T\in\Bc$ be such that $T\ge0$ and $T\notin\Lc_1(\Mcal,\tau)$.
Replacing $T$ with some $T'$ satisfying $0\le T'\le T$, if necessary, we may without loss of generality assume
$T=\sum_{n=1}^\infty\alpha_nq_n$ (with convergence in measure in $\Sc(\Mcal,\tau)$)
for some $0\le\alpha_1\le\alpha_2\le\cdots$ and some pairwise orthogonal projections $(q_n)_{n=1}^\infty$ satisfying $\tau(q_n)=2^{-n}$.
In particular, we have
\begin{equation}\label{eq:suminfty}
\sum_{n=1}^\infty\alpha_n2^{-n}=\infty.
\end{equation}

Let $p=\frac12\left(\begin{smallmatrix}1&1\\1&1\end{smallmatrix}\right)\in M_2(\Cpx)$ and for $n\in\Nats$, let
$$
p_n=p^{\otimes(n-1)}\otimes(1-p)\otimes I_2\otimes I_2\otimes\cdots\in\bigotimes_1^\infty M_2(\Cpx).
$$
Then the projections in the sequence $(p_n)_{n=1}^\infty$ are pairwise orthogonal and $\tau(p_n)=2^{-n}$.
By conjugating $T$ with a unitary in $\Mcal$, we may without loss of generality assume $q_n=p_n$.

If a positive map $\Et:\Bc\to\Sc(\Mcal,\tau)$ extending $E$ were to exist, then
for each $N\in\Nats$, we would need
$$
\Et(T)\ge\Et(\sum_{n=1}^N\alpha_np_n)=\sum_{n=1}^N\alpha_nE(p_n)=\left(\sum_{n=1}^N\alpha_n2^{-n}\right)I,
$$
where $I$ denotes the identity element of $\Mcal$.
Letting $N\to\infty$ and using~\eqref{eq:suminfty}, this contradicts $\Et(T)\in\Sc(\Mcal,\tau)$.
\end{proof}

\end{document}